\documentclass[]{article}

\usepackage{graphicx}
\usepackage{amsmath}
\usepackage{amssymb}
\usepackage{amsthm} 
\usepackage{tikz}
\usetikzlibrary{arrows.meta, positioning}
\usetikzlibrary{decorations.pathmorphing}

\usepackage{pxfonts}
\usepackage{enumerate}
\usepackage{color}
\usepackage{mathdots}
\usepackage{sectsty}
\usepackage[hidelinks]{hyperref}
\usepackage{tikz-cd}
\usepackage{framed}
\usepackage{caption}
\usepackage{adjustbox}
\usepackage{bbold}
\allowdisplaybreaks
\usepackage[nottoc]{tocbibind}

\expandafter\def\expandafter\normalsize\expandafter{%
  \normalsize  
  \setlength\abovedisplayskip{2ex}
  \setlength\belowdisplayskip{2ex}
  \setlength\abovedisplayshortskip{1ex}
  \setlength\belowdisplayshortskip{1ex}
}
\usepackage{titlesec}
\titlespacing*{\section}
{0pt}{2.5ex plus 1ex minus .2ex}{3.3ex plus .2ex}
\titlespacing*{\subsection}
{0pt}{2.5ex plus 1ex minus .2ex}{3.3ex plus .2ex}

\sectionfont{\scshape\centering\fontsize{11}{14}\selectfont}
\subsectionfont{\scshape\fontsize{11}{14}\selectfont}
\usepackage{fancyhdr}

\newcommand\shorttitle{Exchangeability and integrable systems for random characteristic polynomials}
\newcommand\authors{T. Assiotis, M. A. Gunes, J. P. Keating and F. Wei}

\newenvironment{keywords}{}{}
\newtheorem{thm}{Theorem}[section]
\newtheorem{cor}[thm]{Corollary}
\newtheorem{lem}[thm]{Lemma}
\newtheorem{defn}[thm]{Definition}
\newtheorem{rmk}[thm]{Remark}
\newtheorem{prop}[thm]{Proposition}
\newtheorem{conjecture}[thm]{Conjecture}

\newtheorem*{claim*}{Claim}

\newcommand\longvdots[1]{\raisebox{1em}{\rotatebox{-90}{\hbox to #1 {\dotfill}}}}

\newcommand{\E}{\mathbb{E}}

\newcommand{\Nkt}{\Tilde{\mathbb{N}}^{(k)}}

\newcommand{\ind}{\perp\!\!\!\!\!\!\perp} 
\newcommand{\symdif}{\,\Delta\,}
\newcommand{\U}{\mathbb{U}}
\newcommand{\X}{\mathsf{X}}
\newcommand{\haar}{\mathbb{U}(N)}
\newcommand{\defeq}{\stackrel{\mathrm{def}}{=} }
\newcommand{\di}{\mathrm{d}}

\newcommand{\Hfxn}{\mathbf{\Psi}}
\newcommand{\hfxn}{\boldsymbol{\psi}}

\newcommand{\be}{\begin{equation}}
\newcommand{\ee}{\end{equation}}
\newcommand{\bea}{\begin{eqnarray}}
\newcommand{\eea}{\end{eqnarray}}
\newcommand{\bes}{\begin{equation*}}
\newcommand{\ees}{\end{equation*}}
\newcommand{\beas}{\begin{eqnarray*}}
\newcommand{\eeas}{\end{eqnarray*}}

\def\i{{\textnormal{i}}}

\newcommand{\blambda}{\boldsymbol{\lambda}}

\begin{document}

\providecommand{\keywords}[1]
{ \textbf{Keywords:} #1 }



\title{ \large  \bf EXCHANGEABLE ARRAYS AND INTEGRABLE SYSTEMS FOR CHARACTERISTIC POLYNOMIALS OF RANDOM MATRICES}

\date{}

\author{\small THEODOROS ASSIOTIS, MUSTAFA ALPER GUNES, JONATHAN P. KEATING and FEI WEI}
\maketitle

\begin{abstract}
The joint moments of the derivatives of the characteristic polynomial of a random unitary matrix, and also a variant of the characteristic polynomial that is real on the unit circle, in the large matrix size limit, have been studied intensively in the past twenty five years, partly in relation to conjectural connections to the Riemann zeta-function and Hardy's function. We completely settle the most general version of the problem of convergence of these joint moments, after they are suitably rescaled, for an arbitrary number of derivatives and with arbitrary positive real exponents. Our approach relies on a hidden, higher-order exchangeable 
structure, that of an exchangeable array. Using these probabilistic techniques, we then give a combinatorial formula for the leading order coefficient in the asymptotics of the joint moments, when the power on the characteristic polynomial itself is a positive real number and the exponents of the derivatives are integers, in terms of a finite number of finite-dimensional integrals which are explicitly computable. Finally, we develop a method, based on a class of Hankel determinants shifted by partitions, that allows us to give an exact representation of all these joint moments, for finite matrix size, in terms of derivatives of Painlev\'e V transcendents, and then for the leading order coefficient in the large-matrix limit in terms of derivatives of solutions of the $\sigma$-Painlev\'e III' equation. Equivalently, we can represent all the joint moments of power sum linear statistics of a certain determinantal point process behind this problem in terms of derivatives of $\sigma$-Painlev\'e III' transcendents. This gives an efficient way to compute all these quantities explicitly. Our methods can be used to obtain analogous results for a number of other models sharing the same features. 
\end{abstract}

\tableofcontents

\section{Introduction and main results}

\subsection{Random Matrix Theory and Riemann $\zeta$-function}

    Over the past fifty years, there has been significant interest in the surprising duality between problems in analytic number theory and analogous ones in random matrix theory. This connection dates back to the work of Montgomery \cite{Montgomery}, who conjectured that the eigenvalues of $N \times N$ unitary matrices, as the matrix size tends to infinity, can be used to model the behaviour of the zeros of the Riemann $\zeta$-function, $\zeta(z)$, high up the critical line, $\Re (z)=1/2$; see \cite{Bog-KeaI,Bog-KeaII,Bog-KeaIII,RudnickSarnak,ConreySnaithCorr} for subsequent extensions of Montgomery's work.  It is natural to ask whether this connection between the zeros of the Riemann $\zeta$-function and those of the characteristic polynomials of random unitary matrices (i.e.~the eigenvalues) extends to the values of these two functions.  Keating and Snaith \cite{keatingsnaith} suggested that the value distribution of the Riemann $\zeta$-function high on the critical line can be modeled by the value of the characteristic polynomial of a random unitary matrix at 1, when the matrix size is large enough. As an application, they conjectured the leading-order asymptotics of the moments of the $\zeta$-function on the critical line by considering the analogous moments for the characteristic polynomials of unitary random matrices. 
    
    To be more precise, let $\mathbf{A} \in \mathbb{U}(N)$, where, here and throughout this paper, $\mathbb{U}(N)$ denotes the group of $N \times N$ unitary matrices, and let $e^{\textnormal{i}\theta_1}, \ldots, e^{\textnormal{i}\theta_N}$ denote the eigenvalues of $\mathbf{A}$. Consider the characteristic polynomial $V_\mathbf{A}(\theta)$ of $\mathbf{A}$ on the unit circle $\theta \in [0,2\pi)$, given by,
\[
V_\mathbf{A}(\theta) = \det\left(\mathbf{I} - e^{-\textnormal{i}\theta} \mathbf{A}\right) = \prod_{j=1}^{N} \left(1 - e^{\textnormal{i}(\theta_j-\theta)}\right).
\]
It was conjectured in \cite{keatingsnaith} (see also \cite{ConreyConj1,ConreyConj2} for $s=3,4$), that the $T \rightarrow \infty$ asymptotics of
\[
\frac{1}{T} \int_0^T \left| \zeta \left(\frac{1}{2} + \textnormal{i}t\right) \right|^{2s} \mathrm{d}t
\]
can be obtained from the $N \rightarrow \infty$ asymptotics of
\[
\int_{\mathbb{U}(N)} \left|V_\mathbf{A}(0)\right|^{2s} \mathrm{d}\mu_{\mathrm{Haar}}(\mathbf{A}),
\]
where $\mu_{\mathrm{Haar}}$ is the Haar measure on $\mathbb{U}(N)$. This conjecture was extended to later terms in the asymptotic expansion of the moments in \cite{CFKRS1,CFKRS2}, and has been verified heuristically using correlations of divisor functions in \cite{CK1,CK2,CK3,CK4,CK5}. 

Inspired by this, Hughes \cite{Hughes} considered a related quantity, aiming to obtain conjectural values for joint moments of the Riemann $\zeta$-function and its derivative. For a unitary matrix $\mathbf{A}$, instead of the characteristic polynomial itself, he considered the associated function
\[
Z_\mathbf{A}(\theta) = \exp\left(\textnormal{i} N\frac{\theta+\pi}{2}-\textnormal{i}\sum_{j=1}^N \frac{\theta_j}{2}\right) V_{\mathbf{A}}(\theta)
\]
which has the property that for $\theta \in [0, 2\pi)$, $Z_{\mathbf{A}}(\theta) \in \mathbb{R}$ and $|Z_{\mathbf{A}}(\theta)| = |V_{\mathbf{A}}(\theta)|$. This relation between $V_{\mathbf{A}}$ and $Z_{\mathbf{A}}$ is analogous to the relation between the $\zeta$ function and Hardy’s $\mathcal{Z}$-function, defined as follows,
\[
\mathcal{Z}(t) =  \pi^{-\textnormal{i}t/2} \frac{\Gamma(1/4 + \textnormal{i}t/2)}{|\Gamma(1/4+\textnormal{i}t/2)|}\zeta(1/2+\textnormal{i}t),
\]
which also satisfies $|\mathcal{Z}(t)| = \left|\zeta \left(\frac{1}{2} + \textnormal{i}t\right)\right|$, and $\mathcal{Z}(t) \in \mathbb{R}$ for all $t \in \mathbb{R}$. Hughes conjectured in \cite{Hughes} that, for arbitrary positive real parameters $s, h$,  one should have asymptotic convergence of the joint moments as follows,
\begin{equation}
    \lim_{N \to \infty} \frac{1}{N^{s^2+2h}} \int_{\mathbb{U}(N)} \left| Z_\mathbf{A}(0) \right|^{2s-2h} \left|\frac{\mathrm{d}Z_{\mathbf{A}}}{\mathrm{d}\theta} \bigg|_{\theta=0}\right|^{2h} \mathrm{d}\mu_{\text{Haar}}(\mathbf{A}) = \mathcal{P}(s,h),\label{HughesConj}
\end{equation}
for some positive $\mathcal{P}(s,h)$, a claim he was able to prove for $s, h \in \mathbb{N}$, deriving an expression for $\mathcal{P}(s,h)$ for these values of $s,h$. This, in turn, leads to the conjectural asymptotics for the analogous joint moments of $\mathcal{Z}(t)$, introduced by Ingham in 1926 \cite{Ingham},
\begin{equation*}
  \frac{1}{T} \int_0^T \left|\mathcal{Z}(t)\right|^{2s-2h} \left|\mathcal{Z}'(t)\right|^{2h} \mathrm{d}t \sim (\log T)^{s+2h}\eta(s) \mathcal{P}(s,h)  
\end{equation*}
where
\[
\eta(s) = \prod_{\text{primes } p} (1 - p^{-1})^{s^2} \sum_{k=0}^{\infty} p^{-k} \left[ \frac{\Gamma(k+s)}{\Gamma(k + 1)\Gamma(s)} \right]^2.
\]
We note that even though his initial work, and the majority of subsequent works, are centered around the joint moments of $Z_{\mathbf{A}}$, Hughes also conjectured analogous results for  $V_{\mathbf{A}}$ and $\zeta\left(
\frac{1}{2}+\textnormal{i}t\right)$. Our goal in this paper is to completely solve the most general version of this problem on the random matrix side, by considering moments involving an arbitrary number of derivatives and positive real exponents, for both $V_\mathbf{A}$ and $Z_{\mathbf{A}}$. First, we briefly review the literature.

\subsection{A brief history of the problem}\label{HistorySection}
Over the years, there has been significant interest and progress in both proving the convergence in \eqref{HughesConj}, and also obtaining explicit formulae for the leading order coefficient $\mathcal{P}(s,h)$. After Hughes' proof for the case $s, h \in \mathbb{N}$, Conrey, Rubenstein and Snaith \cite{conreyetal} gave an alternative proof for the case $s = h \in \mathbb{N}$, using representations of these moments via multiple contour integrals from \cite{CFKRS1,CFKRS2}. They also gave an alternative expression for the coefficient $\mathcal{P}(s,s)$ in terms of a determinant of Bessel functions, which was later shown to be related to the Painlevé equations by Forrester and Witte in \cite{forrester2002application}. Dehaye \cite{Dehaye2008, Dehaye2010note} gave an independent proof for $s,h \in \mathbb{N}$, and gave another representation of $\mathcal{P}(s,h)$ as a certain combinatorial sum. Winn \cite{winn2012derivative} gave the first proof for $s \in \mathbb{N}, h \in \frac{1}{2}\mathbb{N}$ by expressing the moments in terms of Laguerre polynomials. Basor et al.  in \cite{Basor_2019} took a different approach, using Riemann-Hilbert problem techniques to obtain a further alternative proof, and giving a representation for $\mathcal{P}(s,h)$ in terms of Painlevé transcendents for $s,h \in \mathbb{N}$. Bailey et al \cite{Bailey_2019} extended the methods developed in \cite{conreyetal} to obtain the same Painlevé-representation of $\mathcal{P}(s,h)$. In \cite{altugetal}, the authors extended the method of \cite{conreyetal} to the orthogonal and symplectic groups, including to moments of higher derivatives, and used the results to formulate conjectures for the corresponding moments of families of number-theoretic $L$-functions, along the lines of \cite{KeatingSnaithLfunctions}. 

It is important to point out that Basor et al.~\cite{Basor_2019} showed that when the matrix size is fixed, the joint moments also have an exact representation in terms of Painlevé transcendents (not the same ones that describe the large-matrix limit).     

The large-matrix asymptotics for general positive real parameters $s,h$ were only recently proven for the first time in \cite{assiotis2022joint}. Moreover, it was understood in the same paper that there exists a random variable, defined via a determinantal point process, which serves as a bridge between the problem of joint moments and Painlevé equations. Even though this link was initially only proven for $s \in \mathbb{N}$ in \cite{assiotis2022joint}, it was later shown in \cite{ABGS} that the connection indeed holds for real $s$. Finally, we note that for the case where $Z_\mathbf{A}(\theta)$ is replaced by the characteristic polynomial $V_\mathbf{A}(\theta)$, which models the $\zeta$-function itself, Hughes' conjecture remained unproven until now outside the parameter range $s,h \in \mathbb{N}$.

\subsection{Joint moments of higher order derivatives}\label{convergencesubsec}
The main purpose of the present paper is to establish the $N\to \infty$ asymptotics, and study in detail the leading order coefficient, of the joint moments of arbitrary numbers of higher order
derivatives with arbitrary positive real exponents.  We also obtain exact formulae when $N$ is fixed.
\begin{defn}\label{jointmomdef}
Let $n_1>\cdots> n_k \in \mathbb{N}\cup \{0\}$, and $h_1,\dots, h_k \in \mathbb{R}_+$. Then, we define:
\begin{align}
\mathfrak{G}^{(n_1,\ldots,n_k)}_{N}(h_1,\ldots,h_k)&\stackrel{\mathrm{def}}{=} \int_{\mathbb{U}(N)}\prod_{j=1}^{k}\left|V_{\mathbf{A}}^{(n_{j})}(0)\right|^{2h_{j}}\mathrm{d}\mu_{\mathrm{Haar}}(\mathbf{A}),\label{analogue of Hardyz}\\
\mathfrak{F}^{(n_1,\ldots,n_k)}_{N}(h_1,\ldots,h_k)&\stackrel{\mathrm{def}}{=} \int_{\mathbb{U}(N)}\prod_{j=1}^{k}\left|Z_{\mathbf{A}}^{(n_{j})}(0)\right|^{2h_{j}}\mathrm{d}\mu_{\mathrm{Haar}}(\mathbf{A}).\label{analogue of zeta}
\end{align}
Here, the superscript $n_{j}$ in  $V^{(n_j)}_{\mathbf{A}}$ or $Z^{(n_j)}_{\mathbf{A}}$ means the $n_{j}$-th derivative of $V_{\mathbf{A}}$ or $Z_{\mathbf{A}}$.    
\end{defn}

Observe, that (\ref{HughesConj}) is simply a statement for the asymptotics of $\mathfrak{F}_N^{(1,0)}(h,s-h)$. Barhoumi in his monograph \cite{barhoumi2020new} was the first to establish the asymptotics for $\mathfrak{G}_N^{(n_1,\dots,n_k)}(h_1,\dots,h_k)$ for $h_1,\dots,h_k \in \mathbb{N}$ using an approach based on symmetric function theory. More recently, the asymptotics for $k=2$, for $h_1,h_2 \in \mathbb{N}$, for both $\mathfrak{G}_N^{(n_1,n_1)}(h_1,h_2)$ and $\mathfrak{F}_N^{(n_1,n_2)}(h_1,h_2)$, were obtained in \cite{keatingwei} using an intricate analysis of multiple contour integrals and an expression, different to the one in \cite{barhoumi2020new}, for the leading order coefficient that was shown in \cite{keating-fei}, in the special case of $\mathfrak{F}_N^{(2,0)}(h_1,h_2)$, to be connected to Painlevé equations. We note that the expressions for the leading order coefficients given in \cite{barhoumi2020new,keating-fei,keatingwei} do not make sense for generic real exponents $h_i$.
 
Our first main result gives the convergence of both of these quantities, after appropriate rescalings, for the
full parameter range (in particular without the integrality restrictions), and gives a probabilistic expression for the leading coefficient in terms of joint moments
of certain natural random variables that are defined through a determinantal point process. Thus, in order to
state our result, we need some definitions.

A determinantal point process on a Borel set $\mathfrak{X} \subset \mathbb{R}$ with correlation kernel $\mathcal{K}: \mathfrak{X} \times \mathfrak{X} \to \mathbb{C}$ is a probability measure $\mathfrak{P}$ on $\mathrm{Conf}(\mathfrak{X})$, the space of locally finite collections of points in $\mathfrak{X}$, endowed with a certain topology and corresponding Borel $\sigma$-algebra, see \cite{BorodinDet,JohanssonDet} for the details, such that for any $k\ge 1$ and any measurable and bounded $F:\mathfrak{X}^k \to \mathbb{R}$ with compact support, we have
    \begin{equation*}
        \int_{\mathrm{Conf}(\mathfrak{X})} \sum_{a_{i_1},\ldots, a_{i_k} \in \mathsf{E}} F(a_{i_1},\ldots, a_{i_k}) \mathrm{d}\mathfrak{P}(\mathsf{E})=\int_{\mathfrak{X}^k} F(x_1,\ldots, x_k) \det_{1\leq i,j\leq k} \left[\mathcal{K}(x_i,x_j)\right] \di x_1\cdots \di x_k.
    \end{equation*}

We will be interested in specific determinantal point processes defined via their correlation kernels $\mathfrak{E}^{(s)}$ (we note here that by results of \cite{Borodin_2001,len73}, the kernels $\mathfrak{E}^{(s)}$, completely determine their corresponding point processes), where $s\geq 0$ is a parameter.
\begin{defn}
    Let $s\in \mathbb{R}_{\geq 0}$. Then, we define $\mathbf{P}^{(s)}$ to be the determinantal point process on $\left(-\infty,0\right)\times \left(0,\infty\right)$ with the correlation kernel $\mathfrak{E}^{(s)}$ given by,
        \begin{align*}
\mathfrak{E}^{(s)}(x,y) = \frac{1}{2\pi}\frac{\left(\Gamma(s+1)\right)^2}{\Gamma(2s+1)\Gamma(2s+2)}\frac{G^{(s)}(x)H^{(s)}(y)-G^{(s)}(y)H^{(s)}(x)}{x-y},
\end{align*}
where the functions $G^{(s)}(x), H^{(s)}(x)$ are given by the formulae 
\begin{equation*}
G^{(s)}(x) = 2^{2s-\frac{1}{2}}\Gamma\left(s+\frac{1}{2}\right) \cdot \frac{1}{|x|^{\frac{1}{2}}}J_{s-1/2}\left(\frac{1}{|x|}\right), \ \ H^{(s)}(x) =\mathrm{sgn}(x)2^{2s+\frac{1}{2}}\Gamma\left(s+\frac{3}{2}\right) \cdot \frac{1}{|x|^{\frac{1}{2}}} J_{s+1/2}\left(\frac{1}{|x|}\right),
\end{equation*}
where $J_{\nu}$ denotes the Bessel function with parameter $\nu$.

\end{defn}

This random point process first appeared, in  a completely different setting, in influential work of Borodin and Olshanski \cite{Borodin_2001}, and we will say more in the sequel. Moving on, certain elementary symmetric functions, the power sums and elementary symmetric functions, of the random points of $\mathbf{P}^{(s)}$ will be one of the main ingredients in our first main result.

\begin{defn}\label{DefinitionRandomVariables}
Let $s\ge 0$. Associated to the determinantal point process $\mathbf{P}^{(s)}$, we define random variables $\{\mathsf{q}_n(s)\}_{n\geq 1}$ by, where the random point configuration $\mathsf{E}$ is distributed according to $\mathbf{P}^{(s)}$,
\begin{align}
    \mathsf{q}_1(s)&=\lim_{k\to \infty} \sum_{x\in \mathsf{E}} x \mathbf{1}_{|x|> k^{-2}}\label{Cutoffq_1},\\
    \mathsf{q}_n(s)&= \sum_{x\in \mathsf{E}} x^n, \ \ \textnormal{for } n\ge 2.
\end{align}
Moreover, define random variables $\{\mathsf{Y}_n(s)\}_{n\geq 0}$ by letting $\mathsf{Y}_0(s)\equiv 1$, $\mathsf{Y}_1(s)=-\mathsf{q}_1(s)$ and inductively, for $n\geq 2$,
\begin{equation}\label{PowerSumsToElemSymFunctions}
    \mathsf{Y}_n(s)=-(n-1)! \sum_{j=1}^{n}\frac{1}{(n-j)!}\mathsf{Y}_{n-j}(s)\mathsf{q}_j(s).
\end{equation}
\end{defn}

It is not immediately obvious that these infinite sums are $\mathbf{P}^{(s)}$-a.s. finite but this follows from the results of \cite{Borodin_2001,Qiu}. We note that the principal value sum in the definition of $\mathsf{q}_1(s)$ is required, otherwise the series does not converge. The random variables $\{\mathsf{Y}_n(s)\}_{n\ge 1}$ are really just the elementary symmetric functions of $\mathbf{P}^{(s)}$ (up to explicit universal constants); the formula (\ref{PowerSumsToElemSymFunctions}) is simply Newton's identity, in that it expresses the elementary symmetric functions in terms of power sums. We could have defined them directly, but then one needs to introduce a cutoff in the definition as in (\ref{Cutoffq_1}) for every $n\ge 1$. Both families of random variables will appear in the results below, as in some cases one choice is better suited than the other. Below we normally denote by the symbol $\mathbb{E}$ the expectation with respect to the underlying probability law, which will be clear from context.

With these definitions at hand, our first main result gives the leading order asymptotics of both quantities introduced in Definition \ref{jointmomdef} for the full parameter range.
\begin{thm}\label{convergence theorem}
Let $n_1>n_2>\cdots >n_k \in \mathbb{N}\cup \{0\}$. Let $h_1,\ldots h_k \in \mathbb{R}_{+}$ and $s=\sum_{j=1}^k h_j> 0$. Then, we have the following finite limits,
    \begin{equation}
        \lim_{N\to \infty} \frac{\mathfrak{F}_{N}^{(n_1,\ldots,n_k)}(h_1,\ldots,h_k)}{N^{s^2+2\sum_{j=1}^k n_j h_j} }  = \frac{G(s+1)^2}{G(2s+1)} 2^{-2\sum_{j=1}^k h_j n_j  } \mathbb{E}\left[\prod_{j=1}^k\left|\mathsf{Y}_{n_j}(s)\right|^{2h_j}\right],
        \label{mainthmeq}
    \end{equation}
    \begin{equation}
       \lim_{N\to \infty} \frac{\mathfrak{G}_{N}^{(n_1,\ldots,n_k)}(h_1,\ldots,h_k)}{N^{s^2+2\sum_{j=1}^k n_j h_j} }  = \frac{G(s+1)^2}{G(2s+1)} 2^{-2\sum_{j=1}^k h_j n_j  }\mathbb{E}\left[\prod_{j=1}^k \Bigg|\sum_{m_{j}=0}^{n_{j}}(-\textnormal{i})^{m_j} \binom{n_{j}}{m_{j}} \mathsf{Y}_{n_j-m_j}(s)\Bigg|^{2h_j}\right],\label{mainthmeq2}
    \end{equation}
    where $G(z)$ is the Barnes G-function, with $\gamma$ the Euler-Mascheroni constant,
\begin{equation*}
    G(1+z) = (2\pi)^\frac{z}{2}\exp\left(-\frac{z+z^2(1+\gamma)}{2}\right)\prod_{j=1}^{\infty}\left(1+\frac{z}{j}\right)^j\exp\left(\frac{z^2}{2j}-z\right).
\end{equation*}
    \label{mainthm}
\end{thm}

We note that Theorem \ref{convergence theorem}, in the more natural case of the characteristic polynomial $V_{\mathbf{A}}$, is new whenever the exponent on any of the derivatives of $V_{\mathbf{A}}$ is non-integer. Even in the simplest possible case of $k=2$, $n_1=1,n_2=0$, this confirms the prediction of Hughes \cite{Hughes} from more than 20 years ago; the somewhat simpler case of the asymptotics of $\mathfrak{F}_N^{(1,0)}(h_1,h_2)$, also conjectured by Hughes, was in fact only proven in the last few years as well, see \cite{assiotis2022joint}.

The probabilistic expressions in (\ref{mainthmeq}) and (\ref{mainthmeq2}) in terms of the random variables $\{\mathsf{Y}_n(s)\}_{n\ge 1}$ are aesthetically rather pleasing and, although initially may seem to come out of nowhere, it will be clear from the proof that they are rather natural. On the other hand, using them to compute quantities more explicitly is far from trivial. In proving the results that follow we will not make direct use of these expressions. Nevertheless, we stress that (\ref{mainthmeq}) and (\ref{mainthmeq2}) are the only expressions for the leading order coefficients available in the literature which make sense for generic non-integer exponents $h_i$, even in the simplest case of the moments of a single derivative.

Towards our goal of computing the leading order coefficient more explicitly, our next main result gives a representation of the expectation on the right-hand side of \eqref{mainthmeq}, for $h_1,\ldots, h_k \in \mathbb{N}$, as a finite linear combination of explicit finite-dimensional integrals. These integrals can all be computed explicitly (one simply expands the Vandermonde determinant and these integrals boil down to one-dimensional integrals which are basically moments of a Student's t-distribution) but writing down a general expression becomes very tedious and we will not attempt to do it here; we will instead discuss an alternative recursive way to obtain such expressions via connections to integrable systems below. In particular, the formula in Theorem \ref{mainresult2} gives a concrete expression for the joint moments whenever the exponent on the characteristic polynomial itself is a real number and the exponents on its derivatives are integers. 

These finite-dimensional averages will be taken with respect to the Cauchy \cite{ForresterBook,ForresterWitteCauchy} (also called Hua-Pickrell \cite{Borodin_2001}) measures ${\bf M}_{N}^{(s)}$, with $s\ge 0$, which are probability measures supported on the Weyl chamber,
\begin{equation*}
    \mathbb{W}_N \defeq \left\{\mathbf{x}=(x_1,x_2,\ldots, x_N)\in \mathbb{R}^N  : x_1\geq x_2\geq \cdots\geq x_N\right\}
\end{equation*}
that are given by:
\begin{equation}\label{hpintrodef}
    {\bf M}_{N}^{(s)}(\di x_{1},\ldots,\di x_{N})\defeq\frac{1}{\Tilde{\mathrm{C}}_N^{(s)}} \prod_{i=1}^N \frac{1}{\left(1+x_i^2\right)^{s+N}}  \Delta^2(x_1,\ldots, x_N) \di x_{1}\cdots \di x_{N},
\end{equation}
where we denote by $\Delta$ the Vandermonde determinant,
\begin{equation*}
     \Delta(x_1,\ldots, x_N)\defeq \prod_{1\leq i < k \leq N} (x_i-x_k),
\end{equation*}
and the normalization constant $\Tilde{\mathrm{C}}_N^{(s)}$ is given explicitly by
\begin{equation*}
 \Tilde{\mathrm{C}}_N^{(s)} = {\pi^N2^{-N(N+2s -1)}} \prod_{j=0}^{N-1} \frac{j!\Gamma(2s + N-j)}{\Gamma(s+N-j)^2}.
\end{equation*}
For the rest of the paper, we denote averages taken with respect to these measures as:
\begin{equation*}
\mathbb{E}_N^{(s)}\left[f(\mathsf{x}_1^{(N)},\ldots,\mathsf{x}_N^{(N)})\right] \defeq \int_{\mathbb{W}_N} f(x_1,\ldots, x_N)   {\bf M}_{N}^{(s)}(\di x_{1},\ldots,\di x_{N}).
\end{equation*}
Note that, whenever the function $f$ is invariant under permutations of indices, and this is the case for all functions we will consider in this paper, one can equivalently write:
\begin{equation}\label{HPintegralDef}
\mathbb{E}_N^{(s)}\left[f(\mathsf{x}_1^{(N)},\ldots,\mathsf{x}_N^{(N)})\right] = \frac{1}{N!} \int_{\mathbb{R}^N} f(x_1,\ldots, x_N)   {\bf M}_{N}^{(s)}(\di x_{1},\ldots,\di x_{N}).
\end{equation}
We also observe that by analytic continuation the function $s\mapsto \mathbb{E}_N^{(s)}\left[\bullet\right]$ is holomorphic whenever the integrand is holomorphic in $s$ and the integral on the right hand side of \eqref{HPintegralDef} exists. We note that although for $s\in \mathbb{C}\backslash \mathbb{R}$ the probabilistic interpretation breaks down, abusing notation we will still use the expectation sign $\mathbb{E}_N^{(s)}$ below. As the discussion preceding Theorem \ref{mainthm} suggests, throughout the rest of the paper, we will commonly use elementary symmetric polynomials and power sum polynomials, which we will denote by ($\mathrm{e}_0=\mathrm{p}_0 \equiv 1$)
\begin{align*}
    \mathrm{e}_n\left(x_1,x_2,\ldots, x_m\right)&\overset{\textnormal{def}}{=}\sum_{1\leq i_1<i_2<\cdots<i_n\leq m} x_{i_1} x_{i_2}\cdots x_{i_n},\\
    \mathrm{p}_n\left(x_1,x_2,\ldots, x_m\right)&\defeq \sum_{i=1}^m x_i^n.
\end{align*}
We are now in a position to state our second main result.
\begin{thm}\label{mainresult2}
    Let $n_1> \cdots>n_k \in \mathbb{N}\cup\{0\}$, and $h_j\in \frac{\mathbb{N}}{2}$, for $j=1,\ldots,k$. Moreover, define $L= \sum_j 2h_j n_j$. Then, for any real number $s>\sum_j h_j -\frac{1}{2}$, we have 
    \begin{equation*}
    \mathbb{E}\left[\prod_{j=1}^k \left(\mathsf{Y}_{n_j}(s)\right)^{2h_j}\right]=\frac{(n_1 !)^{2h_1}\cdots(n_k !)^{2h_k}}{L!}\sum_{m=n_1}^{L} (-1)^{m} {L \choose m} \mathbb{E}_{m}^{(s)}\left[\prod_{j=1}^k \left(\mathrm{e}_{n_j} \left(\mathsf{x}_1^{(m)},\mathsf{x}_2^{(m)},\ldots, \mathsf{x}_m^{(m)}\right)\right)^{2h_j}\right].
    \end{equation*}

 In particular, we obtain that the map
    \begin{equation*}
        s \mapsto \mathbb{E}\left[\prod_{j=1}^k\left(\mathsf{Y}_{n_j}(s)\right)^{2h_j}\right]
    \end{equation*}
    is a rational function.
\end{thm}

The above result also has the following application. A priori, it is not clear that the limiting random variables $\mathsf{Y}_n(s)$ are almost surely non-zero. Although it may seem highly unlikely that they are almost surely zero, by some extreme coincidence it could happen. If this were the case, the asymptotics we established earlier would give a bound on the growth of the joint moments, but not the true order. In fact, in the literature on this problem, the proof that the leading coefficient in such asymptotics is non-zero is omitted\footnote{Even when this coefficient is given as an explicit sum it is non-trivial to see that it is strictly positive as the sum involves both positive and negative terms.}, with the only exception, as far as we are aware, being \cite{assiotis2022joint}. In may be possible to use the determinantal point process representation of the random variables $\mathsf{Y}_n(s)$ to conclude that they are not a.s. zero but this is most likely challenging. Instead, we simply obtain an explicit expression for the second moment of $\mathsf{Y}_n(s)$ for any $n\in \mathbb{N}$, and show that it is non-zero and similarly for the expression appearing in the asymptotics in (\ref{mainthmeq2}). Despite their simplicity, the expressions below are new for any choice of the parameter $s$. It is interesting to note that the expression for $V_\mathbf{A}$ and $\zeta$ is much simpler than the one for $Z_\mathbf{A}$ and $\mathcal{Z}$, even though its probabilistic expression is more complicated.

\begin{cor}\label{2ndmomentcor}
     Let $n\in \mathbb{N}$, and $s\in \big(\frac{1}{2},\infty\big)$. Then we have,
\begin{align}\label{2ndmomentexplicit}
     \E\left[\left|\mathsf{Y}_{n}(s)\right|^2\right] =2^{2n}\frac{2s-1}{\prod_{l=1}^{n}(2s-2+l)^2}
\sum_{i,j=0}^{n}\binom{n}{i}\binom{n}{j}&\frac{(-2)^{-2n+i+j}}{2s-1+i+j} \frac{\Gamma(s+i)\Gamma(s+j) \Gamma(2s+n-1)^2}{\Gamma(2s+i-1)\Gamma(2s+j-1)\Gamma(s)^2}, \nonumber \\
\mathbb{E}\left[\Bigg|\sum_{m=0}^{n}(-\textnormal{i})^{m} \binom{n}{m} \mathsf{Y}_{n-m}(s)\Bigg|^2\right]=2^{2n} &\frac{2s-1}{2s-1+2n}\prod_{l=1}^{n}\left(\frac{l+s-1}{l+2s-2}\right)^2.\\
\intertext{Moreover, we have that,}
\mathbb{P}\left(\left|\sum_{m=0}^{n}(-\textnormal{i})^{m} \binom{n}{m} \mathsf{Y}_{n-m}(s)\right|>0\right)>0 \;\; & \textnormal{ and } \;\; \mathbb{P}\left(\left|\mathsf{Y}_n(s)\right|>0\right)>0.\nonumber
\end{align}
\end{cor}

 Although up until now we have restricted our attention to a rather special random matrix model, the Haar-distributed unitary matrices, as an immediate by-product of the methods we use to prove Theorems \ref{convergence theorem} and \ref{mainresult2} we obtain the following general result on characteristic polynomials of infinite permutation-invariant random matrices\footnote{We refer to Section \ref{consistentsection} for a precise explanation of what is meant by an infinite random matrix}. In some sense, and this will be clear in the sequel, parts of Theorem \ref{convergence theorem} and Theorem \ref{mainresult2} are special cases of the theorem below for a specific choice of infinite random matrix $\mathbf{H}$.

\begin{thm}\label{GeneralCharPolyThm}
    Let $\mathbf{H}$ be an infinite Hermitian random matrix so that for any permutation $\tau$ of $\mathbb{N}$ that fixes all but finitely many elements, \begin{equation*}
        \mathsf{Law}(\mathbf{H})=\mathsf{Law}(\mathbf{P}_\tau \mathbf{H} \mathbf{P}_\tau^*),
    \end{equation*}
    where $\mathbf{P}_\tau$ is the permutation matrix that corresponds to $\tau$, and $\mathsf{Law}(\mathfrak{X})$ denotes the law of a random variable $\mathfrak{X}$. Consider, for each $N\in \mathbb{N}$, the $N \times N$ top-left corner of $\mathbf{H}$,  $\mathbf{H}_N=[\mathbf{H}_{ij}]_{i,j=1,\dots,N}$, and its rescaled (reverse) characteristic polynomial,
    \begin{equation*}
        \mathfrak{p}_N(z)=\det\left(\mathbf{I}-\frac{z}{N} \mathbf{H}_N\right)= 1+\sum_{k=1}^N \mathsf{a}_N^{(k)}(\mathbf{H}) z^k.
    \end{equation*}
    Then, whenever there exists $m\in \mathbb{N}$, $p\in [1,\infty)$ such that
    \begin{equation*}
\mathbb{E}\left[\Big|\det\left(\mathbf{H}_k\right)\Big|^p\right]<\infty 
    \end{equation*}
    for $k=1,\ldots,m$, then there exists random variables $\left\{\mathsf{a}_k(\mathbf{H})\right\}_{k=1}^m$ such that, for $k=1,\dots,m$,
    \begin{equation*}
      \mathsf{a}_{N}^{(k)}\left(\mathbf{H}\right) \xrightarrow[]{N\to \infty}    \mathsf{a}_{k}(\mathbf{H})
    \end{equation*}
    in $L^p$ and almost surely with respect to $\mathsf{Law}(\mathbf{H})$. Furthermore, if
  $r_1,\ldots, r_m \in \mathbb{R}_{\geq 0}$, are such that $\sum_{j}r_j=p$, then we have
    \begin{equation*}
      \mathbb{E}\left[\prod_{k=1}^m \Big|\mathsf{a}_{N}^{(k)}\left(\mathbf{H}\right)\Big|^{r_j}\right]  \xrightarrow[]{N\to \infty} \mathbb{E}\left[\prod_{k=1}^m \Big|\mathsf{a}_{k}(\mathbf{H})\Big|^{r_j}\right].
    \end{equation*}
     Moreover, if $r_1,\ldots, r_l \in \mathbb{N}$, $m \geq k_1>\cdots>k_l \geq 1$ and if we define $L\defeq \sum_j r_j k_j$, we then have
    \begin{equation}
    \mathbb{E}\left[\prod_{j=1}^l \left(\mathsf{a}_{k_j}(\mathbf{H})\right)^{r_j}\right]=\frac{(k_1 !)^{r_1}\cdots(k_l !)^{r_l}}{L!}\sum_{m=k_1}^{L} (-1)^{m+L} {L \choose m} \mathbb{E}\left[\prod_{j=1}^l \left(\mathsf{a}_{m}^{(k_j)}(\mathbf{H})\right)^{r_j}\right].
    \end{equation}
\end{thm}

We remark that it was recently proven in \cite{RandomAnalytic} that for a subclass of infinite permutation-invariant random matrices, the class of unitarily-invariant ones, which have been classified in \cite{olshanskivershik,Pickrell}, $\mathfrak{p}_N(z)$ itself converges almost surely with respect to $\mathsf{Law}(\mathbf{H})$, uniformly on compact sets in $\mathbb{C}$, to a random analytic function $\mathfrak{p}_\infty(z)$. The theorem above extends this convergence to convergence of the moments of Taylor coefficients and gives a formula for the joint moments of the limiting coefficients in terms of the pre-limit ones.

\subsection{Number theory conjectures}

Following the philosophy set out in \cite{keatingsnaith,KeatingSnaithLfunctions}, our results stated thus far give rise to the following conjecture, which generalises all previous conjectures on this problem. Here, we write out the special case of Corollary \ref{2ndmomentcor} separately since the leading order coefficient admits a surprisingly simple new formula.

\begin{conjecture}
    Let $n_1>\cdots>n_k \in \mathbb{N}\cup \{0\}$, $h_1,\ldots, h_k\in (0,\infty)$ and $s=\sum_{j=1}^k h_j$. Then, we have that
    
    \begin{align*}
        \frac{1}{T}\int_0^{T} \prod_{j=1}^k \left|\zeta^{(n_j)}\left(\frac{1}{2}+\textnormal{i}t\right)\right|^{2h_j} \mathrm{d}t \sim  \left(\log T\right)^{s^2+2\sum_{j=1}^k h_j n_j} \beta(s) \mathbb{E}\left[\prod_{j=1}^k \Bigg|\sum_{m_{j}=0}^{n_{j}}(-\textnormal{i})^{m_j} \binom{n_{j}}{m_{j}} \frac{\mathsf{Y}_{n_j-m_j}(s)}{2^{n_j}}\Bigg|^{2h_j}\right]
    \end{align*}
    where 
    \begin{equation*}
        \beta(s)\defeq \frac{G(s+1)^2}{G(2s+1)} \prod_{\text{primes} \ p} \left(1-p^{-1}\right)^{s^2}\sum_{k=0}^{\infty} p^{-k}\left(\frac{\Gamma(k+s)}{\Gamma(k+1)\Gamma(s)}\right)^2.
    \end{equation*}
     In particular, for any $n\in \mathbb{N}$ and $s>\frac{1}{2}$ we have
    \begin{align*}
        \frac{1}{T}\int_0^T \left|\zeta^{(n)}\left(\frac{1}{2}+\textnormal{i}t\right)\right|^{2} \left|\zeta\left(\frac{1}{2}+\textnormal{i}t\right)\right|^{2s-2} \mathrm{d}t \sim \left(\log T\right)^{s^2+2n} \beta(s) \frac{2s-1}{2s-1+2n}\prod_{l=1}^{n}\left(\frac{l+s-1}{l+2s-2}\right)^2.
    \end{align*}
    Similarly, we have the analogous relations:
    \begin{align*}
        &\frac{1}{T}\int_0^{T} \prod_{j=1}^k \left|\mathcal{Z}^{(n_j)}(t)\right|^{2h_j} \mathrm{d}t \sim  \left(\log T\right)^{s^2+2\sum_{j=1}^k h_j n_j} \beta(s)  2^{-2\sum_{j=1}^kh_j n_j  } \mathbb{E}\left[\prod_{j=1}^k\left|\mathsf{Y}_{n_j}(s)\right|^{2h_j}\right],\\ &\frac{1}{T}\int_0^T \left|\mathcal{Z}^{(n)}(t)\right|^{2} \left|\mathcal{Z}(t)\right|^{2s-2} \mathrm{d}t \sim \left(\log T\right)^{s^2+2n} \beta(s) \frac{2s-1}{\prod_{l=1}^{n}(2s-2+l)^2}
\\&\Bigg\{\sum_{i,j=0}^{n}\binom{n}{i}\binom{n}{j}\left(-\frac{1}{2}\right)^{2n-i-j}\frac{1}{2s-1+i+j}
\prod_{l=0}^{i-1}(l+s)\prod_{l=0}^{j-1}(l+s)\prod_{l=i+1}^{n}(l+2s-2)\prod_{l=j+1}^{n}(l+2s-2)\Bigg\}.
    \end{align*}
\end{conjecture}

Proving these conjectures in full generality would appear to be extremely difficult and is most likely out of reach of current techniques in analytic number theory. On the other hand, hard number theoretic conjectures originating from random matrix theory considerations have been proved in the past. For example, see \cite{arguin2020fyodorov,arguin2023fyodorovhiarykeating,arguin2019maximum,najnudel2018extreme} for a recent proof of a conjecture \cite{FyodorovKeating,FHK} concerning the maximum of the Riemann $\zeta$ function in typical short intervals on the critical line. It is worth noting that some particular cases of our conjectures have been proved. Specifically, the conjectures (in equivalent form) are proven for $k=2$, with $n_1=1$, $n_2=0$, for the following special exponents $h_1,h_2$ by: Hardy and Littlewood \cite{HardyLittlewood} for $h_1=0, h_2=1$ in 1916, Ingham \cite{Ingham} for $h_1=0, h_2=2$ and $h_1=1, h_2=0$ in 1927, Conrey \cite{Conrey} for $h_2=2, h_1=1$ and $h_2=2, h_1=0$ in 1988 and Conrey and Ghosh \cite{ConreyGhosh} for $h_1=\frac{1}{2}, h_2=\frac{1}{2}$ in 1989 (in this case assuming assuming the Riemann hypothesis).  Finally, we point out that the above conjectures become theorems in the function field setting over $\mathbb{F}_q$ in the limit $q\to\infty$, using equidistribution results due to Deligne and Katz \cite{KatzSarnak}.

\subsection{Connection to integrable systems}

As remarked in Section \ref{HistorySection}, over the past two decades, the problem of joint moments was realized to be closely related to the theory of the Painlev\'e equations (see \cite{Basor_2019,Bailey_2019,assiotis2022joint,ABGS}). In the case of joint moments with only the first derivative, it was discovered only recently that the random variable $\mathsf{Y}_1(s)=-\mathsf{q}_1(s)$ is in fact responsible for this connection between joint moments and integrable systems. To be more precise, it was shown in \cite{assiotis2022joint} and \cite{ABGS} that, for any $s \ge 0$, the function $\tau^{(s)}$ defined by,
\begin{equation*}
    t \mapsto \tau^{(s)}(t)\overset{\textnormal{def}}{=}t \frac{\mathrm{d}}{\mathrm{d}t} \log \mathbb{E} \left[e^{\textnormal{i} t\frac{\mathsf{Y}_1(s)}{2}}\right]
\end{equation*}
satisfies the following form of the $\sigma$-Painlev\'e  III' equation: 
\begin{equation}
  \left(t\frac{\mathrm{d}^2\tau^{(s)}}{\mathrm{d}t^2}\right)^2=-4t\left(\frac{\mathrm{d}\tau^{(s)}}{\mathrm{d}t}\right)^3+\left(4s^2+4\tau^{(s)}\right)\left(\frac{\mathrm{d}\tau^{(s)}}{\mathrm{d}t}\right)^2+t\frac{\mathrm{d}\tau^{(s)}}{\mathrm{d}t}-\tau^{(s)},
  \label{eq:painleve3}
\end{equation}
for any $t\in \mathbb{R}\setminus \{0\}$, along with the initial conditions:
\begin{equation}
\begin{cases}
      \tau^{(s)}(0)=0, & \text{for}  \ s>0,\label{eq:boundary1}\\
      & \\
   \frac{\mathrm{d}}{\mathrm{d}t}\tau^{(s)}(t)\Bigr\rvert_{t = 0}=0, & \text{for} \ s>\frac{1}{2}. 
\end{cases}
\end{equation}
Here, we note that for the derivatives to exist at $t=0$, one requires further conditions on the parameter $s$. That is, in order to guarantee that
\begin{equation*}
    \frac{\mathrm{d}^m}{\mathrm{d}t^m} \mathbb{E}\left[e^{\textnormal{i}t\frac{\mathsf{Y}_1(s)}{2}}\right] \Big\rvert_{t=0}<\infty
\end{equation*}
one needs to have $s>\frac{m-1}{2}$. Moving on, we note that the result above can be useful in many ways. First of all, we can see immediately that all the even integer moments (when they exist) of $\mathsf{Y}_1(s)=-\mathsf{q}_1(s)$ have a representation in terms of the $\sigma$-Painlev\'e III' equation, since we have,
\begin{equation*}
\mathbb{E}\left[\left(\mathsf{Y}_1(s)\right)^{2h}\right]=2^{2h}(-1)^{h}\frac{\mathrm{d}^{2h}}{\mathrm{d}t^{2h}}\left[\exp\left(\int_0^t\frac{\tau^{(s)}(u)}{u}\di u\right)\right]\Bigg|_{t=0}.
\end{equation*}
But more ambitiously, if one can show that the characteristic function of $\mathsf{Y}_1(s)$ is in $L^1(\mathbb{R})$ and if it admits an explicit expression (and this is a highly non-trivial, being connected to deep results on classical solutions of Painlev\'e  equations; see \cite{Umemura1,Umemura2} for what is meant by a classical solution in the context of Painlev\'e  equations), then the expression for $\tau^{(s)}(t)$ can be used, by means of Fourier inversion, to recover an explicit density for $\mathsf{Y}_1(s)$. This is precisely what is done in \cite{ABGS}, for $s\in \mathbb{N}$, where the density is then used to give an explicit expression for any finite (including complex) moments of $\mathsf{Y}_1(s)$. 

Given the previous results presented in our paper, it is then natural to wonder whether certain joint moments of the random variables $\{\mathsf{Y}_n(s)\}_{n\geq 1}$ admit representations in terms of Painlev\'e  transcendents. We answer this question in the affirmative in Theorem \ref{painlevethm} below by representing such joint moments as linear combinations, with polynomial coefficients, of derivatives of solutions to the $\sigma$-Painlev\'e  III' equation. This is presented in terms of the allied family of random variables $\{\mathsf{q}_n(s)\}_{n\ge 1}$ which is more convenient. In fact, we show a more general result, related to a richer object, a certain one-dimensional function, from which the joint moments can be recovered.

Proving Theorem \ref{painlevethm} directly from the definition of the $\{\mathsf{q}_n(s)\}_{n\ge 1}$ appears to be difficult. Instead, we prove it by an approximation from finite $N$.  In order to do this, we first need to find the correct finite-dimensional analogue of the expression on the left hand side of (\ref{formula1}) below.  We then develop a method involving a class of Hankel determinants shifted by partitions (this will be given precise meaning below) that allows us, after rather long and intricate computations, to obtain the following structure theorem for finite $N$. Recall that, for $s\in \mathbb{C}$, $\mathbb{E}_N^{(s)}\left[\bullet\right]$ is interpreted as explained below display \eqref{HPintegralDef}.

\begin{thm}\label{structureforgeneralfinitesize}
Let $k\geq 2$, $n_{2},\ldots,n_{k}\in \mathbb{N}\cup\{0\}$ and $\sum n_j >0$. Let $s\in \mathbb{C}$ with $\Re(s)>2^{-1}(\sum_{q=2}^k qn_q-1)$. Then, for all $t_{1}\geq 0$,
\begin{align}\label{general structure1}
    \frac{(-2\textnormal{i})^{-\sum_{q=2}^{k}qn_{q}}}{N^{n_{2}+\cdots+n_{k}}}\mathbb{E}_{N}^{(s)}\Bigg[e^{-\textnormal{i} t_{1}\sum_{j=1}^{N}\frac{\mathsf{x}^{(N)}_{j}}{N}}\prod_{q=2}^{k}&\left(\sum_{j=1}^{N}\left(\mathsf{x}^{(N)}_{j}-\textnormal{i}\right)^q\right)^{n_{q}}\Bigg]\nonumber \\& =\frac{1}{t_{1}^{\sum_{q=2}^{k}qn_{q}-1}}
    \sum_{m=0}^{\sum_{q=2}^{k}(q-1)n_{q}}t_{1}^{m-1}P_{m}^{(s,N)}(t_{1})\frac{\mathrm{d}^{m}}{\mathrm{d}t_{1}^{m}}\mathbb{E}_{N}^{(s)}\Bigg[e^{-\textnormal{i} t_{1}\frac{1}{N}\sum_{j=1}^{N}\mathsf{x}^{(N)}_{j}}\Bigg],
\end{align}
where $ t_1^{m-1} P_{m}^{(s,N)}(t_{1})$ are polynomials of $t_{1}$ of degree at most $\sum_{q=2}^{k}qn_{q}-1$, respectively.  Moreover, the coefficients of these polynomials are polynomials in $N,s$, and with degrees in terms of $N$ or $s$ no more than $\sum_{q=2}^{k}(q-1)n_{q}$. 
\end{thm}

Analogously to $\mathsf{Y}_1(s)=-\mathsf{q}_1(s)$, whose characteristic function can be represented in terms of the $\sigma$-Painlev\'e  III' equation, its pre-limit version also admits a Painlev\'e  representation. More precisely, it was shown in \cite{ABGS} that the function defined by, for $s>-\frac{1}{2}$,
\begin{equation*}
    t \mapsto t \frac{\mathrm{d}}{\mathrm{d}t} \log \mathbb{E}_N^{(s)} \Bigg[e^{-\textnormal{i} \frac{t}{2N}\sum_{j=1}^{N}\mathsf{x}^{(N)}_{j}}\Bigg],
\end{equation*}
which we denote here by $\tau_N^{(s)}(t)$, satisfies, for all $t\in \mathbb{R}\setminus \{0\}$, the following form of the $\sigma$-Painlev\'e  V equation,
\begin{align*}
    \left(t\frac{\mathrm{d}^2\tau_{N}^{(s)}}{\mathrm{d}t^2}\right)^2=-4t\left(\frac{\mathrm{d}\tau_{N}^{(s)}}{\mathrm{d}t}\right)^3+\left(4s^2+4\tau_{N}^{(s)}+\frac{t^2}{N^2}\right)\left(\frac{\mathrm{d}\tau_{N}^{(s)}}{\mathrm{d}t}\right)^2+t\left(1+\frac{2s}{N}-\frac{2\tau_{N}^{(s)}}{N^2}\right)\frac{\mathrm{d}\tau_{N}^{(s)}}{\mathrm{d}t}& \\ -\left(1+\frac{2s}{N} -\frac{\tau_{N}^{(s)}}{N^2}\right)\tau_{N}^{(s)}.&
\end{align*}

In fact, the working behind proving Theorem \ref{structureforgeneralfinitesize} gives us an explicit (but admittedly complicated) recursive way to compute all the finite dimensional joint moments exactly. It is worth pointing out that having an explicit finite-$N$ expression is valuable in its own right: firstly, having an exact formula at the finite-$N$ level allows one to conjecture lower order terms in the asymptotics of the joint moments of $\zeta$, as in \cite{CFKRS2} (although we do not do this here). More importantly, the connection between random matrix theory and the $\zeta$-function extends to other $L$-functions.  For example, Katz and Sarnak \cite{KatzSarnak} proved several results and provide heuristic evidence, that various statistics in families of $L$-functions (including moments over the families) can be modelled by considering analogous statistics over Haar-distributed unitary, orthogonal and symplectic matrices, depending on the symmetry type of the family (see \cite{ILS,Rubinstein, KeatingSnaithLfunctions} for a number of related results). It follows from the work of Katz and Sarnak that the explicit finite-$N$ expressions for joint moments of charactersistic polynomials of random matrices correspond {\em precisely} to the joint moments of $\zeta$-functions associated with curves over finite fields of $q$ elements in the limit $q\to\infty$. Hence, the relation between the finite-$N$ expression and Painlev\'{e}-V gives a direct connection between analytic number theory and integrable systems in this context. To give the reader a flavour of the explicit finite-$N$ expressions we can obtain, the following is a consequence of the machinery developed in Section \ref{SectionPainleve}.

\begin{prop}\label{finiteNexplicit}
For all $s\in(\frac{1}{2},\infty)$,
\begin{align*}
    \mathfrak{F}_N^{(2,0)}(1,s-1)&= \frac{\mathfrak{F}_N^{(0)}(s)}{16} \Bigg(\frac{N^4}{(2s+3)(2s-1)}
+\frac{4sN^3}{(2s+3)(2s-1)}\\
&+\frac{4(2s^3+s^2-1)N^2}{(2s+3)(2s-1)(2s+1)}-\frac{8sN}{(2s+3)(2s-1)(2s+1)}\Bigg).
\end{align*}
\end{prop}

Returning to our main goal, the following is our main result on connections to integrable systems; again recall the characteristic function of $\mathsf{q}_1(s)=-\mathsf{Y}_1(s)$ is given in terms of $\sigma$-Painlev\'e  III' transcendents. In the simplest possible case: with $k=2$, $s\in \mathbb{N}$, and after taking the derivative with respect to $t_1$ and evaluating at $t_1=0$, Theorem \ref{painlevethm} recovers (in equivalent form; the connection to the random variables $\{\mathsf{q}_n(s)\}_{n\ge 1}$ was not clear yet in \cite{keating-fei}) the main result of \cite{keating-fei}. The following result will be proven by a limit transition as $N \to \infty$ of the various quantities involved in Theorem \ref{structureforgeneralfinitesize}. Making these $N \to \infty$ limits rigorous from their explicit $N \times N$ (generalised)  determinant representations is very difficult, but fortunately we can use our earlier probabilistic limit results as input.

\begin{thm}\label{painlevethm}
Let $k\geq 2$, $n_{2},\ldots,n_{k}\in \mathbb{N}\cup\{0\}$ and $\sum n_j >0$ and suppose $s \in \mathbb{R}$ with $s>2^{-1}({\sum_{\ell=2}^{k}\ell n_{\ell}-1})$. Then, for $t_{1}\geq 0$,
\begin{align}\label{formula1}
 (-2\textnormal{i})^{-\sum_{\ell=2}^{k}\ell n_{\ell}} \mathbb{E}\left[e^{-\textnormal{i} t_{1}\mathsf{q}_{1}(s)}\prod_{j=2}^{k}\mathsf{q}_{j}(s)^{n_{j}}\right]=\frac{1}{t_{1}^{\sum_{\ell=2}^{k}\ell n_{\ell}-1}}
\sum_{m=0}^{\sum_{\ell=2}^{k}(\ell-1) n_{\ell}} t_1^{m-1} \mathcal{A}_{m}^{(s)}(t_{1})\frac{\mathrm{d}^{m}}{\mathrm{d}t_{1}^{m}}\mathbb{E}\Big[e^{-\textnormal{i} t_{1}\mathsf{q}_{1}(s)}\Big],   
\end{align}
where $t_1^{m-1}\mathcal{A}_{m}^{(s)}(t_{1})$ are polynomials in $t_{1}$ of degree at most $\sum_{\ell=2}^{k}(\ell-1)n_{\ell}-1$ and whose coefficients are polynomials themselves of $s$, with degree in $s$ at most $\sum_{\ell=2}^{k}(\ell-1)n_{\ell}$. 
\end{thm}
\noindent Even though we state Theorems  \ref{structureforgeneralfinitesize} and \ref{painlevethm} for $t_1\geq 0$, by noting that
\begin{equation*}
    \mathsf{Law}\big(-\mathsf{x}_1^{(N)},-\mathsf{x}_2^{(N)}\ldots, -\mathsf{x}_N^{(N)}\big)= \mathsf{Law}\big(\mathsf{x}_1^{(N)},\mathsf{x}_2^{(N)}\ldots, \mathsf{x}_N^{(N)}\big)
\end{equation*}
one immediately gets analogous results, both for finite $N$ and in the $N\to \infty$ limit, for $t_1<0$. We omit the statements.

\indent Moving on, the final piece of the puzzle to get Theorem \ref{painlevethm}, or more precisely Theorem \ref{structureforgeneralfinitesize}, is an intricate induction on the indices $n_2,\ldots, n_k$. As alluded to above, the arguments we use for the inductive step will give us a recursive way to obtain explicit formulae for the polynomials $\mathcal{A}_{m}^{(s)}(t_1)$, and hence for joint moments of the random variables $\mathsf{q}_j(s)$ (or equivalently, $\mathsf{Y}_j(s)$). Namely, one can consider the asymptotic expansion of the characteristic function of $\mathsf{q}_1(s)$, as $t_1\rightarrow 0$, and make use of its relation to the $\sigma$-Painlev\'e III' equation to obtain integer moments of $\mathsf{q}_1(s)$ (see for instance \cite{Basor_2019}), allowing one, by virtue of \eqref{formula1}, to compute all the integer joint moments recursively and express them as rational functions of $s$ efficiently. We give a few examples of these formulae in the corollary below. Despite their simplicity, these formulae are new and seem hard to obtain in any other way (and also do not follow from Theorem \ref{mainresult2} since this only involves joint moments and not the characteristic function itself which is a richer object). In particular, they seem very challenging to prove directly from the determinantal point process definition.

\begin{cor}\label{painlevecor}
For all $s\in (\frac{1}{2},\infty)$, we have
\begin{equation*}
    \E\Big[
e^{-\textnormal{i} t_{1}\mathsf{q}_{1}(s)}\mathsf{q}_{2}(s)
\Big]
= - \frac{2s}{t_{1}} \frac{\mathrm{d}}{\mathrm{d}t_{1}} 
\E \Big[
e^{-\textnormal{i} t_{1}\mathsf{q}_{1}(s)}
\Big],
\end{equation*}
and for all $s\in (\frac{3}{2},\infty)$,
\begin{equation*}
 \E\Big[
e^{-\textnormal{i} t_{1}\mathsf{q}_{1}(s)}\mathsf{q}_{2}(s)^2
\Big]=
\frac{4s^2+2}{t_{1}^2} \frac{\mathrm{d}^2}{\mathrm{d}t_{1}^2} 
\E \Big[
e^{-\textnormal{i} t_{1}\mathsf{q}_{1}(s)} 
\Big]
-\frac{12s^2}{t_{1}^3}\frac{\mathrm{d}}{\mathrm{d}t_{1}} 
\E\Big[
e^{-\textnormal{i} t_{1}\mathsf{q}_{1}(s)}  
\Big]-\frac{2}{t_{1}^2}\E\Big[
e^{-\textnormal{i} t_{1}\mathsf{q}_{1}(s)}
\Big]. 
\end{equation*}
\end{cor}

The above formulae looks at first sight to be singular when $t_1=0$, but this is not the case:  the moments of $\mathsf{q}_1(s)$ and $\mathsf{q}_2(s)$ satisfy certain relations and this guarantees that the singular $t_1$ powers in the denominator cancel out at $t_1=0$. In particular, evaluating at $t_1=0$ allows one to get, for instance, the following simple formulae:
\begin{equation*}
    \mathbb{E}[\mathsf{q}_2(s)]=2s \, \mathbb{E}\left[\mathsf{q}_1^2(s)\right]=\frac{2s}{4s^2-1}.
\end{equation*}

Finally, we would like to emphasize that Theorem \ref{painlevethm} and Corollary \ref{painlevecor} not only give Painlev\'e  representations for integer joint moments of $\mathsf{q}_n(s)$, but also for joint moments where the power of $\mathsf{q}_1(s)$ is any (suitable) real number via the  elementary identity
\begin{equation}
    |y|^p = C_p\int_0^\infty \frac{1-\cos(t y)}{t^{p+1}} \mathrm{d}t, \label{ExplicitCp}
\end{equation}
which holds for all $y\in \mathbb{R}$, $p\in (0,2)$ with
\begin{equation}\label{explicitformulaforCp}
    \frac{1}{C_p}= \int_0^\infty \frac{1-\cos(t)}{t^{p+1}}\mathrm{d}t=\frac{1}{p} \cos\left(\frac{\pi p}{2}\right)\Gamma(1-p),
\end{equation}
where the integral is computed via a generalized Fresnel integral. Indeed, using this one immediately gets that if $h=n_1+p$, where $n_1\in 2\mathbb{N}$, $p\in (0,2)$, then  
\begin{equation*}
    \mathbb{E}\left[|\mathsf{q}_1(s)|^{h} \prod_{j=2}^{k}\mathsf{q}_{j}(s)^{n_{j}}\right]= C_p \int_{0}^\infty \frac{\Re\left\{\mathbb{E}\left[\prod_{j=1}^{k}\mathsf{q}_{j}(s)^{n_{j}}\left(1-e^{-\textnormal{i}t \mathsf{q}_1(s)}\right)\right]\right\}}{t^{p+1}}  \mathrm{d}t,
\end{equation*}
where the interchange of the expectation and the integral is justified by use of the trivial inequality $1-\cos(x)\leq \min(2,x^2)$. In particular, the average inside the integral can be written, via Theorem \ref{painlevethm}, in terms of Painlev\'e  transcendents. Furthermore, using the explicit formula obtained in \cite[Proposition 2.6]{ABGS} for $s\in \mathbb{N}$,
\begin{equation}
   \boldsymbol{\varphi}_s(t) \defeq \frac{(-1)^{s(s-1)/2}}{2^{s^2/2}}\frac{G(2s+1)}{G(s+1)^2} \times\frac{\det_{0\leq j,k\leq s-1}\left[I_{j+k+1}\left(2\sqrt{|2t|}\right)\right]}{e^{|t|}|t|^{s^2/2}}=\mathbb{E}\left[e^{-\textnormal{i}t\mathsf{q}_1(s)}\right],\label{ExplicitCharFn}
\end{equation}
where $I_\alpha$ denotes the modified Bessel function of the first kind, one can recover even further explicit formulae for joint moments of $\mathsf{q}_n(s)$, and hence of $\mathsf{Y}_n(s)$, which do not follow from the probabilistic techniques we develop to prove the results of Section \ref{convergencesubsec}. Here, we write out this formula for the cases where the polynomials $\mathcal{A}_m^{(s)}(t)$ admit rather simple expressions; namely, the cases given in Corollary \ref{painlevecor}. To keep our final expression relatively simple, we will assume that $n_1=0$ in the setting described above. 
\begin{cor}
    For all $p\in (0,2)$, and $s\geq 2$, $s\in \mathbb{N}$,
\begin{equation*}
\mathbb{E}\left[|\mathsf{q}_1(s)|^{p}\mathsf{q}_2(s)^2\right]= C_p \int_{0}^\infty\frac{D_s \boldsymbol{\varphi}_s (0) - D_s\boldsymbol{\varphi}_s(t)}{t^{p+1}} \mathrm{d}t,
\end{equation*}
where $\boldsymbol{\varphi}_s$ is as in (\ref{ExplicitCharFn}), $D_s$ is the differential operator defined, for each $f\in C^2(\mathbb{R}$) as,
\begin{equation*}
    D_s f(t)\defeq \frac{4s^2+2}{t^2}\frac{\mathrm{d}^2 f}{\mathrm{d}t^2}-\frac{12s^2}{t^3} \frac{\mathrm{d}f}{\mathrm{d}t}- \frac{2}{t^2}f,
\end{equation*}
and the constant $C_p$ is as in (\ref{explicitformulaforCp}).
\end{cor}

\subsection{Strategy of Proof} 

We give here an outline of the proofs of our main theorems and the difficulties that arise. We first note that the multiple-contour integral expressions used in \cite{keating-fei,keatingwei} and the symmetric function theory framework of \cite{barhoumi2020new} cannot be used to attack the joint moments problem in complete generality as the formulae therein do not make sense for non-integer exponents $h_i$. Thus, we need a different starting point. The first step is to transform the problem into one of moments over Hermitian matrices, with averages taken with respect to the Cauchy random matrix ensemble whose law of eigenvalues is given by (\ref{hpintrodef}). In the simplest case of joint moments with only the first derivative, namely with $k=2$, $n_2=1,n_1=0$, this connection was implicit in Winn's work \cite{winn2012derivative} and brought to center stage in \cite{assiotis2022joint}. It is not hard to prove that it extends to arbitrary joint moments. More precisely, a computation, given in Proposition \ref{connections to the symmetric polynomials}, establishes that both $\mathfrak{G}_N^{(n_1,\dots,n_k)}(h_1,\dots,h_k)$ and $\mathfrak{F}_N^{(n_1,\dots,n_k)}(h_1,\dots,h_k)$  can be written as expectations of certain explicit symmetric polynomials in the eigenvalues of a $N\times N$ Cauchy random matrix $\mathbf{H}_N$. 

At this stage it can be shown\footnote{This allowed us to guess the form of our first main result.}, by putting together general results from \cite{olshanskivershik,Borodin_2001,assiotis2020boundary}, that symmetric polynomials in the eigenvalues of $\mathbf{H}_N$ converge jointly in distribution. We stress however that convergence of the moments does not follow from the general arguments of \cite{olshanskivershik,Borodin_2001,assiotis2020boundary} and one needs non-trivial estimates for this. The Cauchy ensembles do have additional structure, that of a determinantal point process, which potentially could be used for that.  One would like to reduce this problem to uniform in $N$ estimates on the correlation kernel which is given in terms of orthogonal polynomials with respect to an $N$-dependent weight. On the other hand, a naive application of Hadamard-type inequalities will always give estimates that blow up with $N$. The crux of the problem is that one needs to take into account a crucial cancellation due to symmetry around $0$, but in order to do this we have to expand the various determinants in the formulae. The problematic integral terms that cannot be controlled uniformly in $N$, if they appear with absolute values in the integrand, will then exactly cancel out due to symmetry but we lose the determinant structure. It is possible to remedy this issue, to some extent, using some tricks involving the reproducing property of the correlation kernel, but we needed to put restrictions on the exponents that get worse and worse as the number of joint moments we are considering increases. In particular, the end result would have been both sub-optimal and its proof technically involved. 

Instead we take a different approach: if one looks at the problem through the lens of a hidden exchangeable structure, all analytic difficulties disappear entirely. This requires only two, algebraic properties from the Cauchy ensembles: permutation-invariance and consistency under top-left corner projections; in particular all the matrices $\mathbf{H}_N$ are coupled in a single infinite random matrix $\mathbf{H}$ by looking at its $N\times N$ top-left corner. Then, our starting point is the following observation, which as far as we can tell is new\footnote{The fact that the diagonal elements of the random matrix $\mathbf{H}$ form a one-dimensional exchangeable sequence was known already \cite{olshanskivershik,Borodin_2001,assiotis2020boundary}, and also played a role in previous works on this problem in the case of a single joint moment \cite{assiotis2022joint,assiotis2022convergence}. However, both the existence and certainly the use of the higher-dimensional exchangeable structure below is new.}. If $J\in \tilde{\mathbb{N}}^{(k)}$, where $\tilde{\mathbb{N}}^{(k)}$ is the set of subsets of $\mathbb{N}$ of size at most $k$, with $J$ having elements $i_1<i_2<\cdots<i_l$, and if we define the random variables $\mathsf{X}_J$ by
\begin{equation*}
    \X_J\overset{\textnormal{def}}{=} \det\left(\mathbf{H}_{j,m}\right)_{j,m=i_1,\dots,i_l},
\end{equation*}
then $(\mathsf{X}_J)_{J\in \tilde{\mathbb{N}}^{(k)}}$ forms an exchangeable array in the sense of Definition \ref{ExchDef}. This turns out to be very useful as the rescaled elementary symmetric polynomials in the eigenvalues $(\mathsf{x}_1^{(N)}, \mathsf{x}_2^{(N)}, \ldots, \mathsf{x}_N^{(N)})$ of $\mathbf{H}_N$ are nothing but the U-statistics of the exchangeable array by virtue of the linear algebraic identity:
\begin{equation*}
    \mathrm{e}_k\left(\mathsf{x}_1^{(N)}, \mathsf{x}_2^{(N)}, \ldots, \mathsf{x}_N^{(N)}\right)=   \sum_{J\in \{1,\dots,N\}^{(k)}} \X_J,
\end{equation*}
where $\{1,\dots,N\}^{(k)}$ denotes subsets of $\{1,\dots,N\}$ of size exactly $k$. Then, convergence of the moments (and almost sure convergence) follows by a backward martingale convergence argument with respect to the exchangeable filtration $\left(\mathcal{E}_N\right)_{N\ge 1}$ and $\sigma$-algebra $\mathcal{E}=\cap_{N}  \mathcal{E}_N$. Finally, in order to identify the limiting moments with the joint moments of the random variables $\mathsf{Y}_n(s)$ we make use of the results of \cite{Borodin_2001,Qiu}. The main tool employed in \cite{Borodin_2001,Qiu} is the theory of determinantal point processes and using the determinantal structure in this part of the argument does seem unavoidable.

The more interesting use of exchangeability, however, is in proving Theorem \ref{mainresult2}. Our starting point is a new (as far as we can tell) formula for the joint moments of limiting U-statistics in terms of joint moments of elements in the array indexed by disjoint sets. This is given in Proposition \ref{arrayfiniteaverageprop}. By taking into account certain non-trivial combinatorial cancellations it is possible to show further that this formula is equal to the formula in Theorem \ref{mainresult2} in the more general Proposition \ref{finiteaverageprop2}.  Now, in order to prove the formula in Proposition \ref{arrayfiniteaverageprop} one needs the following higher-dimensional generalisation of de-Finetti's theorem, which we were surprised not to find anywhere in the literature. The statement is simple: elements of the array indexed by disjoint sets are conditionally independent given the exchangeable $\sigma$-algebra $\mathcal{E}$. This statement is likely known to experts and in fact can be reduced to some deep structural results on exchangeable arrays, known as the Aldous-Hoover-Kallenberg representation theorems \cite{KallenbergMonograph}, which give a representation of the array in terms of independent uniform random variables. However, in order to make the paper self-contained and accessible to non-experts we give a direct proof from first principles which appears to be new. It goes as follows. It is not hard to show this statement with $\mathcal{E}$ replaced by an auxiliary tail $\sigma$-algebra which is nevertheless different from $\mathcal{E}$. We do note that we need the statement with the $\sigma$-algebra being precisely $\mathcal{E}$ since it is  the one which is well-adapted to convergence of U-statistics. Thankfully, we can show that although the two $\sigma$-algebras are different, $\mathcal{E}$ and the auxiliary $\sigma$-algebra coincide up to sets of measure zero (with respect to the law of the array), and this completes the proof.

Proving Theorem \ref{painlevethm} on connections to integrable systems is the most challenging part of the paper, with all the machinery we built earlier along with key new ideas coming together. The method is partly inspired by the computation in \cite{keating-fei} which treats the special case $k=2$ and $s\in \mathbb{N}$ (after taking the derivative with respect to $t_1$ and evaluating at $t_1=0$), in a more ad-hoc way. This method essentially boils down to establishing exact relations within a class of Hankel determinants shifted\footnote{The precise meaning of this will be explained in Section \ref{SectionPainleve}.} by partitions. However, our proof is not simply a generalisation of the computation in \cite{keating-fei} as there are some fundamental differences. First of all, the computation in \cite{keating-fei} is done directly in the limit and it is heavily reliant on the fact that (a weighted version of) the characteristic function of $\mathsf{Y}_1(s)=-\mathsf{q}_1(s)$ is explicit\footnote{We note that neither $\mathsf{Y}_1(s)$ nor its characteristic function appear explicitly in \cite{keating-fei}, nor it is easy to guess this connection from the working in \cite{keating-fei}, but this is morally the reason the computation there works.}, for $s\in \mathbb{N}$,  in terms of $s\times s$ determinants involving Bessel functions, see (\ref{ExplicitCharFn}). For non-integer $s$ this is no longer the case; as far as we know, there is no explicit formula of any kind for this characteristic function. 

Because of this fundamental obstacle we must take a different approach. We go back to the finite matrix setting and work with $N$-dimensional integrals that approximate the desired quantities in the large $N$-limit. We first need to find what  the right analogue of the left hand side of (\ref{formula1}) is for finite $N$, a task for which we are inspired by the probabilistic connection to the random variables $\left\{\mathsf{q}_n(s)\right\}_{n\ge 1}$. There are of course many choices that will approximate the left hand side of (\ref{formula1}) in the large-$N$ limit but a judicious choice needs to be made that will have all the exact solvability properties that we require and that is the one presented in Theorem \ref{structureforgeneralfinitesize}. These exact relations that we are alluding to we first obtained by formally expanding the exponential in the characteristic function of certain random variables and proceeded to exchange sums and expectations to perform explicit computations. However, these computations are completely formal, as not only does such a series expansion not converge but in fact only finitely many terms in it even make sense, as these random variables\footnote{The reader should simply think of these as generalisations of the Cauchy random variable.} have finitely many moments. Nevertheless, once we know what these relations should be, we prove them rigorously by a completely different and highly non-trivial argument. By initially restricting to $s\in \mathbb{N}$, we introduce a Hankel matrix, whose elements are given by a type of multivariate generating function of the Fourier transform of certain Cauchy-like weights. That is to say, we construct functions $\boldsymbol{\psi}_\gamma \in C^\infty(\mathbb{R}^k, \mathbb{R})$ such that if we define their Hankel determinant by,
\begin{align*}
    \mathbf{\Psi}_N(t_1,\ldots,t_k)= \det_{1\leq i,j\leq N}\left[\boldsymbol{\psi}_{i+j}(t_1,\ldots, t_k)\right],
\end{align*}
then the relations we desire to obtain at the finite-$N$ level boil down instead to proving a relation of the form:
\begin{equation*}
        \prod_{m=2}^k \frac{\partial^{\ell_{m}}}{\partial t_m ^{\ell_{m}}} \mathbf{\Psi}_N(t_1,\ldots, t_k) \Big|_{t_2,\ldots, t_k=0}= \sum_m P_m(t_1;N) \frac{\mathrm{d}^m }{\mathrm{d}t_1^m } \mathbf{\Psi}_N(t_1,0,\ldots,0),
\end{equation*}
for certain polynomials $P_m(\cdot;N)$. We then prove this relation by a delicate induction argument on the order of the derivatives $\ell_2,\ldots,\ell_k$. However, to perform this induction, it transpires that we are required to introduce a generalization of the determinants $\mathbf{\Psi}_N$, indexed by integer partitions $\boldsymbol{\lambda}$ denoted by $\mathbf{\Psi}_{N,\boldsymbol{\lambda}}$. Then, we develop a systematic way to perform the daunting computations needed for the induction argument, which requires delicate information on the $N$-dependence  of the functions $P_m(\cdot;N)$, simultaneously for $\mathbf{\Psi}_N$ and also $\mathbf{\Psi}_{N,\boldsymbol{\lambda}}$ as $\boldsymbol{\lambda}$ ranges over a specific class of partitions. To conclude the proof of the result at the finite-$N$ level we remove the restriction on $s$ by an analytic continuation argument to obtain Theorem \ref{structureforgeneralfinitesize}. Finally, by virtue of our previous probabilistic results, we are able to take the rigorous large-$N$ limit of all the quantities involved in Theorem \ref{structureforgeneralfinitesize}, while paying attention to which terms vanish in the limit, which from the explicit $N$-dimensional integral expressions is not obvious at all.

\subsection{Connections to various areas}

 The theory of exchangeability plays a central role in our arguments. This theory begins with de-Finetti's classical theorem \cite{aldous2} on sequences of random variables. The higher-dimensional structure, that of an exchangeable array and its variations, originates with the seminal works of Hoover \cite{Hoover} and Aldous \cite{AldousRep,aldous2}. It was later developed by Kallenberg \cite{KallenbergArrays1,KallenbergArrays2}, whose monograph \cite{KallenbergMonograph} contains the state of the art of the general theory. There are also many interesting extensions of the notion of exchangeability in different directions, see for example \cite{DeFinettiFinite,DeFinettiMarkov}. Through the decades exchangeability has found many applications, see \cite{AldousUsesExch,AldousICM,AustinExchangeability}. Our work presents a novel, concrete, and rather surprising, application of exchangeability and connects this branch of probability theory to some deep questions in analytic number theory and  to integrable systems.  Finally, there is a body of work on exchangeability and random matrices and on permutation-invariant ensembles, see \cite{ChatterjeeLindeberg} and \cite{Male} for representative examples, but both the questions asked and techniques used to answer them, are different from ours.

Another area, which although we do not use techniques from directly, provided important intuition is the study of measures invariant under the action of inductive limit groups, see \cite{olshanskivershik}. The point process $\mathbf{P}^{(s)}$ first appeared in this setting in the work of Borodin and Olshanski \cite{Borodin_2001}. There is already a detailed exposition of this connection in \cite{assiotis2022convergence,assiotis2022joint} and so we will be brief. The probability measures on the space of infinite Hermitian matrices $\mathbb{H}(\infty)$ which are ergodic under the action of the infinite-dimensional unitary group $\mathbb{U}(\infty)$ were classified in \cite{olshanskivershik,Pickrell} and are parametrized by an infinite-dimensional space $\boldsymbol{\Omega}$. Any invariant, not necessarily ergodic, measure on $\mathbb{H}(\infty)$, by general results, decomposes into ergodic measures via an abstract probability measure $\mathfrak{m}$ on $\boldsymbol{\Omega}$, see \cite{Borodin_2001}. The problem of ergodic decomposition amounts to finding an explicit description of the probability measure $\mathfrak{m}$ (see \cite{OlshanskiHarmonic} for a nice exposition in a related setting). The Cauchy measures on $\mathbb{H}(\infty)$, with parameter $s$, are invariant, and Borodin and Olshanski began the study of their ergodic decomposition measure, denoted by $\mathfrak{m}^{(s)}$, in \cite{Borodin_2001}. In the same paper they gave a partial description of $\mathfrak{m}^{(s)}$ in terms of $\mathbf{P}^{(s)}$. The description of $\mathfrak{m}^{(s)}$ was then finally completed by Qiu in \cite{Qiu} and we remark that this also where the interesting random variable $\mathsf{Y}_1(s)=-\mathsf{q}_1(s)$ first appeared.

Finally, let us say a word on relations to previous works in integrable systems. As far as we are aware, no analogue of the method we present here appeared before, beyond the special case in \cite{keating-fei}. Nevertheless, the computations are influenced, on a conceptual level at least, by the Japanese school of integrable systems that studies Hankel determinants as $\tau$-functions of Painlev\'e equations, see \cite{Okamoto1,Okamoto2}, a circle of ideas that was later developed in the context of gap probabilities for random matrices by Forrester and Witte in \cite{ForresterWittePainleve1,ForresterWittePainleve2}. We believe our method can be used to give connections to integrable systems in a number of allied problems, for example when we take averages over the orthogonal and symplectic groups instead \cite{KeatingSnaithLfunctions,Bettinetal,andrade2024joint,gunes2022characteristic}. Carrying it through though is a substantial task and it will be performed elsewhere.

\subsection{Some questions}

We end this introduction with some questions. There exists a natural generalisation of the law of eigenvalues of Haar-distributed unitary matrices called the circular $\beta$-ensemble (C$\beta$E) depending on a parameter $\beta \ge 0$ (random unitary matrices correspond to $\beta=2$), see \cite{ForresterBook}. The problem of establishing the asymptotics of the joint moments of the characteristic polynomial makes sense in the setting of C$\beta$E and there has been some previous work \cite{forrester2022joint,assiotis2022convergence} in the special case of a joint moment with a single derivative. Our methods extend word-for-word for $\beta=1,4$ to prove corresponding results for general joint moments since permutation-invariant matrix models exist in these cases (now with real and quaternion entries instead of complex numbers). We believe a variation of our methods will also be useful in  tackling the problem for general $\beta$. 

A different variant of the problem is to consider joint moments of characteristic polynomials of Haar-distributed orthogonal and symplectic matrices, see \cite{Bettinetal,andrade2024joint,gunes2022characteristic} for previous work. These asymptotics give conjectures for moments of L-functions in analytic number theory, see \cite{KeatingSnaithLfunctions,Bettinetal,andrade2024joint,gunes2022characteristic}. In the case of the second derivative, these joint moments were shown to be related to the sum of inverse points of the Jacobi ensemble, see \cite{gunes2022characteristic}. Hence, one would expect that joint moments of higher order derivatives admit representations in terms of higher order Laurent polynomials in the eigenvalues of a Jacobi random matrix. It is foreseeable that the methods we use to prove Theorem \ref{painlevethm} would allow one to obtain an expression for these higher order averages in terms of the lowest order one, whose asymptotics are already known. More generally, we expect that our methods can be used to show similar phenomena for all classical ensembles of random matrices; namely give an expression for averages of polynomials of arbitrary degrees in the eigenvalues in terms of averages of the first-order statistics.

In relation to integrable systems, it would be interesting if the joint characteristic function (or a variant) of several of the $\mathsf{Y}_n(s)$ or $\mathsf{q}_n(s)$ random variables is shown to satisfy a closed partial differential equation, which reduces to the $\sigma$-Painlev\'e III' equation when all the variables except the one corresponding to $\mathsf{Y}_1(s)=-\mathsf{q}_1(s)$ are put to $0$. Finally, as mentioned earlier, for $s \in \mathbb{N}$, the characteristic function of $\mathsf{Y}_1(s)=-\mathsf{q}_1(s)$ is explicit in terms of Bessel functions and this allowed to compute all the moments of $\mathsf{Y}_1(s)$ in \cite{ABGS}. It would be very interesting, if for $s\in \mathbb{N}$, the characteristic functions of any of the random variables $\mathsf{Y}_n(s), \mathsf{q}_n(s)$, for $n\ge 2$, are explicit.

\paragraph{Acknowledgements.} All four authors are grateful to the organisers of the conference ``Random Matrices from Quantum Chaos to the Riemann Zeta Function" where this work was initiated. Part of this research was performed while TA was visiting the Institute for Pure and Applied Mathematics (IPAM), which is supported by the National Science Foundation (Grant No. DMS-1925919), and when FW was supported by  NSF grant DMS-1854398.  We thank Nick Simm for bringing our attention to paper \cite{AkemannVernizzi} to compute the integral in Lemma \ref{lem:0820} and David Aldous and Tim Austin for useful correspondence on the general theory of exchangeable arrays.

\paragraph{Open Access.} For the purpose of Open Access, the authors have applied a CC BY public copyright licence to any Author Accepted Manuscript (AAM) version arising from this submission.

\section{Exchangeable arrays and random matrices}\label{consistentsection}

\subsection{Exchangeable arrays}
We present some general results on exchangeable arrays that will be applied to moments of random matrices in the sequel, this application being the main novel contribution of this part of the paper. The techniques we use are (unsurprisingly) rooted in the general theory of exchangeability \cite{KallenbergMonograph}. We first fix some notation. For  each $d\in \mathbb{N}$, and $J\subseteq \mathbb{N}$, let $J^{(d)}$ denote the set of subsets of $J$ of size $d$. Also, let $\Tilde{J}^{(k)}= \cup_{d\leq k}  J^{(d)}$. Let $[N]=\{1,2,\ldots, N\}$ and denote by $\mathsf{Sym}_N$ the symmetric group acting on $[N]$. Finally, recall that for an arbitrary random variable $\mathcal{X}$ taking values in a standard Borel space we write $\mathsf{Law}(\mathcal{X})$ for its law.

An exchangeable array will be our main object of study and is defined as follows.

\begin{defn}\label{ExchDef}
Let $k\ge 1$. A collection of random variables $(\X_J)_{J\in \Tilde{\mathbb{N}}^{(k)}}$, with $\mathsf{X}_J$ taking values in a standard Borel space $(\mathcal{V}, \mathfrak{V})$, is called an exchangeable array if for any permutation $\tau$ of $\mathbb{N}$ that fixes all but finitely many elements, we have that
    \begin{equation}
        \mathsf{Law}\left\{(\mathsf{X}_J)_{J\in \Tilde{\mathbb{N}}^{(k)}}\right\}=       \mathsf{Law}\left\{\left(\X_{\tau(J)}\right)_{J\in \Tilde{\mathbb{N}}^{(k)}}\right\},
    \end{equation}
    where if $J=\{a_1,\ldots, a_d\}$, we define $\tau(J)= \{\tau(a_1),\ldots, \tau(a_d)\}$.
\end{defn}

Let us comment on some features of this definition; for more details see Kallenberg's monograph \cite{KallenbergMonograph}. First, for $k=1$ the definition is exactly that of an exchangeable sequence of random variables. In higher dimensions it is most common in the literature to index the random variables $\mathsf{X}_J$ by sets $J$ of size exactly $k$ instead of size at most $k$. However, the richer object indexed by sets of sizes at most $k$ is required if one wants to study joint moments. Finally, a more common terminology in the literature is that of a ``jointly-exchangeable array"  indicating that we are using a single permutation $\tau$ to act on the elements of $J$. This is to distinguish from a ``separately-exchangeable array" for which each element is acted on by a different permutation. Since separately-exchangeable arrays will not appear in this paper we have dropped this extra word from the terminology.

Initially our random variables $(\X_J)_{J\in \Tilde{\mathbb{N}}^{(k)}}$ are defined on some abstract probability space $(\Omega, \mathcal{B}, \mathbb{Q})$. It will, however, be more convenient for our proofs to define them instead on the canonical space,
\begin{equation*}
    \tilde{\mathcal{V}}\stackrel{\mathrm{def}}{=} {\mathcal{V}}^{\Tilde{\mathbb{N}}^{(k)}},
\end{equation*}
endowed with the $\sigma$-algebra $\Tilde{\mathfrak{V}}$ generated by all sets of the form
\begin{equation*}
    \prod_{J\in \Tilde{\mathbb{N}}^{(k)}} A_J,
\end{equation*}
where all but finitely many $A_J\in \mathfrak{V}$ equal $\mathcal{V}$, as follows. For $I\in \Nkt$, define the measurable function $\mathsf{X}_I$ from $\left(\tilde{\mathcal{V}},\tilde{\mathfrak{B}}\right)$ to  $\left(\mathcal{V},\mathfrak{B}\right)$ by:
\begin{equation*}
    \mathsf{X}_I\left((v_J)_{J\in \Nkt} \right)=v_I.
\end{equation*}
Then, we consider the probability space $(\Tilde{\mathcal{V}}, \Tilde{\mathfrak{V}}, \mathbb{P})$, with $\mathbb{P}$ the unique probability measure on $\left(\tilde{\mathcal{V}},\tilde{\mathfrak{B}}\right)$ defined via the relation:
\begin{equation*}
    \mathbb{P}\left(A_{J_1}\times A_{J_2} \times \cdots \times A_{J_m}\right) \defeq \mathbb{Q}\left(\left\{\mathsf{X}_{J_1}\in A_{J_1} \right\}\cap \left\{\mathsf{X}_{J_2}\in A_{J_2} \right\}\cap \cdots \cap \left\{\mathsf{X}_{J_m}  \in A_{J_m}\right\}\right),
\end{equation*}
for $J_1,\ldots, J_m \in \Nkt$ and $A_{J_1},\dots,A_{J_m}\in \mathfrak{V}$. 
By construction, for any measurable $B\in \tilde{\mathfrak{V}}$ we have,
\begin{equation*}
\mathbb{P}\left((\X_J)_{J\in \Tilde{\mathbb{N}}^{(k)}}\in B\right)=\mathbb{Q}\left((\X_J)_{J\in \Tilde{\mathbb{N}}^{(k)}}\in B\right).
\end{equation*}
Thus, from now on we assume that we are working on the canonical space $\left(\tilde{\mathcal{V}},\tilde{\mathfrak{V}},\mathbb{P}\right)$. We denote by $\mathbb{E}$, the averages taken with respect to $\mathbb{P}$. 
The following definition is fundamental for our purposes.
\begin{defn}
For each $N\ge 1$, we define the $\sigma$-algebras $\mathcal{E}_N$ on $(\Tilde{\mathcal{V}}, \Tilde{\mathfrak{V}})$ by,
\begin{equation*}
    \mathcal{E}_N= \left\{A\in \Tilde{\mathfrak{V}} : \left(v_J\right)_{J\in \Nkt} \in A \Rightarrow \left(v_{\tau(J)}\right)_{J\in \Nkt}\in A \text{  
 for every   } \tau \in \mathsf{Sym}_N\right\}
\end{equation*}
and hence define the exchangeable $\sigma$-algebra $\mathcal{E}=\bigcap_{N\geq 1} \mathcal{E}_N$.  
  
\end{defn}

We then have the following natural higher-dimensional extension of De-Finetti's Theorem \cite{kallenberg}. Surprisingly, as already mentioned in the introduction, we were not able to locate the precise statement in the literature. 
\begin{thm}\label{exchangeableprop}
    Let $(\X_J)_{J\in \Nkt}$ be an exchangeable array, and let $(J_l)_{l=1}^\infty\subset \Nkt$ be disjoint sets. Then, conditioned on $\mathcal{E}$, $(\X_{J_l})_{l=1}^\infty$ are independent.
\end{thm}
We will prove Theorem \ref{exchangeableprop} in two steps. First, we will show the claimed result when $\mathcal{E}$ is replaced by the tail $\sigma$-algebra $\mathcal{G}$ we define below. Then, we will show that even though $\mathcal{G}\subsetneq \mathcal{E}$, they coincide up to sets of measure zero. For the first step, we need some definitions.

\begin{defn}[Tail $\sigma$-algebra]
Consider the map
\begin{align*}
    \mathcal{S} :  \Nkt &\to \Nkt, \nonumber \\
    \{\gamma_1,\ldots, \gamma_l\}&\mapsto\{\gamma_1+1,\ldots,\gamma_l+1\}, \nonumber
\\ \intertext{and with a slight abuse of notation,  denote} 
 \mathcal{S}^m \left((\X_J)_{J\in \Nkt}\right)&\stackrel{\mathrm{def}}{=}  (\X_{\mathcal{S}^m (J)})_{J\in\Nkt}. \nonumber \\
 \intertext{Then, for each $N\ge 1$, define the auxiliary $\sigma$-algebras:}
\mathcal{G}_N&\stackrel{\mathrm{def}}{=}  \sigma\left( \mathcal{S}^N\left((\X_J)_{J\in \Nkt}\right)\right) \nonumber\\
\intertext{and the associated tail $\sigma$-algebra $\mathcal{G}$ by,}
\mathcal{G}&= \bigcap_{N}  \mathcal{G}_N.
\end{align*}
\end{defn}

 Observe that in proving Theorem \ref{exchangeableprop}, there is no loss of generality in assuming that each $J_l$ is given by
 \begin{equation}\label{specificsubsets}
      J_l=\{j_{l-1}+1,\ldots,j_{l}-1, j_l \},
 \end{equation}
 for some increasing sequence $(j_l)_{l\geq 1}$. Indeed, if one proves the result for this special form of $J_l$'s then using exchangeability, it follows that it holds for any disjoint sequence of $J_l$'s. 
 
 Next, to simplify the presentation, we define the random vector:
 \begin{equation*}
     \Hat{\mathsf{X}}_J \defeq \left(\mathsf{X}_I :I\subset J \right)
 \end{equation*}
 for each $J\in \Nkt$ and also write 
 \begin{equation*}
     \mathcal{L} \underset{\mathcal{K}}{\ind} \mathcal{H},  
 \end{equation*}
whenever the $\sigma$-algebras $\mathcal{L}$ and $\mathcal{H}$ are independent conditioned on the $\sigma$-algebra $\mathcal{K}$. That is to say, whenever $A\in \mathcal{L}, B\in \mathcal{H}$, we have:
\begin{equation*}
    \mathbb{E}\left[\mathbf{1}_A \mathbf{1}_B \mid \mathcal{K}\right]=\mathbb{E}\left[\mathbf{1}_A \mid \mathcal{K}\right]\mathbb{E}\left[\mathbf{1}_B \mid \mathcal{K}\right].
\end{equation*}
We also consider the previously defined map $\mathcal{S}$ as a map on $\Tilde{\mathcal{V}}$, defined by:
\begin{align*}
    \mathcal{S} : \Tilde{\mathcal{V}}  &\to \Tilde{\mathcal{V}}, \nonumber \\
    (v_J)_{J\in \Nkt} &\mapsto (v_{\mathcal{S}(J)})_{J\in \Nkt}.
\end{align*}
\begin{lem}
    Let $(j_l)_{l=1}^\infty$ be a strictly increasing sequence of integers and define the corresponding $J_l$ as in \eqref{specificsubsets}. Then, $(\X_{J_l})_{l=1}^\infty$ are independent conditioned on $ \mathcal{G}$.
\end{lem}
\begin{proof}
    Firstly, note that by exchangeability, we have, for each $l$ and $N\geq j_l$, the distributional equality of the random vectors
    \begin{equation*}
        \left(\Hat{\X}_{J_l}, \mathcal{S}^{j_l}\left((\X
_J)_{J\in \Nkt}\right)\right) \stackrel{\mathrm{d}}{=}\left(\Hat{\X}_{J_l}, \mathcal{S}^N\left((\X_J)_{J\in \Nkt}\right)\right), 
    \end{equation*}
   (with $\overset{\textnormal{d}}{=}$ denoting equality in distribution) which,  by means of Corollary 8.10 in \cite{kallenberg}, gives that
\begin{align*}
            \Hat{\X}_{J_l} &\underset{ \mathcal{G}_N}{\ind} \mathcal{G}_{j_l}.\\
        \intertext{In particular, for every $A\in \sigma(\Hat{\X}_{J_l})$, and all $N\geq j_l$, we have that}
        \mathbb{E}\Bigl[\mathbf{1}_A \mid \mathcal{G}_{j_l}\Bigr]=\mathbb{E}\Bigl[\mathbf{1}_A &\mid \mathcal{G}_N\Bigr]\xrightarrow[\mathrm{a.s.}]{N\to \infty} \mathbb{E}\left[\mathbf{1}_A \mid \mathcal{G}\right],\\
        \intertext{where for the equality we used Theorem 8.9 from \cite{kallenberg}. Since the first term appearing above is independent of $N$, we then have that it is precisely equal to the limit, which then implies that }
        \Hat{\X}_{J_l} &\underset{ \mathcal{G}}{\ind} \mathcal{G}_{j_l}.\\
        \intertext{Hence, applying this iteratively, we get that}
        \sigma(\Hat{\X}_{J_1},\ldots, \Hat{\X}_{J_l})&\underset{ \mathcal{G}}{\ind} \mathcal{G}_{j_{l}},
\end{align*}
   for each $l$, from which the result follows.
\end{proof}
Next, we show that $\mathcal{G}$ coincides with $\mathcal{E}$ up to sets of measure $0$ with respect to $\mathbb{P}$, thus proving Theorem \ref{exchangeableprop}.
\begin{prop}
    Let $A\in \mathcal{E}$. Then, there exists $B\in \mathcal{G}$ such that
\begin{equation*}
    \mathbb{P}\left(A\boldsymbol{\symdif} B\right)=0,
\end{equation*}
where $A\boldsymbol{\symdif} B$ denotes the symmetric difference between the sets $A$ and $B$.
\end{prop}
\begin{proof}
We first  claim that for each $A\in \mathcal{E}$, we have that
 \begin{equation*}
        \mathbb{P}\left(A\boldsymbol{\symdif} \mathcal{S}^{-1}(A)\right)=0.
    \end{equation*}
 To prove the claim, we start by fixing $\delta>0$ and noting that since $A\in \Tilde{\mathfrak{V}}$, by means of Lemma 4.16 in \cite{kallenberg}, there exists $m\in \mathbb{N}$ and a set $W\in \Tilde{\mathfrak{V}}$ of the form
    \begin{equation}\label{wdecomp}
        W=Z\times \mathcal{V}^{K_m}
    \end{equation}
    where $K_m\overset{\textnormal{def}}{=}\{J\in \Nkt: \max(J)>m\}$, such that
    \begin{equation*}
        \mathbb{P}\left(A \boldsymbol{\symdif} W\right)<\delta.
    \end{equation*}
    Here, $Z$ is a Borel subset of $\mathcal{V}^{L_m}$ where $L_m\defeq \Nkt\setminus K_m$. Next, define the map
    \begin{align*}
        \mathcal{P}_m :\Tilde{\mathcal{V}}&\to \Tilde{\mathcal{V}} \\
        (v_J)_{J\in \Nkt}&\mapsto (v_{\tau_m(J)})_{J\in \Nkt},
    \end{align*}
    where $\tau_m\in \mathsf{Sym}_{m+1}$ is the $(m+1)-$cycle:  $(m+1 \;\; m \;\; \cdots \;\; 2 \;\; 1 )$.

   Then, note that $\mathcal{P}_m$ is a bijection such that for any $A\in \mathcal{E}$, $\mathcal{P}_m(A)=A$. In particular, we have
   \begin{equation*}
       \mathcal{P}_m(W)\boldsymbol{\symdif} A= \mathcal{P}_m(W)\boldsymbol{\symdif} \mathcal{P}_m(A)=\mathcal{P}_m(W\boldsymbol{\symdif} A).
   \end{equation*}
   Combining this with the fact that $\mathbb{P}$ is invariant under $\mathcal{P}_m$, we get that
   \begin{equation*}
       \mathbb{P}\left(\mathcal{P}_m(W)\boldsymbol{\symdif} A\right)= \mathbb{P}\left(W\boldsymbol{\symdif} A\right)<\delta.
   \end{equation*}
   Now, note that if $J\in L_m$, then
   \begin{equation*}
       \mathcal{S}(J)=\tau_m^{-1}(J).
   \end{equation*}
   Hence, one easily sees that
   \begin{equation*}
       \mathcal{S}^{-1}(W)= \Big\{(v_J)_{J\in \Nkt} : (v_{\mathcal{S}(J)})_{J\in L_m} \in Z \Big\}=\Big\{(v_J)_{J\in \Nkt} : (v_{\tau_m^{-1}(J)})_{J\in L_m} \in Z \Big\}=\mathcal{P}_m(W)
   \end{equation*}
   Thus, we get
    \begin{equation*}
       \mathbb{P}\left(\mathcal{S}^{-1}(W)\boldsymbol{\symdif} W\right)=\mathbb{P}\left(\mathcal{P}_m(W)\boldsymbol{\symdif} W\right)\leq \mathbb{P}\left(\mathcal{P}_m(W)\boldsymbol{\symdif} A\right)+ \mathbb{P}\left(A\boldsymbol{\symdif} W\right)<2\delta.
    \end{equation*}
   Combining these, and using that $\mathbb{P}(\mathcal{S}^{-1}(E))=\mathbb{P}(E)$ for any $E\in \Tilde{\mathcal{B}}$, we get that
   \begin{equation*}
       \mathbb{P}\left(\mathcal{S}^{-1}(A)\boldsymbol{\symdif} A\right)\leq \mathbb{P}\left(A\boldsymbol{\symdif} W\right)+\mathbb{P}\left(W\boldsymbol{\symdif} \mathcal{S}^{-1}(W)\right)+\mathbb{P}\left(\mathcal{S}^{-1}(W)\boldsymbol{\symdif} \mathcal{S}^{-1}(A)\right)<4\delta,
   \end{equation*}
   which proves the claim since $\delta>0$ was arbitrary. Now, define inductively
   \begin{equation*}
       \mathcal{S}^{-m}(A)\overset{\textnormal{def}}{=}\mathcal{S}^{-1}\left(\mathcal{S}^{-(m-1)}(A)\right),
   \end{equation*}
   and the corresponding set,
   \begin{equation*}
       B\overset{\textnormal{def}}{=}\limsup_{m\to \infty} \mathcal{S}^{-m}(A).
   \end{equation*}
   Note that, $B\in \mathcal{G}$. Also note that,
   \begin{align*}
       \mathbb{P}(A\boldsymbol{\symdif} \mathcal{S}^{-m}(A))\leq\mathbb{P}(A\boldsymbol{\symdif} & \mathcal{S}^{-(m-1)}(A))+ \mathbb{P}(\mathcal{S}^{-(m-1)}(A)\boldsymbol{\symdif} \mathcal{S}^{-m}(A))\\&= \mathbb{P}(A\boldsymbol{\symdif} \mathcal{S}^{-(m-1)}(A))+ \mathbb{P}(A\boldsymbol{\symdif} \mathcal{S}^{-1}(A))=0,
   \end{align*}
   where we used induction in the last step. Hence,
   \begin{equation*}
       \mathbb{P}(A\boldsymbol{\symdif} B)= \mathbb{P}(A\boldsymbol{\symdif} \cap_{n\geq 1}\cup_{m\geq n} \mathcal{S}^{-m}(A))\leq\sum_{m\geq 1} \mathbb{P}(A\boldsymbol{\symdif} \mathcal{S}^{-m}(A))=0,
   \end{equation*}
   which completes the proof.
\end{proof}

From now on we assume that $\mathcal{V}=\mathbb{R}$, namely the random variables $\mathsf{X}_J$ are real-valued.
We use the results above to first prove an analogue of the law of large numbers for statistics of exchangeable arrays (which should not be surpsising), and then give a formula for joint moments of the limiting random variables in terms of finitely many random variables in the array. This general moments formula appears to be new.
\begin{prop}\label{exchangeablelln}
    Let $(\X_J)_{J\in \Nkt}$ be an exchangeable array, and let $\mathcal{E}_N$, $\mathcal{E}$ be defined as before. Then, for each $d\leq k$ and fixed $I\in \mathbb{N}^{(d)}$, we have that
\begin{align*}
       \frac{d!}{N^d} \sum_{J\in [N]^{(d)}} \mathsf{X}_J \xrightarrow[N\to \infty]{\mathrm{a.s.}, \; L^1} \mathbb{E}\left[\mathsf{X}_I \mid \mathcal{E}\right]
\end{align*}
    whenever $\X_I\in L^1$. If we further have $\X_I\in L^p$ for $p>1$ then convergence takes place in $L^p$ as well.
\end{prop}
\begin{proof}
    Define the partial sums, noting that we may change $\frac{N^d}{d!}$ to $\binom{N}{d}$ without affecting the result,
    \begin{equation}\label{partialsum}
        \mathsf{T}_N^{(d)}= \frac{1}{{N\choose d}} \sum_{J\in [N]^{(d)}} \X_J.
    \end{equation}
    Then, note that $\mathsf{T}_N$ is $\mathcal{E}_N$-measurable, and for every $I,L \in [N]^{(d)}$ and $A\in \mathcal{E}_N$,
    \begin{equation*}
\mathbb{E}\left[\mathsf{X}_I\mathbf{1}_A\right]=\mathbb{E}\left[\mathsf{X}_L\mathbf{1}_A\right],
    \end{equation*}
    so that we get
    \begin{equation*}
\mathbb{E}\left[\mathsf{X}_I\mid\mathcal{E}_N\right]=\mathbb{E}\left[\mathsf{X}_L\mid\mathcal{E}_N\right].
    \end{equation*}
    In particular, we have
    \begin{equation*}
\mathsf{T}_N^{(d)}=\mathbb{E}\left[\mathsf{T}_N^{(d)}\mid\mathcal{E}_N\right]=\mathbb{E}[\mathsf{X}_I\mid\mathcal{E}_N].
    \end{equation*}
    Thus, we see that $\mathsf{T}_N^{(d)}$ forms a backwards martingale with respect to $\mathcal{E}_N$. Combining this with the assumed integrability of $\mathsf{X}_I$ and using the backwards martingale convergence theorem, the result follows.
\end{proof}
\begin{prop}\label{exchangeablejointmom}
    Let $(\X_J)_{J\in \Nkt}$ be an exchangeable array, and define, for $d\leq k$, the random variables
    \begin{equation*}
        \mathsf{T}^{(d)}\overset{\textnormal{def}}{=}\lim_{N\to \infty} \mathsf{T}_N^{(d)} 
    \end{equation*}
    where $\mathsf{T}_N^{(d)}$ is defined as in \eqref{partialsum}. Then, if for $r_1,\ldots, r_l>0$ and $k_1,\ldots, k_l\in \{1,\ldots,k\}$,
    \begin{equation*}
\mathbb{E}\left[\left|\mathsf{X}_{[k_j]}\right|^{p}\right]<\infty
    \end{equation*}
    for all $j=1,\ldots,l$, where we define $p\overset{\textnormal{def}}{=}\sum_j r_j$, then we have that
    \begin{equation*}
      \lim_{N\to \infty}\frac{1}{N^{\sum_j r_j k_j}} \mathbb{E}\left[\prod_{j=1}^l \left|\mathsf{T}_N^{(k_l)}\right|^{r_j}\right]=  \mathbb{E}\left[\prod_{j=1}^l \left|\mathsf{T}^{(k_l)}\right|^{r_j}\right].
    \end{equation*} 
\end{prop}
\begin{proof}
    Note that under the given conditions, almost sure convergence holds via Proposition \ref{exchangeablelln}. Thus, to prove the convergence of moments, it suffices to show that the sequence
\begin{equation}
    \Bigg\{\prod_{j=1}^l \left|\mathsf{T}_N^{(k_j)}\right|^{r_j}\Bigg\}_{N\geq \max(k_j)}
    \label{jointUI1}
\end{equation}
is uniformly integrable. To see that this follows from the assumptions, note that first applying Young's inequality, and then applying Holder's inequality with exponents $\frac{p}{r_j}$ in both cases, we get that
\begin{align}
    \sup_{N} \mathbb{E}\left[\prod_{j=1}^l \left|\mathsf{T}_N^{(k_j)}\right|^{r_j} \mathbf{1}\Bigg\{\prod_{j=1}^l \left|\mathsf{T}_N^{(k_j)}\right|^{r_j} 
    \geq R\Bigg\}\right ] \leq \sup_N\sum_{m=1}^l  \mathbb{E}\left[\prod_{j=1}^l \left|\mathsf{T}_N^{(k_j)}\right|^{r_j} \mathbf{1}\Bigg\{\left|\mathsf{T}_N^{(k_m)}\right|^p 
    \geq \frac{pR}{lr_m}\Bigg\}\right ]\nonumber\\ \leq \sup_{N} 
    \sum_{m=1}^l  \mathbb{E}\left[\left|\mathsf{T}_N^{(k_m)}\right|^{p} \mathbf{1}\Bigg\{\left|\mathsf{T}_N^{(k_m)}\right|^p 
    \geq \frac{pR}{lr_m}\Bigg\}\right ]^{r_m/p} \prod_{j\neq m}\mathbb{E}\left[ \left|\mathsf{T}_N^{(k_j)}\right|^{p} \right ]^{r_j/p}.
    \label{younineqbound1}
\end{align}
By Proposition \ref{exchangeablelln}, we know that for each $m$,
\begin{align}
    \mathsf{T}_N^{(k_m)}\xrightarrow[N\to \infty]{L^p} &\mathsf{T}^{(k_m)}\\
    \intertext{and hence, the sequence}
   \left\{\left|\mathsf{T}_N^{(k_m)}\right|^p\right\}&_{N\geq \max(k_j)}
\end{align}
 is uniformly integrable. Thus, taking $R\to \infty$ in \eqref{younineqbound1}, we see that the sequence in \eqref{jointUI1} is also uniformly integrable, concluding the proof.
\end{proof}
Next, we give a seemingly new formula for the joint moments of the random variables $\mathsf{T}^{(d)}$ in terms of joint moments of finitely many $\X_J$'s. Its usefulness will be clear in the sequel.
\begin{prop}\label{arrayfiniteaverageprop}
    Fix $k\in \mathbb{N}$. Let $k_1>\cdots> k_l \in \{1,\ldots,k\}$ and $h_1,\ldots,h_l \in \mathbb{N}$. Define $L=\sum_j h_j k_j$ and $p=\sum_j h_j$. Suppose $\{J_m\}_{m=1}^p$ is a partition of the set $\{1,\ldots,L\}$ such that for each $j$, it contains exactly $h_j$ sets of size $k_j$; and further assume that
    \begin{equation*}
        \mathbb{E}\left[\left|\mathsf{X}_{[k_j]}\right|^p\right]<\infty.
    \end{equation*}
   Then, with the random variables $\mathsf{T}^{(d)}$ defined as in Proposition \ref{exchangeablejointmom}, we have that,
   \begin{equation*}
       \mathbb{E}\left[\prod_{j=1}^l \left(\mathsf{T}^{(k_j)}\right)^{h_j}\right]= \mathbb{E}\left[\prod_{m=1}^p \mathsf{X}_{J_m}\right].
   \end{equation*}
\end{prop}
\begin{proof}
    Note that combining the integrability assumption with the fact that $J_m$ are disjoint, we have that
    \begin{equation*}
        \mathbb{E}\left[\prod_{m=1}^{p}\mathsf{X}_{J_m}\;\Big\vert\;\mathcal{E}\right]=\prod_{m=1}^{p}\mathbb{E}\left[\mathsf{X}_{J_m}\mid \mathcal{E}\right]= \prod_{m=1}^p \mathsf{T}^{(|J_m|)}=\prod_{j=1}^l \left(\mathsf{T}^{(k_j)}\right)^{h_j}.
    \end{equation*}
    Taking the expectation of both sides and using tower law gives the desired result.
\end{proof}
\subsection{Consistent permutation-invariant ensembles}

In this section, we prove a number of new results for consistent permutation-invariant ensembles using the theory of  exchangeable arrays discussed in the previous subsection. We restrict our attention to complex Hermitian matrices because this is what we need for our applications to the problem of joint moments. However, essentially all our arguments go through verbatim if we consider ensembles over the reals or quaternions, or if we drop the self-adjoint constraint, as long as we have permutation-invariance. Before we formally define the class of permutation-invariant random matrix ensembles, let us introduce some notation that will be used throughout.

Let $\mathbb{H}(N)$ denote the space of $N\times N$ Hermitian matrices, equipped with the Borel $\sigma$-algebra $\mathcal{B}(\mathbb{H}(N))$ and the Lebesgue measure,
\begin{equation*}
    \mathrm{d}\mathbf{H}= \prod_{i=1}^N \mathrm{d} \mathbf{H}_{ii} \prod_{i<j}\mathrm{d}\Re \mathbf{H}_{ij}\mathrm{d}\Im \mathbf{H}_{ij}.
\end{equation*}
Define the projection maps,
\begin{align*}
    \mathbf{\Lambda}_{N}^{N+1} : \mathbb{H}(N+1) &\to \mathbb{H}(N),\\
    \left[\mathbf{H}_{ij}\right]_{i,j=1,\dots,N+1 }&\mapsto \left[\mathbf{H}_{ij}\right]_{i,j=1,\dots,N},
\end{align*}
mapping each $\mathbf{H}\in \mathbb{H}(N+1)$ to its top-left $N\times N$ corner. Next, define $\mathbb{H}(\infty)$ to be the projective limit of $\mathbb{H}(N)$ under these maps, endowed with the projective limit topology, whose Borel sigma algebra we denote by $\mathcal{B}(\mathbb{H}(\infty))$. Finally, consider the corresponding natural projections:
\begin{align*}
    \mathbf{\Lambda}_{N}^{\infty} : \mathbb{H}(\infty) &\to \mathbb{H}(N),\\
    \left[\mathbf{H}_{ij}\right]_{i,j\in \mathbb{N}}&\mapsto \left[\mathbf{H}_{ij}\right]_{i,j=1,\dots,N}.
\end{align*}
We have the following definition.
\begin{defn}
    A sequence of probability measures $\{\mu_N\}_{N\geq 1}$, with $\mu_N$ on $(\mathbb{H}(N),\mathcal{B}\left(\mathbb{H}(N)\right))$, is called a consistent permutation-invariant ensemble if:
    \begin{itemize}
        \item For each $N\geq 1$, $\tau \in \mathsf{Sym}_N$, and $B\in \mathcal{B}(\mathbb{H}(N))$, 
        \begin{equation*}
        \mu_N \left(\mathbf{H} \in B\right)=\mu_N \left(\mathbf{P}_\tau \mathbf{H}\mathbf{P}_\tau^* \in B\right),
        \end{equation*}
        where recall that $\mathbf{P}_\tau$ is the permutation matrix associated to $\tau$.
        \item For each $N\geq1$, the pushforward of $\mu_{N+1}$ under $\mathbf{\Lambda}_N^{N+1}$ is $\mu_N$.
    \end{itemize}
 We denote by $\mu_\infty$ the unique probability measure on $(\mathbb{H}(\infty),\mathcal{B}(\mathbb{H}(\infty)))$ such that for all $N\geq 1$, the pushforward image of $\mu_\infty$ under $\mathbf{\Lambda}_N^{\infty}$ is $\mu_N$.
\end{defn}

 Throughout this section, we will be considering arbitrary consistent permutation-invariant ensembles $\{\mu_N\}_{N\ge 1}$, associated with the measure $\mu_\infty$ which couples them. We denote expectation with respect to $\mu_\infty$ by $\mathbb{E}$. We let $\mathbf{H}$ denote an infinite random matrix sampled according to $\mu_\infty$ and also define $\mathbf{H}_N=\mathbf{\Lambda}_N^{\infty}(\mathbf{H})$ which is distributed according to $\mu_N$. We denote the random eigenvalues of $\mathbf{H}_N$ by $\big(\mathsf{x}_1^{(N)}, \mathsf{x}_2^{(N)}, \ldots, \mathsf{x}_N^{(N)}\big)$; how we order the eigenvalues will not be important as we will only consider symmetric functions of them.

 We now prove two general results on such matrices.

\begin{prop}\label{random variable}
    Let $k\in \mathbb{N}$. If
     \begin{equation*}
        \mathbb{E}\left[|\det(\mathbf{H}_k)|\right]<\infty
    \end{equation*}
    then there exists a random variable $\mathsf{T}^{(k)}$ such that
    \begin{equation*}
        \frac{k!}{N^k} \mathrm{e}_{k} \left(\mathsf{x}_1^{(N)}, \mathsf{x}_2^{(N)}, \ldots, \mathsf{x}_N^{(N)}\right)\xrightarrow[]{N\to \infty} \mathsf{T}^{(k)}
    \end{equation*}
    almost surely and in $L^1$, where $\mathrm{e}_k$ denotes the $k$-th elementary symmetric polynomial. If we further have that
    \begin{equation*}
        \mathbb{E}\left[\left|\det\left(\mathbf{H}_k\right)\right|^r\right]<\infty,
    \end{equation*}
    for some $r>1$, then convergence takes place in $L^r$ as well. Furthemore, if 
for $r_1, \ldots, r_l>0$ and $k_1,\ldots, k_l \in \mathbb{N}$, we have 
\begin{equation*}
    \mathbb{E}\left[\left|\det\left(\mathbf{H}_{k_j}\right)\right|^s\right]<\infty
\end{equation*}
for all $1\leq j\leq l$, where $s=\sum r_j$, then
\begin{equation}
    \prod_{j=1}^l \frac{(k_j !)^{r_{j}}}{N^{k_j r_j} }\mathbb{E}\left[\prod_{j=1}^l \left|\mathrm{e}_{k_{j}} \left(\mathsf{x}_1^{(N)}, \mathsf{x}_2^{(N)}, \ldots, \mathsf{x}_N^{(N)}\right)\right|^{r_j}\right]\xrightarrow[]{N\to \infty} \mathbb{E}\left[\prod_{j=1}^l \left|\mathsf{T}^{(k_j)}\right|^{r_j}\right].
    \label{jointintegrabilityassump}
\end{equation}
\end{prop}

\begin{proof}
We start by defining, for each $J\in \Tilde{\mathbb{N}}^{(k)}$, the random variables
\begin{equation*}
    \X_J\overset{\textnormal{def}}{=} \det\left(\mathbf{H}^{J}\right),
\end{equation*}
where if $m_1<\cdots<m_d$ are the elements of $J$, with $d\leq k$, we denote:
\begin{equation*}
    \mathbf{H}^J\overset{\textnormal{def}}{=}\left[\mathbf{H}_{ij}\right]_{i,j=m_1,\ldots, m_d}.
\end{equation*}
Then, note that by permutation invariance of the law of  $\mathbf{H}$, if $\tau$ is any permutation of $\mathbb{N}$ fixing all but finitely many elements, then 
\begin{equation*}
    \mathsf{Law}(\mathbf{H})= \mathsf{Law}(\mathbf{P}_\tau \mathbf{H} \mathbf{P}_\tau^*).
\end{equation*}
Write $\mathbf{H}_\tau=\mathbf{P}_\tau \mathbf{H} \mathbf{P}_\tau^*$. Now, observe that for any $J$ described as above, we have that
\begin{equation*}
\mathbf{H}_\tau^J= \left[\mathbf{H}_{ij}\right]_{i,j=\tau(m_1),\ldots, \tau(m_d) }
\end{equation*}
so that $\mathbf{H}_\tau^J$ and $\mathbf{H}^{\tau(J)}$ are two matrices that can be obtained by one another after a reordering of their indices. In particular, they are equivalent under unitary conjugation and hence have the same determinant. Combining these two observations, we see that $(\X_J)_{J\in \Tilde{\mathbb{N}}^{(k)}}$ is an exchangeable array. 

Next, observe that the elementary symmetric polynomials in the eigenvalues can be rewritten as:
\begin{equation*}
    \frac{d!}{N^d}\mathrm{e}_d\left(\mathsf{x}_1^{(N)}, \mathsf{x}_2^{(N)}, \ldots, \mathsf{x}_N^{(N)}\right)=  \frac{d!}{N^d} \sum_{J\in [N]^{(d)}} \det\left(\mathbf{H}^J\right)=  \frac{d!}{N^d} \sum_{J\in [N]^{(d)}} \X_J.
\end{equation*}
Note that the first equality above easily follows by considering the characteristic polynomial $\det(t\mathbf{I}-\mathbf{H})$. Indeed, by means of Vieta's formulas, one easily sees that the $d$-th symmetric polynomial in eigenvalues equals the coefficient of $t^{N-d}$, which then also equals the sum of the determinants given above that one sees by expanding the characteristic polynomial as a sum over the symmetric group. Combining all of these observations, we see that the result now follows by applications of Propositions \ref{exchangeablelln} and \ref{exchangeablejointmom}.
\end{proof}


\begin{prop}\label{finiteaverageprop2}
    Let $k_1\geq \cdots\geq k_l$, and $r_1,\ldots, r_l \in \mathbb{N}$, and $L=\sum_j k_j r_j$.  Then, under the same assumptions as Proposition \ref{random variable}, \begin{equation*}
\mathbb{E}\left[\prod_{j=1}^l \left(\mathsf{T}^{(k_j)}\right)^{r_j}\right]=\frac{(k_1 !)^{r_1}\cdots(k_l !)^{r_l}}{L!}\sum_{m=k_1}^{L} (-1)^{m+L} {L \choose m} \mathbb{E}\left[\prod_{j=1}^l \left(\mathrm{e}_{k_j} \left(\mathsf{x}_1^{(m)},\mathsf{x}_2^{(m)},\ldots, \mathsf{x}_m^{(m)}\right)\right)^{r_j}\right].
    \end{equation*}
    \label{combinatorialformula2}
\end{prop}  
\begin{proof}
We start by noting that if we let
\begin{align}
   \left\{\mathrm{I}_m^{(j)}\right\}_{j,m}
\end{align}
be a partition of $[L]$ such that for each $j,m$, $\mathrm{I}_m^{(j)}$ contains $k_j$ elements; then by means of Proposition \ref{arrayfiniteaverageprop}, we get that our claim amounts to showing:

 \begin{equation}
        \mathbb{E}\left[\prod_{j,m} \det\left(\mathbf{H}^{\mathrm{I}_m^{(j)}}\right)\right] =\frac{(k_1 !)^{r_1}\cdots(k_l !)^{r_l}}{L!} \sum_{m=k_1}^L  (-1)^{m+L} {L \choose m} \mathbb{E}\left[\prod_{j=1}^l \left(\sum_{J\in [m]^{(k_j)}} \det\left(\mathbf{H}^{J}\right)\right)^{r_j}\right]
        \label{generalfindimaverage2}.
\end{equation}
    To prove the formula above, we will expand the right-hand side of the equation, and use permutation-invariance to note cancellations. With this aim, for each $k\in \{k_1,\ldots, L\}$ define the sets
    \begin{equation*}
        A_k= \Bigg\{\mathcal{J}=\left(J_1^{(1)},\ldots, J_{r_1}^{(1)},J_1^{(2)},\ldots, J_{r_2}^{(2)},\ldots, J_1^{(l)},\ldots, J_{r_l}^{(l)} \right) : J_m^{(j)} \in [L]^{(k_j)}, \left|\bigcup_{j,m} J_m^{(j)}\right|=k\Bigg\}.
    \end{equation*}
    Let the symmetric group $\mathsf{Sym}_L$ act on $A_k$ such that for any $\tau\in \mathsf{Sym}_L$, we have
    \begin{multline*}
        \tau\left(J_1^{(1)},\ldots, J_{r_1}^{(1)},J_1^{(2)},\ldots, J_{r_2}^{(2)},\ldots, J_1^{(l)},\ldots, J_{r_l}^{(l)} \right)\\= \left(\tau\left(J_1^{(1)}\right),\ldots, \tau\left(J_{r_1}^{(1)}\right),\tau\left(J_1^{(2)}\right),\ldots, \tau\left(J_{r_2}^{(2)}\right),\ldots, \tau\left(J_1^{(l)}\right),\ldots, \tau\left(J_{r_l}^{(l)}\right) \right).
    \end{multline*}
    Next, consider the terms appearing on the expansion of the right-hand side of \eqref{generalfindimaverage2} that are given, for any $\mathcal{J}\in \bigcup_{k=k_1}^L A_k$, by:
    \begin{equation*}
       M_{\mathcal{J}}\overset{\textnormal{def}}{=}  \mathbb{E}\left[\det\left(\mathbf{H}^{J_1^{(1)}}\right)\cdots \det\left(\mathbf{H}^{J_{r_1}^{(1)}}\right)\cdots\det\left(\mathbf{H}^{J_{1}^{(l)}}\right)\cdots\det\left(\mathbf{H}^{J_{r_l}^{(l)}}\right)\right].
    \end{equation*}
    Then, note that by permutation-invariance we have that for any $\tau \in \mathsf{Sym}_L$, the average defining $M_\mathcal{J}$ remains unchanged if the random matrix $\mathbf{H}$ is replaced by $\mathbf{H}_\tau= \mathbf{P}_\tau \mathbf{H} \mathbf{P}_{\tau}^*$. In addition, noting that for any $J\subseteq [L]$, $\mathbf{H}_\tau^{J}$ and $\mathbf{H}^{\tau(J)}$ are two matrices that are equal up to re-ordering of the indices, we see that they are equivalent under unitary conjugation and hence have the same determinant. Combining these two observations, we easily see that for any $\tau\in \mathsf{Sym}_L$, we have that $M_\mathcal{J}= M_{\tau(\mathcal{J})}$. Next, with a slight abuse of notation, for any $\mathcal{J}\in \bigcup_{k=k_1}^{L}A_k$, let $\mathsf{Orb}_m(\mathcal{J})$ denote the number of times an average corresponding to the $\mathsf{Sym}_L$-orbit of $\mathcal{J}$ appears on the expansion of the $m$-th summand on the right-hand side of \eqref{generalfindimaverage2}, where $m=k_1,\ldots, L$. Then, realize that if $\mathcal{J}\in A_k$, it follows that the terms of the form $M_{\tau(\mathcal{J})}$ as $\tau$ runs over $\mathsf{Sym}_L$ only appear on the expansion of the $m$-th summand on the right-hand side of \eqref{generalfindimaverage2} if $k\leq m\leq L$, and we have the equality
    \begin{equation*}
        \mathsf{Orb}_m(\mathcal{J})= {m \choose k} \mathsf{Orb}_k(\mathcal{J}).
    \end{equation*}
    Thus, if $k\neq L$, we have that the contribution coming from the $\mathsf{Sym}_L$-orbit of $\mathcal{J}$ equals
    \begin{equation*}
       M_{\mathcal{J}} \mathsf{Orb}_k (\mathcal{J}) \sum_{m=k}^L {L \choose m}  (-1)^m {m\choose k} =0.
    \end{equation*}
    Furthermore, note that if $\mathcal{J}\in A_L$, then 
        $\Big\{J_m^{(j)}\Big\}_{j,m}$ forms a partition of $[L]$. In particular, we get by permutation-invariance that
    \begin{equation*}
        M_{\mathcal{J}}= \mathbb{E}\left[\prod_{j,m} \det\left(\mathbf{H}^{\mathrm{I}_m^{(j)}}\right)\right],
    \end{equation*}
    for all such $\mathcal{J}$. Next, observe that the terms corresponding to $A_L$ only appear in the sum in \eqref{generalfindimaverage2} when $m=L$, in which case the coefficient that multiplies the average is simply $1$. Hence, combining these observations with the simple combinatorial formula
    \begin{equation*}
        \left|A_L\right|= \frac{L!}{(k_1 !)^{r_1}\cdots(k_l!)^{r_l} },
    \end{equation*}
    gives the desired result.
\end{proof}

\begin{rmk}
Proposition \ref{finiteaverageprop2} can be generalised to the setting of arbitrary real-valued exchangeable arrays and the proof is the same. In the final formula one simply replaces $\mathrm{e}_{k_j} \left(\mathsf{x}_1^{(m)},\mathsf{x}_2^{(m)},\ldots, \mathsf{x}_m^{(m)}\right)$ by $\sum_{J\in [m]^{(k_j)}} \mathsf{X}_J$. 
\end{rmk}

\begin{proof}[Proof of Theorem \ref{GeneralCharPolyThm}]
This is a direct consequence of Propositions \ref{random variable} and \ref{finiteaverageprop2} above by virtue of Vieta's formulae.    
\end{proof}

\section{Proofs of convergence and explicit formulae}\label{Convergencesection}
\subsection{Cauchy ensemble and convergence of moments}
The main results of this section will be obtained by using various results from the last section on consistent permutation-invariant ensembles. The specific ensemble that will be used here is called the Cauchy \cite{ForresterBook,ForresterWitteCauchy} (also Hua-Pickrell \cite{Borodin_2001}) random matrix ensemble, which is a sequence of probability measures $\mu_N$ on $\mathbb{H}(N)$ given by
\begin{equation*}
    \mu_N^{(s)}(\mathrm{d}\mathbf{H})\overset{\textnormal{def}}{=} \mathrm{const} \cdot \det\left(\mathbf{I}+\mathbf{H}^2\right)^{-s-N} \mathrm{d}\mathbf{H},
\end{equation*}
where $s \ge 0$, and the explicit expression of the normalization constant will be ommitted as it will not play a role in what follows. Note that, via Weyl's integration formula, one can compute that the eigenvalue distribution of a random matrix $\mathbf{H}$ sampled according to $\mu_N^{(s)}$ is given by ${\bf M}_{N}^{(s)}$ from the introduction as defined in \eqref{hpintrodef}.

\indent The reason these measures are closely related to the Haar measure on $\mathbb{U}(N)$ is that under the transformation $x_j=\cot(\theta_j/2)$, averages over the eigenvalues of a Haar-unitary random matrix can be written as an average against $ {\bf M}_{N}^{(s)}$.
\begin{prop}\label{connections to the symmetric polynomials}
Let $n_{1}> \cdots > n_{k}\geq 0$ be integers. Let $h_{1},\ldots,h_{k} \in (0,\infty)$ and $s=\sum_{j=1}^{k}h_{j}$. Then, we have the following expression for $\mathfrak{F}_{N}^{(n_{1},\ldots,n_{k})}(h_{1},\ldots,h_{k})$,
\beas
\frac{\mathfrak{F}_{N}^{(n_{1},\ldots,n_{k})}(h_{1},\ldots,h_{k})}{\mathfrak{F}_{N}^{(0)}(s)}=
{2^{-\sum_{j=1}^{k}2h_{j}n_{j}}}\mathbb{E}_{N}^{(s)}\left[\prod_{j=1}^{k}\left|\Xi_{n_{j}}\left(\mathsf{x}_{1}^{(N)},\mathsf{x}_2^{(N)},\ldots,\mathsf{x}_{N}^{(N)}\right)\right|^{2h_{j}}\right], 
\eeas
and for $\mathfrak{G}_{N}^{(n_{1},\ldots,n_{k})}(h_{1},\ldots,h_{k})$,
\begin{align*}
 \frac{\mathfrak{G}_{N}^{(n_1,\ldots,n_k)}(h_1,\ldots,h_k)}{{\mathfrak{G}_{N}^{(0)}(s)}} =  {2^{-\sum_{j=1}^{k}2h_{j}n_{j}}}
\mathbb{E}_{N}^{(s)}\left[\prod_{j=1}^{k}\left|\sum_{\substack{m_{j}=0}}^{n_{j}}(-\textnormal{i})^{m_{j}}\widetilde{\Xi}_{n_{j}-m_{j}}\right|^{2h_j}\right],
\end{align*}
where the polynomials $\Xi_{m}$ and $\widetilde{\Xi}_{m}$ are given by,
\begin{align}\label{defofQ}
\Xi_{n_{j}}(x_{1},x_2\ldots,x_{N})&=\sum_{l=0}^{n_{j}}a_{n_{j},l}\mathrm{e}_{l}(x_{1},x_2,\ldots,x_{N}),\\
\widetilde{\Xi}_{n_{j}-m_{j}}(x_1,x_2,\dots,x_N)&=N^{m_{j}}\binom{n_{j}}{m_{j}}\Xi_{n_{j}-m_{j}}(x_{1},x_2,\dots,x_{N}),\nonumber
\end{align}
with the constants $a_{n,l}$ given as follows,
\be\label{defofa}
a_{n,l}\defeq(-1)^{\frac{n+l}{2}}\sum_{\substack{\sum_{i=1}^{N}m_{i}=n\\\textnormal{only}~m_{1},m_{2},\ldots,m_{l}~\textnormal{are}~\textnormal{odd~integers}}}\binom{n}{m_{1}, \ldots, m_{N}}.
\ee
\end{prop}
\begin{proof}
By the definition of $Z_{\mathbf{A}}(\theta)$, we have
\be\label{expression of ZA}
Z_{\mathbf{A}}(\theta)=(-2)^{N}\prod_{n=1}^{N}\sin\left(\frac{\theta-\theta_{n}}{2}\right).
\ee
Note that, for $l\geq 1$,
\begin{equation*}
\frac{\mathrm{d}^{l}}{\mathrm{d}\theta^{l}}\sin\left(\frac{\theta-\theta_{n}}{2}\right)=2^{-l}(-1)^{\frac{l-1}{2}}\sin\left(\frac{\theta-\theta_{n}}{2}\right)\cot\left(\frac{\theta-\theta_{n}}{2}\right)
\end{equation*}
when $l$ is an odd integer; and 
\begin{equation*}
    \frac{\mathrm{d}^{l}}{\mathrm{d}\theta^{l}}\sin\left(\frac{\theta-\theta_{n}}{2}\right)=2^{-l}(-1)^{\frac{l}{2}}\sin\left(\frac{\theta-\theta_{n}}{2}\right)
\end{equation*}
when $l$ is an even integer. So we can compute:
\begin{align}
 \frac{\mathrm{d}^{(n_{j})}Z_{\mathbf{A}}(\theta)}{\mathrm{d}\theta^{n_{j}}}\Big|_{\theta=0}&=(-2)^{N}\sum_{m_{1}+\cdots+m_{N}=n_{j}}\binom{n_{j}}{m_{1}, \ldots, m_{N}}\prod_{n=1}^{N}\frac{\mathrm{d}^{m_{n}}\sin(\frac{\theta-\theta_{n}}{2})}{\mathrm{d}\theta^{m_{n}}}\Big|_{\theta=0}\nonumber\\
&=\frac{Z_{\mathbf{A}}(0)}{2^{n_{j}}}\sum_{m_{1}+\cdots+m_{N}=n_{j}}\binom{n_{j}}{m_{1}, \dots, m_{N}}\prod_{\substack{n=1\\m_{n}~\text{is~even}}}^{N}(-1)^{\frac{m_{n}}{2}}\prod_{\substack{n=1\\m_{n}~\text{is~odd}}}^{N}(-1)^{\frac{m_{n}+1}{2}}\cot\left(\frac{\theta_{n}}{2}\right).\label{1019formula1}
\end{align}
Note that, if $m_1+\cdots+m_N=n_j$, then $\#\{i:m_{i}~\text{is~odd} \}\leq n_{j}$. Hence, the above can be simplified to 

\be
\frac{\mathrm{d}^{(n_{j})}Z_{\mathbf{A}}(\theta)}{\mathrm{d}\theta^{n_{j}}}\Big|_{\theta=0}=\frac{Z_{\mathbf{A}}(0)}{2^{n_{j}}}
\sum_{l=0}^{n_{j}}\sum_{\substack{i_{1},\dots,i_{l}=1\\i_{1}<\cdots<i_{l}}}^{N}c_{i_{1},\dots,i_{l}}\cot\left(\frac{\theta_{i_{1}}}{2}\right)\cdots\cot\left(\frac{\theta_{i_{l}}}{2}\right),\label{formula2}
\ee
where the constants $c_{i_{1},\ldots,i_{l}}$ are given by,
\be
c_{i_{1},\ldots,i_{l}}=
(-1)^{\frac{n_{j}+l}{2}}\sum_{\substack{\sum_{i=1}^{N}m_{i}=n_{j}\\\text{only}~m_{i_{1}},m_{i_{2}},\ldots,m_{i_{l}}~\text{are}~\text{odd~integers}}}\binom{n_{j}}{m_{1}, \ldots, m_{N}}.\ \label{formula3}
\ee
But then one immediately sees that $c_{i_{1},\ldots,i_{l}}=a_{n_{j},l}$ for any $i_1<\cdots<i_l$ so that we have,

\begin{align*}\label{formula4}
\frac{\mathrm{d}^{(n_{j})}Z_{\mathbf{A}}(\theta)}{\mathrm{d}\theta^{n_{j}}}\Big|_{\theta=0}&=\frac{Z_{\mathbf{A}}(0)}{2^{n_{j}}}\sum_{l=0}^{n_{j}}a_{n_{j},l}\mathrm{e}_{l}\Big(\cot\left(\frac{\theta_{1}}{2}\right),\ldots,\cot\left(\frac{\theta_{N}}{2}\right)\Big) \\ &= \frac{Z_{\mathbf{A}}(0)}{2^{n_j}} \Xi_{n_j}\left(\cot\left(\frac{\theta_1}{2}\right),\ldots, \cot\left(\frac{\theta_N}{2}\right)\right).
\end{align*}
Next, by applying Weyl's integration formula for the Haar measure on $\mathbb{U}(N)$, we see that
\begin{align*}
    \mathfrak{F}_{N}^{(n_{1},\ldots,n_{k})}(h_{1},\ldots,h_{k})
&=\int_{\U(N)}\prod_{j=1}^{k}\big|Z_{\mathbf{A}}^{(n_{j})}(0)\big|^{2h_{j}}\mathrm{d}\mu_{\mathrm{Haar}}(\mathbf{A})
\\ &=\frac{1}{2^{\sum_{j=1}^{k}2h_{j}n_{j}}(2\pi)^N N!} \int_{[0,2\pi]^N}  \prod_{1\leq j < m \leq N} |e^{\i \theta_j} - e^{\i \theta_m}|^2 \prod_{n=1}^N|1-e^{\i \theta_n}|^{2s} 
\\&\prod_{j=1}^{k}\Bigg|\Xi_{n_j}\left(\cot\left(\frac{\theta_1}{2}\right),\ldots, \cot\left(\frac{\theta_N}{2}\right)\right)\Bigg|^{2h_{j}}
\mathrm{d}\theta_1 \cdots \mathrm{d}\theta_N.
\end{align*}
Then, applying the transformation $x_j=\cot(\theta_j/2)$ gives that
\beas
    \mathfrak{F}_{N}^{(n_{1},\ldots,n_{k})}(h_1,\dots,h_k)=
\frac{\mathrm{K}_N(s)}{2^{\sum_{j=1}^{k}2h_{j}n_{j}}}
\mathbb{E}_{N}^{(s)}\Bigg[\prod_{j=1}^{k}\Bigg|\Xi_{n_j}\left(\mathsf{x}_{1}^{(N)},\mathsf{x}_2^{(N)},\ldots,\mathsf{x}_{N}^{(N)}\right)\Bigg|^{2h_{j}}\Bigg], 
\eeas
for some constant $\mathrm{K}_N(s)$ that is independent of $h_1,\ldots, h_k$. Plugging in $h_j=0$ for all $j$ to both sides of the equation now gives the desired result for $\mathfrak{F}_N^{(n_1,\ldots,n_k)}(h_1,\ldots,h_k)$. To get the second part, simply note that
\begin{equation*}
    V_{\mathbf{A}}^{(n_j)}(0)=e^{-\textnormal{i} N \frac{\pi}{2}+\textnormal{i}\sum_{n=1}^N \frac{\theta_n}{2}} 
\sum_{m_{j}=0}^{n_{j}}\binom{n_{j}}{m_{j}} 
\left(-\frac{\textnormal{i} N}{2}\right)^{m_{j}}\frac{\mathrm{d}^{(n_{j}-m_{j})}Z_{\mathbf{A}}(\theta)}{\mathrm{d}\theta^{n_{j}-m_{j}}}\Bigg|_{\theta=0},
\end{equation*}
which gives the desired conclusion.
\end{proof}

We can now prove Theorem \ref{convergence theorem}.

\begin{proof}[Proof of Theorem \ref{convergence theorem}]
The fact that $\{\mu_N^{(s)}\}_{N\geq 1}$ forms a consistent permutation-invariant (in fact it is unitarily invariant) ensemble was shown in Proposition 3.1 in \cite{Borodin_2001}. We write $\mu_\infty^{(s)}$ for the corresponding measure coupling them on infinite matrices, with respect to which all almost-sure statements below are understood. Now, using the fact that
\begin{equation*}
\mathbb{E}^{(s)}\left[\left|\det\left(\mathbf{H}_{n_j}\right)\right|^{2s}\right]= \mathbb{E}_{n_j}^{(s)} \left[\prod_{m=1}^{n_j} \left|\mathsf{x}_m^{(n_j)}\right|^{2s}\right]<\infty,
\end{equation*}
where $\mathbb{E}^{(s)}$ denotes averages taken with respect to $\mu_\infty^{(s)}$, 
we obtain, by means of Proposition \ref{random variable}, that there exist random variables $\{\widetilde{\mathsf{Y}}_n(s)\}_{n\geq 0}$ with $\widetilde{\mathsf{Y}}_0(s)=1$, such that
\begin{equation*}
\frac{\Xi_{n_j}\left(\mathsf{x}_{1}^{(N)},\mathsf{x}_2^{(N)},\ldots,\mathsf{x}_{N}^{(N)}\right)}{N^{n_j}}\xrightarrow{N\to \infty} \widetilde{\mathsf{Y}}_{n_j}(s)
\end{equation*}
in $L^{2s}\big(\mu_\infty^{(s)}\big)$ and also almost surely. Here, we crucially use the fact that
\begin{equation*}
    a_{n_j,l}=
    \begin{cases}
       O\left(N^{\frac{n_j-l}{2}}\right), & \text{  if  }  n_j-l \text{  
 is even},\\
        0,  & \text{    otherwise},
 
    \end{cases}
\end{equation*}
and $a_{n_j,n_j}=(-1)^{n_j} (n_j)!$. Hence, by combining Proposition \ref{connections to the symmetric polynomials} with the asymptotic equivalence, see for example \cite{keatingsnaith},
\begin{equation*}
     \mathfrak{G}_{N}^{(0)} (s)=\mathfrak{F}_{N}^{(0)} (s) \sim N^{s^2}\frac{G(s+1)^2}{G(2s+1)}
\end{equation*}
and arguing as in Proposition \ref{exchangeablejointmom}, we obtain convergence of all joint moments 
    \begin{align*}
        \lim_{N\to \infty} \frac{\mathfrak{F}_{N}^{(n_1,\ldots,n_k)}(h_1,\ldots,h_k)}{N^{s^2+2\sum_{j=1}^k n_j h_j} }  &= \frac{G(s+1)^2}{G(2s+1)} 2^{-2\sum_{j=1}^k h_j n_j  } \mathbb{E}\left[\prod_{j=1}^k\left|\widetilde{\mathsf{Y}}_{n_j}(s)\right|^{2h_j}\right],
        \label{mainthmeq}\\
       \lim_{N\to \infty} \frac{\mathfrak{G}_{N}^{(n_1,\ldots,n_k)}(h_1,\ldots,h_k)}{N^{s^2+2\sum_{j=1}^k n_j h_j} }  &= \frac{G(s+1)^2}{G(2s+1)} 2^{-2\sum_{j=1}^k h_j n_j  }\mathbb{E}\left[\prod_{j=1}^k \Bigg|\sum_{m_{j}=0}^{n_{j}}(-\textnormal{i})^{m_j} \binom{n_{j}}{m_{j}} \widetilde{\mathsf{Y}}_{n_j-m_j}(s)\Bigg|^{2h_j}\right].
    \end{align*}
It only remains to identify the random variables $\widetilde{\mathsf{Y}}_j(s)$ with the random variables $\mathsf{Y}_j(s)$ from Definition \ref{DefinitionRandomVariables}. For this we make use of some results from the integrable probability literature. Observe that, for all $j\ge 1$, we have the almost sure limits,
\begin{equation}\label{question1}
\widetilde{\mathsf{Y}}_j(s)=\lim_{N\to \infty}\frac{(-1)^{j}j!\mathrm{e}_j\left(\mathsf{x}_1^{(N)}, \mathsf{x}_2^{(N)}, \ldots, \mathsf{x}_N^{(N)}\right)}{N^j}.
\end{equation}
It will then suffice, by the definition of the $\mathsf{Y}_j(s)$ and Newton's identities, to show that, for any $M\ge 1$, we have the convergence in distribution,
\begin{equation}\label{powersumsidentification}
\left(N^{-j}\mathrm{p}_j\left(\mathsf{x}_1^{(N)}, \mathsf{x}_2^{(N)}, \ldots, \mathsf{x}_N^{(N)}\right)\right)_{j=1}^M \overset{\textnormal{d}}{\longrightarrow} \left(\mathsf{q}_j(s)\right)_{j=1}^M.
\end{equation}
It is known, by combining the results of \cite{Borodin_2001,Qiu}, that, with $\left(\mathsf{x}_1^{(N)},\dots,\mathsf{x}_N^{(N)}\right)$ distributed according to ${\bf M}_{N}^{(s)}$, and with the convention ${\mathsf{x}}_i^{(N)}\equiv 0$ for $i<1$ or $i>N$,
\begin{align}
\left(\left(N^{-1}\max\left\{\mathsf{x}_{i}^{(N)},0\right\}\right)_{i=1}^\infty,\left(N^{-1}\max\left\{-\mathsf{x}_{N+1-i}^{(N)},0\right\}\right)_{i=1}^\infty,N^{-1}\sum_{i=1}^N \mathsf{x}_i^{(N)},N^{-2}\sum_{i=1}^N \left(\mathsf{x}_i^{(N)}\right)^2\right) \nonumber\\ \overset{\textnormal{d}}{\longrightarrow}\left(\left(\alpha_i^+\right)_{i=1}^\infty,\left(\alpha_i^-\right)_{i=1}^\infty,\lim_{k\to \infty}\sum_{i=1}^\infty \mathbf{1}_{\left|\alpha_i^\pm\right| \ge k^{-2}}\left(\alpha_i^+-\alpha_i^-\right),\sum_{i=1}^\infty \left[\left(\alpha_i^+\right)^2+\left(\alpha_i^-\right)^2\right]\right),\label{ConvergenceOnOmega}
\end{align}
(convergence in the product-space in (\ref{ConvergenceOnOmega}) is coordinate-wise) where $\left\{\alpha_i^+\right\}\sqcup \left\{-\alpha_i^-\right\}\overset{\textnormal{d}}{=}\mathbf{P}^{(s)}$.  Note that, this completely determines the law of the right hand side in (\ref{ConvergenceOnOmega}), as the $\left(\alpha_i^+\right)_{i=1}^\infty$, $\left(\alpha_i^-\right)_{i=1}^\infty$ are known to be strictly ordered, see \cite{Qiu}. Finally, by virtue of (\ref{ConvergenceOnOmega}), along with Proposition 2.3 of \cite{assiotis2020boundary}, which then implies that, for $k\ge 3$, $N^{-k}\sum_{i=1}^N\left(\mathsf{x}_i^{(N)}\right)^k$ converges to $\sum_{i=1}^\infty \left(\alpha_i^+\right)^k+\sum_{i=1}^\infty\left(-\alpha_i^-\right)^k$. We obtain (\ref{powersumsidentification}) which concludes the proof.
\end{proof}



\subsection{Explicit formulae}
Next, we apply the combinatorial formulae proved in Section \ref{consistentsection} for the even-integer moments of the limiting random variables to our current setting. Via this representation, we further deduce that our random variables $\{\mathsf{Y}_n(s)\}_{n\geq 1}$ are non-trivial, which implies that Theorem \ref{convergence theorem} not only gives an upper bound, but rather gives a sharp order of growth of the joint moments, as expected.
\begin{proof}[Proof of Theorem \ref{mainresult2}]
    The first part of the result is a direct application of Proposition \ref{finiteaverageprop2}. The last claim immediately follows from the formula combined with the fact that for any $k$ and non-negative integers $r_1, r_2,\dots, r_k$, 
    \begin{equation*}
        s\mapsto \mathbb{E}_k^{(s)}\left[  \sum_{\tau\in \mathsf{Sym}_k} \left(\mathsf{x}_{\tau(1)}^{(k)}\right)^{r_1}\left(\mathsf{x}^{(k)}_{\tau(2)}\right)^{r_2} \cdots \left(\mathsf{x}^{(k)}_{\tau(k)}\right)^{r_k}\right]
    \end{equation*}
    is a rational function, which can be shown by an elementary computation.
\end{proof}
The rationality claim of Theorem \ref{mainresult2} turns out to be very useful in proving Corollary \ref{2ndmomentcor}. The strategy to prove the result above will be as follows: we will initially restrict to $s\in \mathbb{N}$ and compute the averages in \eqref{2ndmomentexplicit} by calculating the corresponding integrals over $\mathbb{U}(N)$ and taking $N\to \infty$. We will then show that the expressions we obtain extend to rational functions, which, when combined with Theorem \ref{mainresult2}, will give us two rational functions that agree on integers, which then implies they must agree everywhere. We will then show that the explicit expressions that appear in \eqref{2ndmomentexplicit} are strictly positive, and conclude the final claim of the corollary. \\
\indent To complete the aforementioned first step, we state and prove a sequence of lemmas regarding certain averages of eigenvalues of Haar-unitary random matrices. Recall that via Weyl's integration formula, it is known that if $\lambda_1,\ldots,\lambda_N$ denote the eigenvalues of a Haar-unitary random matrix, ordered so that 
\begin{equation*}
    \lambda_j= e^{\textnormal{i}\theta_j}\;\;\; \textnormal{  with  }\;\;\;\theta_j\in [0,2\pi), \;\;\;  \theta_1\leq \cdots\leq \theta_N, 
\end{equation*}
then the law of the $\theta_j$'s is given by:
\begin{equation*}
    \mathrm{d}\nu_{\mathrm{Haar}}(\theta_1,\ldots,\theta_N)\defeq \frac{1}{(2\pi)^N} \prod_{1\leq j<k\leq N} \left|e^{\textnormal{i}\theta_j}- e^{\textnormal{i}\theta_k}\right|^2 \mathrm{d}\theta_1\cdots \mathrm{d}\theta_N.
\end{equation*}
Thus, for the rest of this section, we denote 
\begin{equation*}
\mathbb{E}_{\haar}\left[f(\lambda_1,\ldots, \lambda_N)\right] \stackrel{\mathrm{def}}{=} \int f\left(e^{\textnormal{i}\theta_1},\ldots, e^{\textnormal{i}\theta_N}\right) \mathrm{d}\nu_{\mathrm{Haar}}(\theta_1,\ldots, \theta_N).
\end{equation*}
\begin{lem}
Let $s, n\in \mathbb{N}$. Then, as $N\to \infty$, we have
\label{serieslem1}

\begin{align}\label{contourint1}
{\mathfrak{F}_{N}^{(n,0)}(1,s-1)}
&\sim \frac{\partial^n}{\partial z_1^n} 
\frac{\partial^n}{\partial w_1^n} \E_{\haar}\left[\prod_{i=1}^{s}(z_{i}w_{i})^{-\frac{N}{2}}\prod_{j=1}^N \prod_{i=1}^s (\lambda_j - z_i) (\bar{\lambda}_j - w_i) \right] \Bigg|_{\substack{z_1=\cdots=z_s=1 \\ w_1=\cdots=w_s=1}},\\
{\mathfrak{G}_{N}^{(n,0)}(1,s-1)}
&\sim \frac{\partial^n}{\partial z_1^n} 
\frac{\partial^n}{\partial w_1^n} \E_{\haar}\left[\prod_{j=1}^N \prod_{i=1}^s (\lambda_j - z_i) (\bar{\lambda}_j - w_i) \right] \Bigg|_{\substack{z_1=\cdots=z_s=1 \\ w_1=\cdots=w_s=1}},
\end{align}
where $\lambda_{1},\ldots,\lambda_{N}$ are the random eigenvalues of a Haar-distributed $\mathbf{A}\in \U(N)$.
\end{lem}
\begin{proof}
Let $\mathbf{A}\in \mathbb{U}(N)$ and let $(e^{\textnormal{i} \theta_1},\ldots, e^{\textnormal{i} \theta_N})=(\lambda_1,\ldots, \lambda_N)$ denote its eigenvalues. Then, define the characteristic polynomial, this time as a function on the unit circle instead of $[0,2\pi)$ by:
\begin{equation*}
    \widetilde{V}_{\mathbf{A}}(z)\overset{\textnormal{def}}{=}\prod_{j=1}^{N}(1-ze^{-\textnormal{i}\theta_{j}})= \prod_{j=1}^N (1-z\bar{\lambda}_j) = \prod_{j=1}^N \bar{\lambda}_j  (\lambda_j -z)
\end{equation*}
and its corresponding rescaled version:
\begin{equation*}
    \widetilde{Z}_{\mathbf{A}}(z)=e^{-\textnormal{i} N \frac{\pi}{2}+\textnormal{i}\sum_{n=1}^N \frac{\theta_n}{2}}z^{-\frac{N}{2}}\widetilde{V}_{\mathbf{A}}(z).
\end{equation*}
Then, direct computation gives that:
\[\widetilde{Z}_{\mathbf{A}^*}^{(n)}(1)=(-1)^{N} \overline{\widetilde{Z}_{\mathbf{A}}^{(n)}(1)}.\]
In particular, one may compute,
\begin{align*}
    &\frac{\partial^n}{\partial z_1^n} \frac{\partial^n}{\partial w_1^n} \prod_{i=1}^{s}(z_{i}w_{i})^{-\frac{N}{2}}\prod_{j=1}^{N}\prod_{i=1}^{s}(\lambda_{j}-z_{i})(\bar{\lambda}_j-w_{i})
\Big|_{\substack{z_1=\cdots=z_s=1\\ w_1=\cdots=w_s=1 }}\\  &= \frac{\partial^n}{\partial z_1^n} 
 z_1^{-\frac{N}{2}} \prod_{j=1}^N (\lambda_j-z_1)\frac{\partial^n}{\partial w_1^n} 
 w_1^{-\frac{N}{2}}\prod_{j=1}^N (\bar{\lambda}_j-w_1)\prod_{i=2}^{s}(z_{i}w_{i})^{-\frac{N}{2}}\prod_{j=1}^{N}\prod_{i=2}^{s}(\lambda_{j}-z_{i})(\bar{\lambda}_j-w_{i})
\Big|_{\substack{z_1=\cdots=z_s=1\\ w_1=\cdots=w_s=1 }}\\
&=  (-1)^{Ns} \widetilde{Z}_{\mathbf{A}^*}^{(n)}(1) \widetilde{Z}_\mathbf{A}^{(n)}(1)  \left(\widetilde{Z}_{\mathbf{A}^*}(1)\right)^{s-1} \left(\widetilde{Z}_\mathbf{A}(1)\right)^{s-1} = \left|\widetilde{Z}_\mathbf{A}^{(n)}(1)\right|^{2} \left|\widetilde{Z}_\mathbf{A}(1)\right|^{2s-2}.
\end{align*}
In particular, taking the average of both sides and moving the differentiation outside the integral (since all derivatives are uniformly bounded as the $\lambda_j's$ range over the unit circle): 
\begin{align*}
    \frac{\partial^n}{\partial z_1^n} \frac{\partial^n}{\partial w_1^n} \E_{\mathbb{U}(N)} \Bigg[\prod_{i=1}^{s}(z_{i}w_{i})^{-\frac{N}{2}}\prod_{j=1}^{N}\prod_{i=1}^{s}(\lambda_{j}-z_{i})(\bar{\lambda}_{j}-&w_{i})\Bigg]
\bigg|_{\substack{z_1=\cdots=z_s=1\\ w_1=\cdots=w_s=1 }}\\ &=
\int_{\mathbb{U}(N)} \left|\widetilde{Z}_{\mathbf{A}}^{(n)}(1)\right|^{2} \left|\widetilde{Z}_{\mathbf{A}}(1)\right|^{2s-2}\mathrm{d}\mu_{\mathrm{Haar}}(\mathbf{A}).
\end{align*}
Now, the leading coefficient in the asymptotics of 
\begin{equation*}
 \int_{\mathbb{U}(N)} \left|\widetilde{Z}_{\mathbf{A}}^{(n)}(1)\right|^{2} \left|\widetilde{Z}_{\mathbf{A}}(1)\right|^{2s-2}\mathrm{d}\mu_{\mathrm{Haar}}(\mathbf{A})  \ \textnormal{  and  } \  \mathfrak{F}_{N}^{(n,0)}(1,s-1),
\end{equation*}
are the same, since plugging in $z=e^{\textnormal{i}\theta}$ and differentiating with respect to $\theta$ would only add terms of lower order derivatives, which, by Theorem \ref{mainthm} has strictly lower order growth. This completes the proof of the first part of the lemma. The proof of the second part follows similarly.
\end{proof}

In the following lemmas, partly inspired by computations in the physics literature \cite{AkemannVernizzi}, we evaluate the right-hand side of \eqref{contourint1}.
\begin{lem}
\label{lem:0820}
For every $s\in \mathbb{N}$, we have
\begin{equation*}
    \E_{\haar}\left[\prod_{j=1}^N \prod_{i=1}^s (\lambda_j - z_i) (\bar{\lambda}_j - \overline{w}_i ) \right] =\frac{\det_{1\leq i,j\leq s}\Big[\sum_{l=0}^{N+s-1} (z_i \overline{w}_j)^l\Big]}{\prod_{1\leq i<j \leq s}(z_j-z_i) 
\prod_{1\leq i<j \leq s}(\overline{w}_j - \overline{w}_i)},
\end{equation*}
where $\lambda_{1},\ldots,\lambda_{N}$ are the random eigenvalues of a Haar-distributed $\mathbf{A}\in \U(N)$.
\label{explicitcontourint}
\end{lem}

\begin{proof}
Define $\zeta_1,\ldots,\zeta_s=z_1,\ldots,z_s$ and $\zeta_{s+1},\ldots,\zeta_{s+N} = e^{\textnormal{i}\theta_1}, \ldots, e^{\textnormal{i}\theta_N}$; $\eta_1,\ldots,\eta_s=w_1,\ldots,w_s$ and $\eta_{s+1},\ldots,\eta_{s+N} = e^{\textnormal{i}\theta_1}, \ldots, e^{\textnormal{i}\theta_N}$. Then, via Weyl's integration formula, we get that the given average equals:
\begin{align*}
&\frac{1}{(2\pi)^N N!} \int_{[0,2\pi]^N} \prod_{j=1}^N \left(\prod_{m=1}^s (z_m - e^{i \theta_j}) \right)
\left(\prod_{m=1}^s (\overline{w}_m - e^{-\textnormal{i} \theta_j})\right) \Delta(e^{\textnormal{i}\theta_1},\ldots,e^{\textnormal{i}\theta_N})
\Delta(e^{-\textnormal{i}\theta_1},\ldots,e^{-\textnormal{i}\theta_N})
\mathrm{d}\boldsymbol{\theta} \\
&= \frac{1}{(2\pi)^N N!\Delta(z_1,\ldots,z_s)
\Delta(\overline{w}_1,\ldots,\overline{w}_s)} 
\int_{[0,2\pi]^N} 
\det_{1\leq i,j\leq N+s} [\zeta_i^{j-1}]
\det_{1\leq i,j\leq N+s} [\bar{\eta}_i^{j-1}]
\mathrm{d}\boldsymbol{\theta}&  \\
&= \frac{1}{(2\pi)^N N!\Delta(z_1,\ldots,z_s)
\Delta(\overline{w}_1,\ldots,\overline{w}_s)} 
\int_{[0,2\pi]^N} 
\det_{1\leq i,j\leq N+s} \left[\sum_{r=1}^{N+s}\zeta_i^{r-1} \bar{\eta}_j^{r-1} \right]
\mathrm{d}\boldsymbol{\theta}.
\end{align*}
Next, observe that if we denote,
\begin{equation*}
    K_{n}(z,\overline{w})=\sum_{l=0}^{n-1}z^{l}\overline{w}^{l},
\end{equation*}
then, for each $n$, we have that 
\begin{equation*}
    \frac{1}{2\pi}\int_{0}^{2\pi}K_{n}(z,e^{-\textnormal{i}\theta})K_{n}(e^{\textnormal{i}\theta},\overline{w})\mathrm{d}\theta=K_{n}(z,\overline{w}).
\end{equation*}
Then, arguing analogously to the proof of Gaudin's lemma (see, for instance, \cite[Proposition 5.1.2]{loggases}), we get that for each $m\in \{s+1,\ldots, N+s\}$:
\begin{equation*}
   \frac{1}{2\pi} \int_{0}^{2\pi} \det_{1\leq i,j\leq m} \left[K_{N+s}(\zeta_i, \bar{\eta}_j)\right] \mathrm{d} \theta_{m-s}= (N+s+1-m) \det_{1\leq i,j\leq m-1} \left[K_{N+s}(\zeta_i, \bar{\eta}_j)\right].
\end{equation*}
Applying this iteratively we get the desired result.
\end{proof}

\begin{lem}
For $s\in\mathbb{N}, m,n \in \mathbb{N}\cup\{0\}$, with $\lambda_{1},\ldots,\lambda_{N}$ the random eigenvalues of a Haar-distributed $\mathbf{A}\in \U(N)$, we have,
\label{serieslem2}
\begin{align*}
&\frac{\partial^n}{\partial z_1^n} 
\frac{\partial^m}{\partial w_1^m} 
\E_{\haar}\left[\prod_{j=1}^N \prod_{i=1}^s (\lambda_j - z_i) (\bar{\lambda}_j - w_i) \right] \Bigg|_{\substack{z_1=\cdots=z_s=
w_1=\cdots=w_s=1}} \;\;\;\;\;\;\;\;\;\;\;\;\;\;\;\;\;\;\;\;\;\;\;\; \;\;\;\;\;\;\;\;\;\;\;\\
&=n!m!\frac{1}{(s-1+n)!}\frac{1}{(s-1+m)!}\prod_{j=2}^{s}\frac{1}{((s-j)!)^2} N^{s^2+n+m} \det_{1\leq i,j\leq s}\left[\frac{1}{p_i+q_j+1}\right] 
+ O(N^{s^2+n+m-1}),
\end{align*}
where $p_{1}=s-1+n,q_1=s-1+m$, $p_{i}=q_{i}=s-i$, $i=2,\ldots,s$.
\end{lem}

\begin{proof}
By Lemma \ref{explicitcontourint}, we have
\begin{align*}
    \E_{\haar}\left[\prod_{j=1}^N \prod_{i=1}^s (\lambda_j - z_i) (\bar{\lambda}_j - w_i) \right] 
=\frac{\det_{1\leq i,j\leq s}\left[ \sum_{l=0}^{N+s-1}(z_iw_j)^l \right]}{\prod_{1\leq i<j\leq s} (z_j-z_i) \prod_{1\leq i<j \leq s} (w_j-w_i)}.
\end{align*}
Thus, letting $z_s=w_s=1$, we then get:
\begin{align*}
   &\lim_{w_{s-1}\rightarrow1} \lim_{z_{s-1}\rightarrow1} 
\frac{\det_{1\leq i,j\leq s}\left[ \sum_{l=0}^{N+s-1}(z_iw_j)^l\right] }{\prod_{1\leq i<j\leq s-1} (z_j-z_i) \prod_{1\leq i<j \leq s-1} (w_j-w_i) \prod_{1\leq i \leq s-1}(1-z_i)(1-w_i)}\\ &=\lim_{w_{s-1}\rightarrow1} \lim_{z_{s-1}\rightarrow1} 
\frac{\det_{1\leq i,j\leq s}\left[ \sum_{l=0}^{N+s-1}(z_iw_j)^l\right]}{(1-z_{s-1}) (1-w_{s-1}) \prod_{1\leq i<j\leq s-2} (z_j-z_i)  (w_j-w_i) \prod_{1\leq i \leq s-2}(1-z_i)^2(1-w_i)^2 } \\
&= \frac{1 }{  \prod_{1\leq i<j\leq s-2} (z_j-z_i)  (w_j-w_i) \prod_{1\leq i \leq s-2}(1-z_i)^2(1-w_i)^2 } \\ &\;\;\;\;\;\;\;\;\;\;\;\;\;\;\;\;\;\;\;\;\;\;\;\;\;\;\;\;\;\;\;\;\;\;\;\;\;\;\;\;\;\;\;\;\;\;\;\;\;\;\;\;\;\;\;\;\;\;\;\;\;\;\;\;\;\;\;\;\;\;\;\;\;\times \frac{\partial}{\partial z_{s-1}} \frac{\partial}{\partial w_{s-1}} \det_{1\leq i,j\leq s}\left[ \sum_{l=0}^{N+s-1}(z_iw_j)^l \right]\Bigg|_{ \substack{ w_{s-1}=w_s=1 \\ z_{s-1}=z_s=1}}.
\end{align*}
Continuing the process above where for each $j=s,s-1,\dots,1$, we are getting a factor of $(1-z_j)^{s-j}$ in the denominator at the $j$-th step, we finally obtain:
\begin{multline*}
    \frac{\partial^n}{\partial z_1^n} 
\frac{\partial^m}{\partial w_1^m} \frac{\det_{1\leq i,j\leq s}\left[ \sum_{l=0}^{N+s-1}(z_iw_j)^l \right] }{\prod_{1\leq i<j\leq s} (z_j-z_i) \prod_{1\leq i<j \leq s} (w_j-w_i)} \Bigg|_{ \substack{ w_1=\cdots=w_s=1 \\ z_1=\cdots=z_s=1}}\\= (n!) (m!) \frac{1}{(s-1+n)! (s-1+m)!} \prod_{j=2}^s \frac{1}{((s-j)!)^2}
\frac{\partial^{n+s-1}}{\partial z_1^{n+s-1}} \frac{\partial^{m+s-1}}{\partial w_1^{m+s-1}} \\ \;\;\;\;\;\; \prod_{i=2}^s \left(
\Bigg(\frac{\partial}{\partial z_i}\Bigg)^{s-i} \Bigg(\frac{\partial}{\partial w_i}\Bigg)^{s-i} \right ) 
\det_{1\leq i,j\leq s}\left[ \sum_{l=0}^{N+s-1}(z_iw_j)^l \right]\Bigg|_{ \substack{ w_1=\cdots=w_s=1 \\ z_1=\cdots=z_s=1}}.
\end{multline*}
Then, one can compute:
\begin{align*}
    &\frac{\partial^{n+s-1}}{\partial z_1^{n+s-1}} \frac{\partial^{m+s-1}}{\partial w_1^{m+s-1}} 
\prod_{i=2}^s \left(
\big(\frac{\partial}{\partial z_i}\big)^{s-i} \big(\frac{\partial}{\partial w_i}\big)^{s-i} \right ) 
\det_{1\leq i,j\leq s}\left[ \sum_{l=0}^{N+s-1}(z_iw_j)^l \right]\Bigg|_{ \substack{ w_1=\cdots=w_s=1 \\ z_1=\cdots=z_s=1}} \\  &=\det_{1\leq i,j\leq s}\left[ \frac{\partial^{p_i+q_j}}{\partial z_i^{p_i} \partial w_j^{q_j}}
\sum_{l=0}^{N+s-1}(z_iw_j)^l \right] \Bigg|_{ \substack{ w_1=\cdots=w_s=1 \\ z_1=\cdots=z_s=1}}= \det_{1\leq i,j\leq s}\left (\int_0^N x^{p_i+q_j} \mathrm{d}x \right) +O(N^{s^2+m+n-1}),
\end{align*}
where the last equality is obtained by a standard application of the Euler-Maclaurin formula. The result follows from the above.
\end{proof}
\begin{proof}[Proof of Corollary \ref{2ndmomentcor}]
We start by proving the second equality in \eqref{2ndmomentexplicit}. Note that by Theorem \ref{mainresult2}, we know that
\begin{equation*}
    s\mapsto  \mathbb{E}\left[\Bigg|\sum_{m=0}^n (-\textnormal{i})^{m} \binom{n }{m } \mathsf{Y}_{n-m}(s)\Bigg|^2\right]
\end{equation*}
is rational on $s>\frac{1}{2}$. Furthermore, by a combination of Lemmas \ref{serieslem1}, \ref{explicitcontourint} and \ref{serieslem2}, one can see that for $s\in \mathbb{N}$,  $s\geq 2$
\begin{equation*}
    \mathbb{E}\left[\Bigg|\sum_{m=0}^n (-\textnormal{i})^{m} \binom{n }{m } \mathsf{Y}_{n-m}(s)\Bigg|^2\right]
    =2^{2n} \frac{G(2s+1)}{G(s+1)^2}\frac{(n!)^2}{((s-1+n)!)^2}\prod_{j=2}^{s}\frac{1}{((s-j)!)^2}  \det_{1\leq i,j\leq s}\left[\frac{1}{p_i+p_j+1}\right],
\end{equation*}
where $p_1=s-1+n$ and
$p_i=s-i$, for $i=2,\ldots,s$. Thus, if one shows that the right-hand side of the above equation extends to a rational function on $s>\frac{1}{2}$, then we have an equality of two rational functions on the real line that agree on infinitely many points, and hence they must agree for all $s>\frac{1}{2}$. To see that this holds, note that by the general formula for Cauchy matrices, we have that
\beas
\det_{1\leq i,j\leq s}\left[\frac{1}{p_i+p_j+1} \right]
&=& \frac{ \prod_{1\leq i<j\leq s}(p_j-p_i)^2}{\prod_{1\leq i, j\leq s}(p_i+p_j+1) } \\
&=& \frac{1}{2s-1+2n} \prod_{i=2}^s\frac{(n+i-1)^2}{ (2s+n-i)^2}
\frac{\prod_{i=2}^s ((s-i)!)^2}{\prod_{2\leq i,j\leq s} (2s+1-i-j)}.
\eeas
Substituting this into the above expression and simplifying it, we obtain that
\begin{equation*}
    \mathbb{E}\left[\Bigg|\sum_{m=0}^n (-\textnormal{i})^{m} \binom{n }{m } \mathsf{Y}_{n-m}(s)\Bigg|^2\right]
    =2^{2n} \frac{2s-1}{2s-1+2n}\prod_{l=1}^{n}\left(\frac{l+s-1}{l+2s-2}\right)^2
\end{equation*}
for $s\in \mathbb{N}$, $s\geq 2$ and hence for all $s>\frac{1}{2}$. This gives the second equality in \eqref{2ndmomentexplicit}. For the first equality, note that 
\begin{align*}
    &\frac{\partial^n}{\partial z_1^n} 
\frac{\partial^n}{\partial w_1^n} \mathbb{E}_{\haar}\left[\prod_{i=1}^{s}(z_{i}w_{i})^{-\frac{N}{2}}\prod_{j=1}^N \prod_{i=1}^s (\lambda_j - z_i) (\bar{\lambda}_j - w_i)\right
]\Bigg|_{\substack{z_1=\cdots=z_s=1 \\ w_1=\cdots=w_s=1}}\\&=\sum_{i,j=0}^{n}\binom{n}{i}\binom{n}{j}\prod_{l=0}^{n-i-1}\left(-\frac{N}{2}-l\right)\prod_{l=0}^{n-j-1}\left(-\frac{N}{2}-l\right) 
\frac{\partial^{i}}{\partial z_1^{i}} 
\frac{\partial^{j}}{\partial w_1^{j}} \mathbb{E}_{\haar}\left[\prod_{j=1}^N \prod_{i=1}^s (\lambda_j - z_i) (\bar{\lambda}_j - w_i) \right]\Bigg|_{\substack{z_1=\cdots=z_s=1 \\ w_1=\cdots=w_s=1}}.
\end{align*}
Hence, arguing similarly as above, we obtain the first equality in \eqref{2ndmomentexplicit}. Note that for the second part of the Corollary, one only needs to show positivity of the expressions that appear in \eqref{2ndmomentexplicit} for $s>\frac{1}{2}$. Positivity of the second equation is obvious. To see the positivity of the first one, we use the result that a symmetric Cauchy matrix 
\begin{equation*}
    \mathbf{M}=\left[\frac{1}{q_{i}+q_{j}}\right]_{i,j=0,\dots,n}
\end{equation*}
is positive definite if and only if all the $q_{i}$'s are positive and mutually distinct. In particular, if we let $q_{i}=s-\frac{1}{2}+i$, $i=0,\ldots,n$, then
\begin{equation*}
     \mathbf{M}=\left[\frac{1}{2s-1+i+j}\right]_{i,j=0,\ldots,n}
\end{equation*}
is positive definite. Now, consider the vector $\mathbf{x}=(x_{0},x_1,\ldots,x_{n})^{\textnormal{T}}$ with
\begin{equation*}
  x_{i}=\binom{n}{i}\left(-\frac{1}{2}\right)^{n-i}
\prod_{l=0}^{i-1}(l+s)\prod_{l=i+1}^{n}(l+2s-2).
\end{equation*}
Then, we see that
\[\lim_{N\rightarrow \infty} \frac{\mathfrak{F}_N^{(n,0)}(1,s-1)}{N^{s^2+2n}} 
=\frac{G^2(s+1)}{G(2s+1)} \frac{2s-1}{\prod_{l=1}^{n}(2s-2+l)^2}\mathbf{x} ^{\textnormal{T}} \mathbf{M}\mathbf{x}>0,\]
which completes the proof of the result.
\end{proof}
\section{Connections to integrable systems}\label{SectionPainleve}
\subsection{Reduction to finite $N$ and $s\in \mathbb{N}$}
As explained in the introduction, Theorem \ref{painlevethm} will be proven by a limiting argument: we obtain the analogue of the differential relation in \eqref{formula1} at the finite-$N$ level in Theorem \ref{structureforgeneralfinitesize}, and then take the $N\to \infty$ limit. It turns out that by virtue of the following analytic continuation argument, we do not lose generality by restricting to $s\in \mathbb{N}$ in Theorem \ref{structureforgeneralfinitesize}. This, in turn, will allow us to represent the quantities of interest as derivatives of smooth functions, which, from their original definition, is not at all obvious. In particular, from now on we will restrict our attention to $s \in \mathbb{N}$ in proving Theorem \ref{structureforgeneralfinitesize}.

\begin{lem}\label{reductiontoints}
   Assume the statement of Theorem \ref{structureforgeneralfinitesize} holds for all $s\in \mathbb{N}$, with $s>2^{-1}(\sum_{q=2}^k qn_q-1)$. Then, the same statement holds for all $s\in \mathbb{C}$ with $\Re(s)>2^{-1}(\sum_{q=2}^k qn_q-1)$.
\end{lem}
\begin{proof}
    Let $l\in \mathbb{N}$ be the smallest integer so that $l>2^{-1}(\sum_{q=2}^k qn_q -1)$. Then, we have for all $s\in \mathbb{C}$ with $\Re(s)\geq l$ that
\begin{equation*}
    \left|\frac{\mathrm{d}^m}{\mathrm{d}t_1^m} \int_{\mathbb{R}^N} \prod_{j=1}^N \frac{e^{-\textnormal{i}\frac{t_1}{N} x_j}}{\left(1+x_j^2\right)^{s+N}} \Delta^2(\mathbf{x}) \mathrm{d}\mathbf{x}\right|\leq \frac{1}{N^m}\int_{\mathbb{R}^N} \Bigg|\sum_{j=1}^N x_j
\Bigg|^m \prod_{j=1}^N \frac{1}{\left(1+x_j^2\right)^{l+N}} \Delta^2(\mathbf{x}) \mathrm{d}\mathbf{x}, 
\end{equation*}
for $m=0,\ldots,\sum_{q=2}^k (q-1)n_q$ and also that
\begin{align*}
\Bigg|\int_{\mathbb{R}^N}\prod_{q=2}^k\left(\sum_{j=1}^N (x_j-\textnormal{i})^q\right)^{n_q} \prod_{j=1}^N &\frac{e^{-\textnormal{i}\frac{t_1}{N} x_j}}{\left(1+x_j^2\right)^{s+N}} \Delta^2(\mathbf{x}) \mathrm{d}\mathbf{x}\Bigg|\\ &\leq \int_{\mathbb{R}^N}\prod_{q=2}^k\left(\sum_{j=1}^N \big(1+x_j^2\big)^{q/2}\right)^{n_q} \prod_{j=1}^N \frac{1}{\left(1+x_j^2\right)^{l+N}} \Delta^2(\mathbf{x}) \mathrm{d}\mathbf{x}.
\end{align*}
Now, observe that proving \eqref{general structure1} amounts to showing the equality after the averages are replaced by these integrals, due to the cancellation of the normalization constants of $\mathbf{M}_N^{(s)}$. Thus, using the bounds for the integrals above and the fact that $P_{m}^{(s,N)}(t_1)$ are polynomials in $s$, we see that both sides satisfy the conditions of Carlson's lemma. Therefore, if they agree for $s\in \mathbb{N}$, $s\geq l$, then they agree for all $s$ with $\Re(s)\geq l$. Hence, noting that both sides are analytic on $\Re(s)>2^{-1}(\sum_{q=2}^k qn_q-1)$ and using the identity theorem, we get the desired result.
\end{proof}
The vast majority of the rest of this section will be devoted to proving Theorem \ref{structureforgeneralfinitesize}. In the end, we will show that taking the $N\to \infty$ limit indeed implies Theorem \ref{painlevethm}, which, without the machinery we developed earlier, is highly non-trivial.\\
\indent Here, we stress again that in proving Theorem \ref{structureforgeneralfinitesize}, by virtue of Lemma \ref{reductiontoints}, one may restrict to $s\in \mathbb{N}$, which we henceforth do. The strategy we will adopt can then be summarized in the following steps:
\begin{itemize}
    \item \textbf{Step 1}: We will write the averages of the form
    \begin{equation}\label{painleveprelimitquantity}
        \mathbb{E}_N^{(s)}\left[e^{-\textnormal{i} {t_1}\sum_{j=1}^N \mathsf{x}_j^{(N)}} \left(\sum_{j=1}^N \left(\mathsf{x}_j^{(N)}-\textnormal{i}\right)^2\right)^{\ell_2
        } \cdots \left(\sum_{j=1}^N \left(\mathsf{x}_j^{(N)}-\textnormal{i}\right)^k\right)^{\ell_{k}
        } \right],
    \end{equation}
in terms of derivatives of Hankel-determinants. To be more precise, up to a power of $N$, we will show that the above average can be given as
    \begin{equation*}
        \prod_{m=2}^k \frac{\partial^{\ell_{m}}}{\partial t_m ^{\ell_{m}}} \det_{0\leq i,j\leq N-1} \left[\boldsymbol{\psi}_{i+j}(t_1,\ldots, t_k)\right]
        \Big|_{t_2,\ldots, t_k=0}= \prod_{m=2}^k \frac{\partial^{\ell_{m}}}{\partial t_m ^{\ell_{m}}} \mathbf{\Psi}_N(t_1,\ldots, t_k) \Big|_{t_2,\ldots, t_k=0},
    \end{equation*}
     for some explicit functions $\boldsymbol{\psi}_{\gamma}(t_1,\ldots,t_k)$ with $\mathbf{\Psi}_N(t_1,\dots,t_k)$ denoting their Hankel determinant.
     \item \textbf{Step 2:} Next, we will introduce the associated functions $\mathbf{\Psi}_{N,\boldsymbol{\lambda}}$ where $\boldsymbol{\lambda}$ are integer partitions. These are basically generalizations of the Hankel determinant on the left hand-side of the equation above, such that $\mathbf{\Psi}_{N,\emptyset}= \mathbf{\Psi}_N$. We will then show that the $t_m$-derivatives of $\mathbf{\Psi}_N$, where $m=2,\ldots,k$ are given as linear combinations of $\mathbf{\Psi}_{N,\boldsymbol{\lambda}}$ where $\boldsymbol{\lambda}$ ranges over a specific class of partitions. 
     \item \textbf{Step 3:} Moving on, we will show that the quantities $\mathbf{\Psi}_{N,\boldsymbol{\lambda}}$ satisfy certain recursive relations as $\boldsymbol{\lambda}$ ranges over a specific class of integer partitions.
     \item \textbf{Step 4:} Finally, we will use these recursive relations to prove the result by a delicate induction argument.
\end{itemize}
It is also worth noting that even though considering power sums after shifting the eigenvalues by $-\textnormal{i}$ might look unnatural, it is done to ensure that we are able to write the relations between Hankel determinants and their shifted versions in a concise manner. Indeed, the joint moments of power sums without such shifts can be written as linear combinations of the averages in \eqref{painleveprelimitquantity}. In addition, due to our earlier results, the additional terms one obtains from this shift do not make any contribution in the large-$N$ limit.

\subsection{Joint moments and Hankel determinants}\label{computationfirstsec}
Our starting point is the following proposition.

\begin{prop}\label{represented as a Hankel}
Let $N\geq 1, k\geq 2$ and $l_2,\ldots,l_k\geq 0$ be integers, and let $s>{(\sum_{m=2}^k m l_m - 1)}/{2}$ be an integer. Let $\mathsf{U}(a,b;z)$ be the confluent hypergeometric function of the second kind, and let
\bea\label{definition of am}
\phi_{m}(t_{1})\overset{\textnormal{def}}{=}\exp(-t_{1})\mathsf{U}(1-N-s,2-2N-2s+m;2t_{1}).
\eea 
Then, for $t_1\geq 0$, we have
\begin{align*}
\eqref{painleveprelimitquantity}
 = C_{N}^{(s)}(-2\textnormal{i})^{\sum_{m=2}^k m \ell_{m}} 
 \prod_{m=2}^k \frac{\partial^{\ell_m}}{\partial t_m^{\ell_m}}
 \det_{0\leq i,j\leq N-1}\left[ \sum_{m_2,\ldots,m_k=0}^\infty \frac{t_2^{m_2}\cdots t_k^{m_k}}{m_2 !\cdots m_k!} \phi_{i+j+\sum_{l=2}^k l m_l}(t_1)
 \right]\Bigg\vert_{t_2=\cdots=t_k=0}
\end{align*}
for some constant $C_{N}^{(s)}$, whose explicit value will not play a role in what follows. 
\end{prop}
It will be clear, while proving the result above, that the earlier restriction to $s\in \mathbb{N}$ guarantees that the power sum considered above is well-defined, and has derivatives of all orders. This will then allow us to use the recursive relations for $\phi_\gamma$ to obtain desired relations for these series, which will be one of the main ingredients in proving various relations between Hankel determinants of concern. Moving on, to prove the proposition above, we will need two elementary lemmas.

\begin{lem}\label{elemlem1}
Let $g(t_1,\ldots, t_N)$ be a symmetric function of the form
\begin{equation*}
    g(t_1,\ldots, t_N)=\sum_{r=1}^R a_r \prod_{i=1}^N f_{r,i}(t_i)
\end{equation*}
for functions $f_{r,i} \in C^{2N-2}(\mathbb{R})$. Then, 
\begin{equation*}
   \prod_{1\leq i<j\leq n} \Bigg(\frac{\partial}{\partial t_j}-\frac{\partial}{\partial t_i} \Bigg)^2 g(t_1,\ldots, t_N) \; \Big\rvert_{t_1=\cdots=t_N=t}=N! \sum_{r=1}^R a_r \det_{0\leq i,j\leq N-1} \left[f_{r,i+1}^{(i+j)}(t)\right].
\end{equation*}
\end{lem}

\begin{lem}\label{hypergeomgrowthlem}
    Let $N\geq 1$, $s \in \mathbb{N}$, and let $I\subset \mathbb{R}$ be a compact set. Then, there exists a constant $K=K_{N,I,s}$ so that for all $t\in I$, 
    \begin{equation}
        \left|\phi_m(t)\right| \leq K_{N,I,s} m^{N+s} 
    \end{equation}
  where $\phi_m$ is defined as in \eqref{definition of am}.
\end{lem}
\begin{proof}
    Using the well-known relation between the confluent hypergeometric function $\mathsf{U}$ and generalised Laguerre polynomials (see, for instance, 
    \cite[p. 755]{arfken}), we have 
    \begin{equation*}
        \mathsf{U}(1-N-s,2-2N-2s+m;2t_{1})=(-1)^{N+s-1} (N+s-1)!\: \mathsf{L}_{N+s-1}^{(1-2N-2s+m)}(2t_1),
    \end{equation*}
    where the generalised Laguerre polynomial $\mathsf{L}_n^{(\alpha)}$ is given by
    \begin{equation*}
        \mathsf{L}_n^{(\alpha)}(z)=\frac{z^{-\alpha}e^{z}}{n!} \frac{\mathrm{d}^n}{\mathrm{d}z^n}\left(e^{-z} z^{n+\alpha}\right).
    \end{equation*}
Computing the derivative explicitly gives the desired bound immediately.
\end{proof}
\begin{proof}[Proof of Proposition \ref{represented as a Hankel}]
We start by noting that (from here on the constant $C_N^{(s)}$ might change from line to line) 
\begin{align*}
\eqref{painleveprelimitquantity} 
&=C_N^{(s)} e^{Nt_1} \int_{\mathbb{R}^N} \prod_{j=1}^N \frac{e^{-\textnormal{i}t_1 (x_j-\textnormal{i})}}{\left(1+x_j^2\right)^{s+N}} \prod_{m=2}^k \left(\sum_{j=1}^N (x_j-\textnormal{i})^m\right)^{\ell_m} \Delta^2(x_1,\ldots, x_N)\mathrm{d} \mathbf{x}\\ &= C_N^{(s)}e^{Nt_1}  \prod_{1\leq i<j\leq N} \Bigg(\frac{\partial}{\partial w_j}-\frac{\partial}{\partial w_i} \Bigg)^2  \int_{\mathbb{R}^N} \prod_{j=1}^N \frac{e^{-\textnormal{i}w_j (x_j-\textnormal{i})}}{\left(1+x_j^2\right)^{s+N}} \prod_{m=2}^k \left(\sum_{j=1}^N (x_j-\textnormal{i})^m\right)^{\ell_m} \mathrm{d} \mathbf{x} \Big\rvert_{w_1=\cdots=w_N=t_1}.
 \end{align*}
Now, expanding the powers of $\sum_j(x_j-\textnormal{i})^m$ in the final integral above and applying Lemma \ref{elemlem1} to this symmetric function, we obtain that 
\begin{equation*}
    \eqref{painleveprelimitquantity}  = C_N^{(s)} \textnormal{i}^{\sum_{m=2}^k m \ell_m} e^{Nt_1} \sum_{\substack{ \mu_{2,1}+\cdots +\mu_{2,N}=\ell_2\\[-4pt] \vdots  \\[2pt] \mu_{k,1}+\cdots+\mu_{k,N}=\ell_k}} \prod_{m=2}^k {\ell_m \choose {\mu_{m,1},\dots, \mu_{m,N}}} \det_{0\leq i,j\leq N-1} \left[ f^{(i+j+\sum_{m=2}^k m \mu_{m,j+1} )}(t_1)\right]
\end{equation*}
where the function $f$ is given, for $t\geq 0$, as
\begin{equation*}
    f(t)\overset{\textnormal{def}}{=}\int_{\mathbb{R}} \frac{e^{-\textnormal{i}t(x-\textnormal{i})}}{(1+x^2)^{s+N}}\mathrm{d}x = 
    \frac{2\pi}{2^{2N+2s-1} \Gamma(N+s)}
    e^{-2t} \mathsf{U}(1-N-s,2-2N-2s;2t),
\end{equation*}
(see \cite[p.349 formula (9)]{zwillinger2007table}).
But then, combining the growth condition from Lemma \ref{hypergeomgrowthlem} and the differential relation (see \cite[p. 258 formula (14)]{hammer1953higher})
    \[
\frac{\mathrm{d}^m}{\mathrm{d}t^m} \Big( e^{-2t} \mathsf{U}(1-N-s, 2-2N-2s; 2t) \Big)
= (-2)^m e^{-2t} \mathsf{U}(1-N-s, 2-2N-2s+m; 2t),
\]
we have the desired result.
\end{proof}

\subsection{Hankel determinants shifted by partitions}

Following the strategy sketched earlier we define the Hankel determinants:
\begin{equation}\label{Hndef}
    \Psi_N(t_1,\ldots,t_k)=\det_{0\leq i,j\leq N-1}\left[ \sum_{m_2,\dots,m_k=0}^\infty \frac{t_2^{m_2}\cdots t_k^{m_k}}{m_2 !\cdots m_k!} \phi_{i+j+\sum_{l=2}^k l m_l}(t_1)
 \right].
\end{equation}
Next, to make the presentation clearer, we fix some notation that will be used throughout the rest of this section.
Firstly, let
\begin{equation*}
    \psi_{\gamma}(t_1,\ldots,t_k)\overset{\textnormal{def}}{=}\sum_{m_2,\ldots,m_k=0}^\infty \frac{t_2^{m_2}\cdots t_k^{m_k}}{m_2 !\cdots m_k!} \phi_{\gamma+\sum_{l=2}^k l m_l}(t_1),
\end{equation*}
for any $\gamma\in \mathbb{N}$. 
Indeed, since we are mostly concerned about the $N\to \infty$ behaviour of these determinants, it becomes more natural to consider the determinant given in \eqref{Hndef} with the replacement $t_{1}\rightarrow {t_{1}}/{N}, \ldots, t_{k}\rightarrow {t_{k}}/{N}$. Hence, we introduce 
\begin{equation*}
\boldsymbol{\psi}_\gamma(t_1,\ldots,t_k)\overset{\textnormal{def}}{=}\psi_{\gamma}\left(\frac{t_1}{N},\ldots, \frac{t_k}{N}\right),
\end{equation*}
and accordingly define,
\[
\mathbf{\Psi}_{N}(t_1,\ldots,t_k)= \det_{0\leq i,j \leq N-1} \Big[\boldsymbol{\psi}_{i+j}(t_1,\ldots,t_k) \Big].
\]
Note that these functions are all well-defined due to Lemma \ref{hypergeomgrowthlem}. Next, we define shifted and weighted versions of the Hankel matrix whose determinant defines $\mathbf{\Psi}_N$. For each integer partition $\boldsymbol{\lambda}=(\lambda_1\geq \cdots \geq \lambda_p > 0)$, we let
\begin{equation*}
    \mathbf{A}_{N,\boldsymbol{\lambda}}(t_1,\ldots, t_k)\overset{\textnormal{def}}{=}\left[\boldsymbol{\psi}_{i+j+\lambda_{N-j}}(t_1,\ldots, t_k)\right]_{i,j=0,\ldots,N-1},
\end{equation*}
where if $p\leq N-1$, we let 
$\lambda_{p+1}=\cdots=\lambda_N=0.$ When $p>N$, we define $\mathbf{A}_{N,\boldsymbol{\lambda}}(t_1,\ldots, t_k)=0$.
Accordingly, we let 
$$\mathbf{\Psi}_{N,\boldsymbol{\lambda}}(t_1,\ldots,t_k)\defeq \det[\mathbf{A}_{N,\boldsymbol{\lambda}}].$$ 
Note that in the special case, $\boldsymbol{\lambda}=\emptyset = (0,0,\ldots,0)$, one has $\mathbf{\Psi}_{N,\emptyset}=\mathbf{\Psi}_N$.
Related to these matrices, we also define their weighted version:
\[
{\mathbf{A}}_{N,\boldsymbol{\lambda}}^{(w)} (t_1,\ldots, t_k) \overset{\textnormal{def}}{=} \left[(i+j+\lambda_{N-j})\boldsymbol{ \psi}_{i+j+\lambda_{N-j}}(t_1,\ldots, t_k)\right]_{i,j=0,\ldots,N-1} ,
\]
and let
\[
\mathbf{\Psi}_{N,\boldsymbol{\lambda}}^{(w)}(t_1,\ldots,t_k) =
\det[\mathbf{A}_{N,\boldsymbol{\lambda}}^{(w)}].
\]
Finally, we define the main objects that will allow us to obtain recurrences for derivatives of $\mathbf{\Psi}_{N,\boldsymbol{\lambda}}$, which will become very crucial for our induction argument.

\begin{defn}
Let $h \in \mathbb{N} \cup \{0\}$.
    For each integer partition $\boldsymbol{\lambda}=(\lambda_1,\ldots, \lambda_p)$ with $p\leq N$, define its translated version:
\begin{equation*}
  \mathcal{S}_h \boldsymbol{\lambda}=  (\lambda_1+h,\ldots, \lambda_p+h,h,\ldots,h).
\end{equation*}
Hence, we define the functions $\mathbf{\Psi}_{N,\boldsymbol{\lambda},h}$ by,
\begin{equation*}
    \mathbf{\Psi}_{N,\boldsymbol{\lambda},h}(t_1,t_2,\ldots,t_k)\overset{\textnormal{def}}{=} \mathrm{Tr}\left[\mathrm{adj}(\mathbf{A}_{N, \boldsymbol{\lambda}})\mathbf{A}_{N,\mathcal{S}_h\boldsymbol{\lambda}}\right], 
\end{equation*}
and $\mathbf{\Psi}_{N,\boldsymbol{\lambda},h}^{(w)}$ by,
\begin{equation*}
  \mathbf{\Psi}_{N,\boldsymbol{\lambda},h}^{(w)}(t_1,t_2,\ldots,t_k)\overset{\textnormal{def}}{=} \mathrm{Tr}\left[\mathrm{adj}\left({{\mathbf{A}}}_{N, \boldsymbol{\lambda}}\right){\mathbf{A}}_{N,\mathcal{S}_h\boldsymbol{\lambda} }^{(w)} \right],
\end{equation*}
where $\mathrm{adj}$ denotes the adjugate of a matrix.
\end{defn}

It turns out that in most of our computations we will have expressions which involve a certain distinguished partition. 
These partitions are defined as follows:

\begin{defn}
Let $1\le q \le n$. Define the partition $\boldsymbol{\lambda}_{n,q}$ of $n$ to be the partition  $(n-q+1,1,\ldots,1)$ where the number of $1$'s that appear at the end is $q-1$.
\end{defn}

In what follows, we will first show certain recursive and differential relations satisfied by the functions $\boldsymbol{\psi}_\gamma$, and then leverage these to obtain relations between derivatives of $\mathbf{\Psi}_{N}$ and its shifted versions defined above.

\begin{lem}\label{relationforhlem}
 For each $\gamma\geq 0$, we have that
\begin{equation}\label{generalcasewithrespecttot1}
\boldsymbol{\psi}_{\gamma+1} = \frac{\boldsymbol{\psi}_{\gamma}}{2} - \frac{N}{2} \frac{\partial \boldsymbol{\psi}_{\gamma}}{\partial t_{1}},
\end{equation}
and also
\begin{align}
\boldsymbol{\psi}_{\gamma+2} &= \frac{N(N+s-1-\gamma)}{2(t_1+t_2)} \boldsymbol{\psi}_\gamma + \left(\frac{N(2-2N-2s+\gamma)}{2(t_1+t_2)} + \frac{t_1}{t_1+t_2}\right) \boldsymbol{\psi}_{\gamma+1}\nonumber \\
&+\frac{1}{2(t_1+t_2)} \sum_{q=3}^k \big((q-1)t_{q-1} - qt_q \big) \boldsymbol{\psi}_{\gamma+q}+ \frac{k t_k}{2(t_1+t_2)} \boldsymbol{\psi}_{\gamma+1+k}.
\label{0807eq1}
\end{align}
\end{lem}
\begin{proof}
    This follows from the properties of the confluent hypergeometric functions
    \begin{align}
(b-a-1)\mathsf{U}&(a,b-1;t)-(b-1+t)\mathsf{U}(a,b;t)+t\mathsf{U}(a,b+1;t)=0,\\
&\frac{\mathrm{d}\mathsf{U}(a,b;t)}{\mathrm{d}t}=\mathsf{U}(a,b;t)-\mathsf{U}(a,b+1;t),
\end{align}
which can be found in \cite[p.257 formula (5) and p.258 formula (10)]{hammer1953higher}. Indeed, the first claim is immediate from the second identity above, whereas using the first one, we obtain:
\begin{align*}
   \boldsymbol{\psi}_{\gamma+2}(t_1,\dots ,t_k)
&=\frac{N(N+s-1-\gamma)}{2t_1} \boldsymbol{\psi}_\gamma(t_1,\dots,t_k) - \frac{1}{2t_1} \sum_{q=2}^k qt_q \boldsymbol{\psi}_{\gamma+q}(t_1,\dots,t_k)\nonumber\\ 
&+ \left( \frac{N(2-2N-2s+\gamma)}{2t_1} +1 \right) \boldsymbol{\psi}_{\gamma+1}(t_1,\dots,t_k)
+ \frac{1}{2t_1} \sum_{q=2}^k qt_q \boldsymbol{\psi}_{\gamma+1+q}(t_1,\dots, t_k).
\end{align*}
Now, note that the only term $\boldsymbol{\psi}_{\gamma+2}$ appearing on the right-hand side above is the $q=2$ term of the first sum. Moving that to the left-hand side we get the desired result.
\end{proof}

\begin{lem}\label{generalexpansiondependingonrecursive}
Let $k\geq 2$ and $q\geq 2$. Let $\boldsymbol{\lambda}$ be a partition. Then we have:
\begin{align}\label{generalizationiterationformula}
\Hfxn_{N,\boldsymbol{\lambda},q}(t_1,\ldots,t_{k})
&=\frac{N(N+s-1)}{2(t_1+t_2)} \Hfxn_{N,\boldsymbol{\lambda},q-2}(t_1,\ldots,t_{k})
- \frac{N}{2(t_1+t_2)} {\Hfxn}^{(w)}_{N,\boldsymbol{\lambda},q-2}(t_1,\ldots,t_{k})\nonumber \\
&+\, \frac{N(1-2N-2s)}{2(t_1+t_2)}  \Hfxn_{N,\boldsymbol{\lambda},q-1}(t_1,\ldots,t_{k})
+ \frac{N}{2(t_1+t_2)} {\Hfxn}^{(w)}_{N,\boldsymbol{\lambda},q-1}(t_1,\ldots,t_{k}) \nonumber\\
&+\, \frac{t_1}{t_1+t_2} \Hfxn_{N,\boldsymbol{\lambda},q-1}(t_1,\ldots,t_{k})
+ \frac{kt_k}{2(t_1+t_2)} \Hfxn_{N,\boldsymbol{\lambda},q+k-1}(t_1,\ldots,t_{k})\nonumber\\
&+\frac{1}{2(t_{1}+t_{2})}\sum_{l=3}^{k}\big((l-1)t_{l-1}-lt_{l}\big)\Hfxn_{N,\boldsymbol{\lambda},q+l-2}(t_1,\ldots,t_{k}).
\end{align}
\end{lem}
\begin{proof}
The claim follows by applying the second recursive relation in Lemma \ref{relationforhlem} to $\mathbf{\Psi}_{N,\boldsymbol{\lambda},q}$ and the linearity of the trace.
\end{proof}
\begin{lem}
\label{generalderivative and translations}
For each $q=2,\ldots,k$ we have that
\beas
\frac{\partial \Hfxn_{N,\boldsymbol{\lambda}}}{\partial t_{q}}=\frac{1}{N}\Hfxn_{N,\boldsymbol{\lambda},q}
\eeas
and also that
\bea\label{t1operateongeneralG}
\Hfxn_{N,\boldsymbol{\lambda},1}=-\frac{N}{2}\frac{\partial \Hfxn_{N,\boldsymbol{\lambda}}}{\partial{t_{1}}}+\frac{N}{2}\Hfxn_{N,\boldsymbol{\lambda}}.
\eea
\end{lem}
\begin{proof}
   The first claim follows by noting that
    \begin{equation*}
        \frac{\partial \hfxn_\gamma}{\partial t_q}=\frac{1}{N} \hfxn_{\gamma+q}
    \end{equation*}
whereas the second claim follows from \eqref{generalcasewithrespecttot1}.
\end{proof}
Next, we give recursive formulae for $\Hfxn_{N,\boldsymbol{\lambda}}$. To do this in a compact way, we first need some notation.
\begin{defn}
    For integers $r \geq 0, b\geq 1$ and $p$ with $1\leq p\leq b$, we define
\begin{equation*}
    \Hfxn_N^{(r)}[b;p]\overset{\textnormal{def}}{=}\left(\Hfxn_{N,\boldsymbol{\lambda}_{b,p}}, \Hfxn_{N,\boldsymbol{\lambda}_{b,p+1}},\ldots, \Hfxn_{N,\boldsymbol{\lambda}_{b,b}},0,\ldots,0 \right)^\textnormal{T},
\end{equation*}
where $r$ specifies the number of $0$'s that appear.
\end{defn}
Next, we give recursive formulae for the vectors defined above. These recursions will represent certain vectors as linear combinations of other allied vectors after they are multiplied by suitable matrices. These matrices do have explicit expressions, which we give in the Appendix.

\begin{prop}\label{prop:recursivevector}
    Let $N,k,l$ be integers with $N\geq 1, k\geq 2, l\geq 3$. Then, there are matrices $\mathbf{B}^{(l)}$ and $\mathbf{Q}_m^{(l)}$ with $m=0,1,\ldots,k+1$, all of which, except $\mathbf{Q}_2^{(l)}$, have all entries independent of $N$, so that:
    \begin{align*}
        \Hfxn_{N}^{(0)}&[l;1]= \frac{k t_k (-1)^{k+1} \mathbf{B}^{(l)}}{2(t_1+t_2)} \left(-\frac{N}{2} \frac{\partial}{\partial t_1} + \frac{N}{2}\right)\Hfxn_{N}^{(1)}[l+k-2;k]  \\
        &+\mathbf{B}^{(l)} \Bigg(\frac{N}{2} \frac{\partial}{\partial t_1} - \frac{N}{2} + \frac{N}{2(t_1+t_2)}\Bigg(
    \sum_{m=3}^k  \mathsf{diff}_{m-1,m} \frac{\partial}{\partial t_{m-1}} + kt_k \frac{\partial}{\partial t_k} \Bigg) \Bigg) \Hfxn_{N}^{(1)}[l-1;1]
    \\ & + \frac{1}{2(t_1+t_2)} \sum_{m=1}^{k-2}  \mathsf{diff}_{m+1,m+2}  (-1)^m \mathbf{Q}_{m+2}^{(l)}   \Hfxn_{N}^{(0)}[l+m;1]\\
    &+ \left( \frac{t_1}{t_1+t_2} \mathbf{Q}_1^{(l)} + \frac{N}{2(t_1+t_2)} {\mathbf{Q}}_0^{(l)}  \right) \Hfxn_{N}^{(0)}[l-1;1]
+ \frac{N}{2(t_1+t_2)} \mathbf{Q}_2^{(l)} \Hfxn_{N}^{(0)}[l-2;1] \\ &- \frac{\mathbf{B}^{(l)}}{2(t_1+t_2)} \sum_{m=0}^{k-3}\mathsf{diff}_{m+2,m+3}(-1)^m (-\frac{N}{2} \frac{\partial}{\partial t_1} + \frac{N}{2}) \Hfxn_{N}^{(1)}[l+m;m+2]\\ & + \frac{1}{2(t_1+t_2)}  (-1)^{k-1} k t_k \mathbf{Q}_{k+1}^{(l)}\Hfxn_{N}^{(0)}[l+k-1;1]\\& +\frac{\mathbf{B}^{(l)} Nkt_{k}}{2(t_{1}+t_{2})}\sum_{h=2}^{k-1} (-1)^{k+h} \frac{\partial }{\partial t_h} \Hfxn_{N}^{(1)}[l+k-1-h;k+1-h] \\& + \frac{\mathbf{B}^{(l)} N}{2(t_1+t_2)} \sum_{m=4}^k (-1)^{m-1} \mathsf{diff}_{m-1,m} \sum_{h=2}^{m-2} (-1)^h \frac{\partial }{\partial t_h} \Hfxn_{N}^{(1)}[l+m-2-h;m-h],
    \end{align*}
    where we denote
    \begin{equation*}
        \mathsf{diff}_{p,q}= \mathsf{diff}_{p,q}(t_1,\ldots,t_k)\overset{\textnormal{def}}{=} pt_p-qt_q.
    \end{equation*}
  
\end{prop}
\begin{proof}
    We start by noting that after a straightforward but lengthy computation, where we use Lemma \ref{generalexpansiondependingonrecursive} and \cite[Theorem 17]{keating-fei}, we obtain:
    \begin{align*}
         \Hfxn_{N}^{(0)}[l;1]=&\mathbf{B}^{(l)}\Bigg(-\Hfxn_{N}^{(1)}[l-1;1]+\frac{N(N+s-1)}{2(t_{1}+t_{2})}\mathcal{R}_{2}+\frac{(1-2N-2s)N+2t_{1}}{2(t_{1}+t_{2})}\mathcal{R}_{1}\\&-\frac{N}{2(t_{1}+t_{2})}\mathcal{R}_{2}^{(w)}\nonumber+\frac{N}{2(t_{1}+t_{2})}\mathcal{R}_{1}^{(w)}
+\frac{1}{2(t_1+t_2)} \sum_{m=3}^k \mathsf{diff}_{m-1,m} \mathcal{R}_m +
\frac{kt_k}{2(t_1+t_2)} \mathcal{R}_{k+1}\Bigg)
    \end{align*}
    where $\mathcal{R}_m \in \mathbb{R}^l$ are given by: 
    \begin{equation*}
        (\mathcal{R}_m)_j \overset{\textnormal{def}}{=} 
    \begin{cases}

        0,  & j=1, \\
             \noalign{\vskip4pt}
         \displaystyle\sum_{h=2}^{j} (-1)^h \Hfxn_{N,\blambda_{l-h,j+1-h},h+\alpha(m)-2} &  j=2,\ldots, l-1, \\ 
         \noalign{\vskip4pt}
         \Hfxn_{N,\emptyset, l+\alpha(m)-2} & j=l,
       \end{cases}
    \end{equation*}
    with $\alpha(2)=0$, and $\alpha(m)=m$ for all other $m$. We also denote by
    ${\mathcal{R}}_1^{(w)},$ and ${\mathcal{R}}_2^{(w)}$ the vectors in $\mathbb{R}^l$ defined in the same way as $\mathcal{R}_1,\mathcal{R}_2$ above with $\Hfxn$ replaced by its weighted version ${\Hfxn}^{(w)}$. Next, for $m=3,\ldots,k+1$, we decompose $\mathcal{R}_m = \mathcal{R}_m'+\mathcal{R}_m^* \in \mathbb{R}^l$ where these latter vectors are given as:
    \begin{align*}
        (\mathcal{R}_m')_j &\overset{\textnormal{def}}{=} 
    \begin{cases}
         \displaystyle\sum_{h=1}^{j+m-2} (-1)^{h+m-2} \Hfxn_{N,\blambda_{l+m-2-h,j+m-1-h},h} &  j=1,\ldots, l-1, \\ 
         \noalign{\vskip4pt}
         \Hfxn_{N,\emptyset, l+m-2} & j=l,
       \end{cases}
       \\
        (\mathcal{R}_m^*)_j &\overset{\textnormal{def}}{=} 
        \begin{cases}
         \displaystyle\sum_{h=1}^{m-1} (-1)^{h+m-1} \Hfxn_{N,\blambda_{l+m-2-h,j+m-1-h},h} &  j=1,\ldots, l-1, \\ 
         \noalign{\vskip4pt}
         0 & j=l.
       \end{cases}
    \end{align*}
Then, we apply Lemma \ref{generalderivative and translations} to $\mathcal{R}_m^*$ with $t_q$ being $t_h$ for the $h$-th summand in the entries of $\mathcal{R}_m^*$. For the remaining terms, by virtue of \cite[Theorem 17, Proposition 20 ]{keating-fei}, after yet another demanding computation, we obtain the identities:
\begin{align*}
    \mathbf{B}^{(l)}\mathcal{R}_1=\mathbf{Q}_1^{(l)} \Hfxn_{N}^{(0)}[l-1;1]\: &,\;\;\;\;  \mathbf{B}^{(l)}{\mathcal{R}}_1^{(w)}=\widetilde{\mathbf{Q}}_1^{(l)} \Hfxn_{N}^{(0)}[l-1;1],\\
    \mathbf{B}^{(l)}\mathcal{R}_2=\widehat{\mathbf{Q}}_2^{(l)} \Hfxn_{N}^{(0)}[l-2;1]\: &,\;\;\;\;  \mathbf{B}^{(l)}{\mathcal{R}}_2^{(w)}=\widetilde{\mathbf{Q}}_2^{(l)} \Hfxn_{N}^{(0)}[l-2;1],\\ \noalign{\vskip4pt}
    \mathbf{B}^{(l)}\mathcal{R}_m'=  \: \mathbf{Q}_m^{(l)} & \Hfxn_{N}^{(0)}[l+m-2;1].\\
    \intertext{All matrices involved above are given in the Appendix. Hence, defining}
    {\mathbf{Q}}_{0}^{(l)}\defeq -2N\mathbf{Q}_{1}^{(l)}+\widetilde{\mathbf{Q}}_{1}^{(l)}+(1-2s)\mathbf{Q}_{1}^{(l)},& \;\;\; \textnormal{and} \;\;\; \mathbf{Q}_{2}^{(l)}\defeq (N+s-1)\widehat{\mathbf{Q}}_{2}^{(l)}-\widetilde{\mathbf{Q}}_{2}^{(l)},
\end{align*}
we obtain the desired result. We note that even though 
both $-2N\mathbf{Q}_{1}^{(l)}$ and $\widetilde{\mathbf{Q}}_{1}^{(l)}$ contain terms linear in $N$, ${\mathbf{Q}}_{0}^{(l)}$ is independent of $N$ due to cancellations. Similarly,
although both $\widetilde{\mathbf{Q}}_2^{(l)}$ and $(N+s-1)\widehat{\mathbf{Q}}_{2}^{(l)}$ contain terms of order $N^2$, these exactly cancel out so that entries of $\mathbf{Q}_{2}^{(l)}$ are $O(N)$.
\end{proof}

The cancellations that take place at the end of the above proof, and the consequent bounds on the entries of the matrices therein will be very crucial while obtaining the desired order of $N$ for the polynomials appearing in Theorem \ref{structureforgeneralfinitesize}.\\
\indent In order to be able to use the recursion above successively, we will also require suitable initial conditions.
\begin{lem}\label{lem:initialcond}
For each $N\geq 1$, we have that
\begin{align*}
    \mathbf{\Psi}_{N,\blambda_{2,1}}&=\frac{N^2}{8}\frac{\partial^2\Hfxn_{N}}{\partial t_{1}^2}-\frac{N^2}{4}\frac{\partial \Hfxn_{N}}{\partial{t_{1}}}+\frac{N^2}{8}\Hfxn_{N}+\frac{N}{2}\frac{\partial \Hfxn_{N}}{\partial t_{2}},
    \\
     \mathbf{\Psi}_{N,\blambda_{2,2}}&=\frac{N^2}{8}\frac{\partial^2\Hfxn_{N}}{\partial t_{1}^2}-\frac{N^2}{4}\frac{\partial \Hfxn_{N}}{\partial{t_{1}}}+\frac{N^2}{8}\Hfxn_{N}-\frac{N}{2}\frac{\partial \Hfxn_{N}}{\partial t_{2}}.
\end{align*}
\end{lem}

\begin{proof}
By the definition of $\Hfxn_{N,\blambda,h}$, we have
\[
\Hfxn_{N,\blambda_{1,1},1} = \Hfxn_{N,\blambda_{2,1}} + \Hfxn_{N,\blambda_{2,2}},
\quad 
\Hfxn_{N,\emptyset,2} = \Hfxn_{N,\blambda_{2,1}} - \Hfxn_{N,\blambda_{2,2}}.
\]
Noting that $\Hfxn_{N,\emptyset,1}=\Hfxn_{N,\blambda_{1,1}}$, the result then follows by virtue of Lemma \ref{generalderivative and translations}.
\end{proof}


\noindent We then arrive at the following proposition.

\begin{prop}\label{recursionfirstprop}
Let $l\geq 3$ and $1\leq q\leq l$. Let $k\geq 2$ and $i_{2},n_{3},\ldots,n_{k}$ be non-negative integers.  Then define, 
\bea\label{definitionforgeneralmathcalFlq}
\mathcal{L}_{N,l,q}^{(i_{2},n_{3},\ldots,n_{k})}(t_{1})\defeq \frac{\partial^{i_{2}}}{\partial t_{2}^{i_{2}}}\frac{\partial^{n_{3}}}{\partial t_{3}^{n_{3}}}\cdots\frac{\partial^{n_{k}}}{\partial t_{k}^{n_{k}}}\mathbf{\Psi}_{N,\boldsymbol{\lambda}_{l,q}}\Bigg|_{t_{2}=\cdots=t_{k}=0},
\eea
\bea\label{definitionforgeneralmathcalFlqv2}
\mathcal{L}_{N}^{(i_{2},n_{3},\ldots,n_{k})}(t_{1})\defeq \frac{\partial^{i_{2}}}{\partial t_{2}^{i_{2}}}\frac{\partial^{n_{3}}}{\partial t_{3}^{n_{3}}}\cdots\frac{\partial^{n_{k}}}{\partial t_{k}^{n_{k}}}\mathbf{\Psi}_{N}\Bigg|_{t_{2}=\cdots=t_{k}=0}.
\eea
Let $m\geq 0$ be an integer.
Denote
\beas
\textbf{L}_{l}^{(i_{2},n_{3},\ldots,n_{k})} = \begin{pmatrix}
\mathcal{L}_{N,l,1}^{(i_{2},n_{3},\ldots,n_{k})}(t_{1}),
\ldots,
\mathcal{L}_{N,l,l}^{(i_{2},n_{3},\ldots,n_{k})}(t_{1}) 
\end{pmatrix}^\textnormal{T}
\eeas
and
\beas
\hat{\textbf{L}}_{l,m}^{(i_{2},n_{3},\ldots,n_{k})} = \begin{pmatrix}
\mathcal{L}_{N,l+m,m+2}^{(i_{2},n_{3},\ldots,n_{k})}(t_{1}), 
\ldots ,
\mathcal{L}_{N,l+m,l+m}^{(i_{2},n_{3},\ldots,n_{k})}(t_{1}),\: 0 
\end{pmatrix}^\textnormal{T}.
\eeas
Then, we have that
\bea\label{generalrecursiveconditionformathcalG}
\textbf{L}_{l}^{(i_{2},n_{3},\ldots,n_{k})}
=\mathbf{B}^{(l)}\left(\frac{N}{2} \frac{\mathrm{d}}{\mathrm{d} t_1} - \frac{N}{2}\right)
\begin{pmatrix}
\textbf{L}_{l-1}^{(i_{2},n_{3},\ldots,n_{k})} \\
0
\end{pmatrix}
+\sum_{n_2=0}^{i_2} \frac{(-1)^{i_2-n_2} \binom{i_2}{i_2-n_2} (i_2-n_2)! }{2t_1^{i_2-n_2+1}}\textbf{V}_{l}^{(n_{2},n_{3},\ldots,n_{k})}.
\eea
Here $\textbf{V}_{l}^{(n_{2},n_{3},\ldots,n_{k})}$ is given by the complicated explicit expression
\begin{align*}
&\textbf{V}_{l}^{(n_{2},n_{3},\ldots,n_{k})}
= \mathbf{B}^{(l)}N \Bigg( 
\sum_{m=3}^k (m-1) n_{m-1} \begin{pmatrix}
\textbf{L}_{l-1}^{(n_{2},n_{3},\ldots,n_{k})} \\
0
\end{pmatrix}
- \sum_{m=3}^k m n_{m}
\begin{pmatrix}
\textbf{L}_{l-1}^{(n_2,\ldots,n_{m-1}+1,n_m-1,\ldots,n_k)} \\
0
\end{pmatrix}\\
&+ k n_k \begin{pmatrix}
\textbf{L}_{l-1}^{(n_{2},n_{3},\ldots,n_{k})} \\
0
\end{pmatrix}
\Bigg) + (2t_1 \mathbf{Q}_1^{(l)} + N {\mathbf{Q}}_0^{(l)})
\textbf{L}_{l-1}^{(n_{2},n_{3},\ldots,n_{k})}
 + N \mathbf{Q}_2^{(l)} 
\textbf{L}_{l-2}^{(n_{2},n_{3},\ldots,n_{k})}\\
&+ \sum_{m=1}^{k-2} (m+1) n_{m+1} (-1)^m \mathbf{Q}_{m+2}^{(l)} 
\textbf{L}_{l+m}^{(n_2,\ldots,n_{m+1}-1,\ldots,n_k)} + (-1)^{k-1} k n_k \mathbf{Q}_{k+1}^{(l)}\textbf{L}_{l+k-1}^{(n_2,n_{3},\ldots,n_k-1)}\\
&- \sum_{m=1}^{k-2} (m+2) n_{m+2} (-1)^m \mathbf{Q}_{m+2}^{(l)}
\textbf{L}_{l+m}^{(n_2,\ldots,n_{m+2}-1,\ldots,n_k)}- \mathbf{B}^{(l)} \Bigg(-\frac{N}{2} \frac{\mathrm{d}}{\mathrm{d} t_1} + \frac{N}{2}\Bigg) \Bigg\{\sum_{m=0}^{k-3} \Bigg[(m+2) n_{m+2} (-1)^m \\ & 
\hat{\textbf{L}}_{l,m}^{(n_2,\ldots,n_{m+2}-1,\ldots,n_k)} 
- (m+3) n_{m+3} (-1)^m
\hat{\textbf{L}}_{l,m}^{(n_2,\ldots,n_{m+3}-1,\ldots,n_k)} \Bigg]
+(-1)^k k n_k \hat{\textbf{L}}_{l,k-2}^{(n_2,n_3,\ldots,n_{k}-1)}
\Bigg\} 
\\ &+\mathbf{B}^{(l)} N \sum_{m=4}^k \Bigg[(-1)^{m-1} (m-1) n_{m-1} \sum_{h=2}^{m-2} (-1)^h 
\hat{\textbf{L}}_{l,m-2-h}^{(n_2,\ldots,n_{h}+1,\ldots,n_{m-1}-1,\ldots,n_k)} 
\\
&-(-1)^{m-1} m n_{m} \sum_{h=2}^{m-2} (-1)^h 
\hat{\textbf{L}}_{l,m-2-h}^{(n_2,\ldots,n_{h}+1,\ldots,n_{m}-1,\ldots,n_k)}  \Bigg] 
 + \mathbf{B}^{(l)}N kn_k \sum_{h=2}^{k-1} (-1)^{k+h} 
\hat{\textbf{L}}_{l,k-1-h}^{(n_2,\ldots,n_{h}+1,\ldots,n_k-1)}.
\end{align*}
Furthermore, the following initial conditions are satisfied, where we drop dependence on $t_1$ in the notation, 
\begin{align}\label{initialconditionforthegeneralcase}
\mathcal{L}_{N,1,1}^{(n_2,\ldots,n_k)}
&= -\frac{N}{2} \frac{\mathrm{d}\mathcal{L}_{N}^{(n_2,\ldots,n_k)}}{\mathrm{d} t_1} + \frac{N\mathcal{L}_{N}^{(n_2,\ldots,n_k)}}{2}, \nonumber\\
\mathcal{L}_{N,2,1}^{(n_2,\ldots,n_k)}
&= \frac{N^2}{8} \frac{\mathrm{d}^2\mathcal{L}_{N}^{(n_2,\ldots,n_k)}}{\mathrm{d}t_1^2}  - \frac{N^2}{4} \frac{\mathrm{d}\mathcal{L}_{N}^{(n_2,\ldots,n_k)}}{\mathrm{d} t_1}  + \frac{N^2\mathcal{L}_{N}^{(n_2,\ldots,n_k)}}{8}   
+ \frac{N \mathcal{L}_{N}^{(n_2+1,\ldots,n_k)}}{2} , \nonumber\\
\mathcal{L}_{N,2,2}^{(n_2,\ldots,n_k)}
&= \frac{N^2}{8} \frac{\mathrm{d}^2 \mathcal{L}_{N}^{(n_2,\ldots,n_k)}}{\mathrm{d} t_1^2}  - \frac{N^2}{4} \frac{\mathrm{d}\mathcal{L}_{N}^{(n_2,\ldots,n_k)}}{\mathrm{d} t_1}  + \frac{N^2\mathcal{L}_{N}^{(n_2,\ldots,n_k)}}{8}   
- \frac{N\mathcal{L}_{N}^{(n_2+1,\ldots,n_k)}}{2}.
\end{align}
\end{prop}
\begin{proof}
    The result follows via a direct but long explicit computation where we differentiate both sides of the recursion in Proposition \ref{prop:recursivevector} and the initial condition \eqref{t1operateongeneralG} in Lemma \ref{lem:initialcond} with respect to $t_2,\ldots, t_k$ and let $t_2=\cdots=t_k=0$.
\end{proof}
Let us make a few comments on the specific features of the recursion above that will aid the inductive process in Section \ref{inductionsect}. Note that for fixed $l, n_2,\ldots, n_k$; all vectors $\textbf{L}_b^{(m_2,\ldots,m_k)}$ appearing therein can be classified as:
\begin{itemize}
    \item $m_j=n_j$ for $j=2,\ldots,k$, and $b<l$,
    \item $m_j\leq n_j$ for $j=2,\ldots, k$ with strict inequality for some $j$,
    \item There exist $2\leq j<i\leq k$ such that $m_j>n_j$, but $n_i>m_i$.
\end{itemize}
Our inductive assumption will be defined via the upper indices $(n_2,\ldots,n_k)$, in a way that for any tuple $(m_2,\ldots,m_k)$ given as in the second and third bullet points above, the corresponding vectors will have the desired properties for any $b$. Then, we will iterate over the sub-indices $l$, which will be possible due to the first bullet point above. For more details, see Section \ref{inductionsect}.\\

We also note, that in the above expression for $\textbf{V}_{l}^{(n_{2},n_{3},\ldots,n_{k})}$, certain terms include negative integers in their upper index. Even though these are not defined, this only happens when $n_j=0$ for some $j$, and this term is pre-multiplied by $n_j$, so that these terms disappear automatically. Of course, one can just define such terms to be $0$, and end up with the same statement, but noting that their coefficients are $0$ explains how these terms never \textit{appear} when one runs through the computation for an index for which this situation occurs.\\

As discussed earlier, our main goal is to obtain an invertible way to go between the $t_2,\ldots, t_q$-derivatives of $\Hfxn_N$ and $\Hfxn_{N,\boldsymbol{\lambda}}$. To write their relation in a more compact way, we now introduce the following functions:
\begin{align*}
  \mathsf{D}_{i}(t_{1},\ldots, t_{k})=N\left\{(N+s-1) \Hfxn_{N,\emptyset,i}-{\Hfxn}^{(w)}_{N,\emptyset,i} +(1-2N-2s)\Hfxn_{N,\emptyset, i+1}\nonumber
+{\Hfxn}^{(w)}_{N,\emptyset, i+1}+2\frac{t_1}{N}\Hfxn_{N,\emptyset,i+1}\right\}.
\end{align*}
where the parameter $s$ in the above is the same parameter that is used to define
$\Hfxn_{N,\blambda,h}$ and $\Hfxn^{(w)}_{N,\blambda,h}$.
As before, we will be considering the $t_2,\ldots, t_k$ derivatives of these functions evaluated at $t_2=\cdots=t_k=0$. Thus, we define
\begin{equation*}
    \mathsf{D}_{i}^{(n_{2},\ldots,n_{k})}(t_{1})
\defeq \frac{\partial^{n_{2}}}{\partial t_{2}^{n_{2}}}\cdots\frac{\partial^{n_{k}}}{\partial t_{k}^{n_{k}}}\mathsf{D}_{i}(t_{1},\ldots,t_{k})\Big|_{t_{2}=\cdots=t_{k}=0}.
\end{equation*}
Then, using the following identities, which can be proved directly through their definition,
\beas
&& \Hfxn_{N,\emptyset,h} = \sum_{j=1}^{h}(-1)^{j-1} \Hfxn_{N,\blambda_{h,j}},  \quad
\Hfxn_{N,\emptyset,h}^{(w)} = \sum_{j=1}^{h} (-1)^{j-1} (2N-2j+h) \Hfxn_{N,\blambda_{h,j}}, \quad \;\;\;\;\;\;\;\;\;\;\;(h\geq 1) , \\
&& \Hfxn_{N,\blambda,0} = N \Hfxn_{N,\blambda},   \quad \;\;\;\;\;\;\;\;\;\;\;\;\;\;\;
\Hfxn_{N,\blambda,0}^{(w)} = (N(N-1) + n) \Hfxn_{N,\blambda},
\quad \;\;\;\; (\blambda \text{ is a partition of } n) ,
\eeas
we have the expressions
\begin{align}\label{D0repinL}
\mathsf{D}_0^{(n_{2},\ldots,n_k)} &=
 {N^2s} \mathcal{L}_N^{(n_{2},\ldots,n_{k})} -2sN\mathcal{L}_{N,1,1}^{(n_{2},\ldots,n_{k})} +2t_{1}\mathcal{L}_{N,1,1}^{(n_{2},\ldots,n_{k})} \\
 &= (sN^2-Nt_1)  \frac{d \mathcal{L}_N^{(n_{2},\ldots,n_{k})}}{dt_1}
 + Nt_1\mathcal{L}_N^{(n_{2},\ldots,n_{k})}, \nonumber
\end{align}
and, for $j\geq 1$,
\begin{align}\label{formulaforDj}
    \mathsf{D}_j^{(n_{2},\ldots,n_k)} =\sum_{i=1}^{j+1}(-1)^{i-1}\Bigg( {(2-2s+j-2i)N} \mathcal{L}_{N,j+1,i}^{(n_{2},\ldots,n_{k})}+2t_{1}\mathcal{L}_{N,j+1,i}^{(n_{2},\ldots,n_k)}\Bigg)& \nonumber
    \\+\sum_{i=1}^{j}(-1)^{i-1} {(-N+s+2i-j-1)N} \mathcal{L}_{N,j,i}^{(n_{2},\ldots,n_k)}.&
\end{align}
\begin{lem}\label{generaliedexpandingofpartialderivatives}
Let $N\geq 1$. If $k=2$, then for any $m\in \mathbb{N}$, we have
\bea\label{special case k=2}
N \frac{\partial \Hfxn_{N}}{\partial t_{2}}
=\frac{\mathsf{D}_{0}}{2(t_{1}+t_{2})}+
\sum_{i=2}^{m}\frac{\mathsf{D}_{i-1}t_{2}^{i-1}}{2(t_{1}+t_{2})^{i}}+\frac{t_{2}^m}{(t_{1}+t_{2})^{m}}\Hfxn_{N,\emptyset,m+2},
\eea
If $k\geq 3$, then whenever $2\leq q\leq k$, $m\in \mathbb{N}$ we have:
\begin{align*}
&N\frac{\partial \Hfxn_{N}}{\partial t_{q}}=\sum_{i=2}^{q-1}\frac{t_{i}}{(t_1+t_2)^m}R_{i}(t_{1},\ldots,t_{k})+\frac{\mathsf{D}_{q-2}}{2(t_1+t_2)}+\sum_{i=2}^m \frac{1}{(2(t_1+t_2))^i} \sum_{j=0}^{i-1} \binom{i-1}{j} (kt_k)^j \nonumber \\
&
\sum_{\substack{h_{2}+\cdots+h_{k}=i-1-j\\h_{2}=\cdots=h_{q-1}=0\\h_{q}\geq 0,\ldots,h_{k}\geq 0}}\prod_{n=q}^{k}(nt_{n})^{h_{n}} 
\Bigg(\sum_{\substack{h_3' \leq h_3 \\ \ldots \\ h_{k-1}'\leq h_{k-1}}}a_{h_{2},h_{3}',\ldots,h_{k-1}',h_{k}}^{(i,j)}\mathsf{D}_{q+\sum_{n=1}^{k-2}nh_{n+1}+(k-2)h_{k}-\sum_{n=3}^{k-1}h_{n}'+(k-1)j-2}\Bigg) + \frac{1}{(2(t_1+t_2))^m} \nonumber \\
 & \;\;\;\; \sum_{j=0}^{m} \binom{m}{j} (kt_k)^j
\sum_{\substack{h_{2}+\cdots+h_{k}=m-j\\h_{2}=\cdots=h_{q-1}=0\\h_{q}\geq 0,\ldots,h_{k}\geq 0}}\prod_{n=q}^{k}(nt_{n})^{h_{n}}\Bigg(\sum_{\substack{h_3' \leq h_3 \\ \ldots \\ h_{k-1}'\leq h_{k-1}}}a_{h_{2},h_{3}',\ldots,h_{k-1}',h_{k}}^{(m+1,j)}\Hfxn_{N,\emptyset, q+\sum_{n=1}^{k-2}nh_{n+1}+(k-2)h_{k}-\sum_{n=3}^{k-1}h_{n}'+(k-1)j}\Bigg), 
\end{align*}
where $R_{2},\ldots,R_{q-1}$ are smooth functions of $t_{1},\ldots,t_{k}$ whose explicit forms are not relevant in what follows, and 
\bea\label{coefficientsbeforegeneralD}
&&a_{h_{2},h_{3}',\ldots,h_{k-1}',h_{k}}^{(i,j)}=\frac{(i-1-j)!(-1)^{h_{3}'+\cdots+h_{k-1}'+h_{k}}}{h_2! h_k! \prod_{n=3}^{k-1}(h_{n}-h_{n}')!h_n'!}.
\eea
\end{lem}

\begin{proof}
We set $\boldsymbol{\lambda}=\emptyset$ and iteratively replace $\Hfxn_{N,\boldsymbol{\lambda},l+q-2}$ for $l=3,\ldots,k+1$ via formula (\ref{generalizationiterationformula}). Then also making use of Lemma \ref{generalderivative and translations}, we obtain for $q=2,\ldots,k$ and for $m\geq 1$,
\beas
N \frac{\partial \Hfxn_{N}}{\partial t_{q}}&=& \Hfxn_{N,\emptyset,q}\\
&=&\frac{\mathsf{D}_{q-2}}{2(t_1+t_2)}
+\sum_{i=2}^m \frac{1}{(2(t_1+t_2))^i} \sum_{j=0}^{i-1} \binom{i-1}{j} (kt_k)^j\\ &&\sum_{l_1=1}^{k-2} \sum_{l_2=1}^{k-2} \cdots \sum_{l_{i-1-j}=1}^{k-2}\Big(\prod_{n=1}^{i-1-j}\big((l_{n}+1)t_{l_{n}+1}-(l_{n}+2)t_{l_{n}+2}\big)\Big)\mathsf{D}_{q+l_1+\cdots+l_{i-1-j}+(k-1)j-2} \\
&& + \frac{1}{(2(t_1+t_2))^m}  \sum_{j=0}^{m} \binom{m}{j} (kt_k)^j\\
&&\sum_{l_1=1}^{k-2}\sum_{l_2=1}^{k-2}\cdots \sum_{l_{m-j}=1}^{k-2} \Big(\prod_{n=1}^{m-j}\big((l_{n}+1)t_{l_{n}+1}-(l_{n}+2)t_{l_{n}+2}\big)\Big)\Hfxn_{N,\emptyset,q+l_1+\cdots+l_{m-j}+(k-1)j }.
\eeas
Here, we use the convention that when $j=m$,
\beas
\sum_{l_1=1}^{k-2}\sum_{l_2=1}^{k-2}\cdots \sum_{l_{m-j}=1}^{k-2} \Big(\prod_{n=1}^{m-j}\big((l_{n}+1)t_{l_{n}+1}-(l_{n}+2)t_{l_{n}+2}\big)\Big)\Hfxn_{N,\emptyset, q+l_1+\cdots+l_{m-j}+(k-1)j}=\Hfxn_{N,\emptyset, q+(k-1)m},
\eeas
and when $j\leq m-1$ and $k=2$,
\beas
\sum_{l_1=1}^{k-2}\sum_{l_2=1}^{k-2}\cdots \sum_{l_{m-j}=1}^{k-2} \prod_{n=1}^{m-j}\big((l_{n}+1)t_{l_{n}+1}-(l_{n}+2)t_{l_{n}+2}\big)\Hfxn_{N,\emptyset, q+l_1+\cdots+l_{m-j}+(k-1)j}=0.
\eeas
We also use a similar convention for the second summand in terms of $\mathsf{D}_{b}$'s. Then, it is not hard to check that
\begin{align}
&\sum_{l_1=1}^{k-2} \sum_{l_2=1}^{k-2} \cdots \sum_{l_{i-1-j}=1}^{k-2}\Big(\prod_{n=1}^{i-1-j}\big((l_{n}+1)t_{l_{n}+1}-(l_{n}+2)t_{l_{n}+2}\big)\Big)\mathsf{D}_{q+l_1+\cdots+l_{i-1-j}+(k-1)j-2}\nonumber\\
&=\sum_{\substack{h_{2}+\cdots+h_{k}=i-1-j\\h_{2}\geq 0,\ldots,h_{k}\geq 0}}\prod_{n=2}^{k}(nt_{n})^{h_{n}}\Bigg(\sum_{\substack{h_3' \leq h_3 \\ \ldots \\ h_{k-1}'\leq h_{k-1}}}a_{h_{2},h_{3}',\ldots,h_{k-1}',h_{k}}^{(i,j)}\mathsf{D}_{q+\sum_{n=1}^{k-2}nh_{n+1}+(k-2)h_{k}-\sum_{n=3}^{k-1}h_{n}'+(k-1)j-2}\Bigg)\nonumber,
\end{align}
and also that
\begin{align*}
&\sum_{l_1=1}^{k-2}\sum_{l_2=1}^{k-2}\cdots \sum_{l_{m-j}=1}^{k-2} \Big(\prod_{n=1}^{m-j}\big((l_{n}+1)t_{l_{n}+1}-(l_{n}+2)t_{l_{n}+2}\big)\Big)\Hfxn_{N,\emptyset,q+l_1+\cdots+l_{m-j}+(k-1)j}\\
&=\sum_{\substack{h_{2}+\cdots+h_{k}=m-j\\h_{2}\geq 0,\ldots,h_{k}\geq 0}}\prod_{n=2}^{k}(nt_{n})^{h_{n}}\Bigg(\sum_{\substack{h_3' \leq h_3 \\ \ldots \\ h_{k-1}'\leq h_{k-1}}}a_{h_{2},h_{3}',\ldots,h_{k-1}',h_{k}}^{(m+1,j)}\Hfxn_{N,\emptyset, q+\sum_{n=1}^{k-2}nh_{n+1}+(k-2)h_{k}-\sum_{n=3}^{k-1}h_{n}'+(k-1)j}\Bigg),
\end{align*}
where $a_{h_{2},h_{3}',\ldots,h_{k-1}',h_{k}}^{(i,j)}$ is given as in (\ref{coefficientsbeforegeneralD}).
The above two equalities do not exist when $k=2$, so we state the case $k=2$ separately. In addition, we move all terms appearing in the above two equalities, except when $h_2=0,\ldots,h_{q-1}=0$, into $R_i(t_1,\ldots,t_k)$ with $i=2,\ldots,q-1$. Combining these observations, the result follows.
\end{proof}
As in for the previous relations, we will take appropriate $t_2,\ldots,t_k$-derivatives of the above result. But first, let us note that writing the $t_q$ derivative of $ \Hfxn_{N}$ in the specific way above, as linear combinations of $\mathsf{D}_j$ and $\Hfxn_{N,\emptyset,l}$ with coefficients that are powers of $t_n$ for $n=q,\ldots,k$; allows one to have control over which of these terms appear with which upper indices after suitable derivatives are evaluated at $t_2=\ldots=t_k=0$. In particular, when one evaluates these derivatives after choosing $m$ suitably in the above result, only the $\mathsf{D}_j$'s with upper indices having desirable features will remain. Namely, during the inductive process, the upper and lower indices on the remaining terms will help ensure the specific degree conditions on the polynomials $P_{m}^{(s,N)}(t_1)$ given in Theorem \ref{structureforgeneralfinitesize}.
\begin{prop}\label{generalexpressionformathcalFn2nk}
Let $a_{h_{2},h_{3}',\ldots,h_{k-1}',h_{k}}^{(i,j)}$ be given as in (\ref{coefficientsbeforegeneralD}). Then we have,
for $k=2$, $n_2\geq 1$, that
\bea\label{expressionsforfi2special}
\mathcal{L}_{N}^{(n_{2})}(t_{1})=\frac{1}{2N} \sum_{i=0}^{n_{2}-1} (n_{2}-1)! \sum_{j=0}^{n_{2}-1-i} \binom{n_2-1-j}{i} \frac{(-1)^{n_{2}-1-i-j}}{j!}\frac{\mathsf{D}_{i}^{(j)}(t_{1})}{t_{1}^{n_{2}-j}} ,
\eea
and for $k\geq 3$, $n_2\geq 1$, that
\begin{multline}\label{prop415first}
    \mathcal{L}_{N}^{(n_2,\ldots,n_k)}(t_1)=\frac{1}{2N} \sum_{i_{2}=0}^{n_2-1} (-1)^{n_2-1-i_{2}}\frac{(n_{2}-1)!}{i_{2}!}
\frac{\mathsf{D}_0^{(i_{2},n_3,\ldots,n_k)}(t_{1})}{t_1^{n_2-i_{2}}}  \\ 
+  \sum_{i=1}^{n_2+\cdots+n_k-1} \frac{1}{2^{i+1} N}\sum_{j=0}^{i} \binom{i}{j} 
\sum_{h_2+\cdots+h_k=i-j}  \binom{n_2-1}{h_2} \binom{n_k}{h_k+j} \binom{n_3}{h_3} \ldots \binom{n_{k-1}}{h_{k-1}} (h_{k}+j)!\\k^{j+h_{k}}
\prod_{n=2}^{k-1}n^{h_{n}}h_{n}! \sum_{i_{2}=0}^{n_2-1-h_2}\binom{n_2-1-h_2}{i_{2}} \frac{(i+n_2-1-h_2-i_{2})! (-1)^{n_2-1-h_2-i_{2}}}{t_1^{i+n_2-h_2-i_{2}} i!} \\
 \sum_{\substack{h_3' \le h_3 \\ \ldots \\ h_{k-1}' \le h_{k-1}}} a_{h_2,h_3',\ldots,h_{k-1}',h_{k}}^{(i+1,j)} \mathsf{D}_{\sum_{n=1}^{k-2}nh_{n+1}+(k-2)h_{k}-\sum_{n=3}^{k-1}h_{n}'+(k-1)j}^{(i_{2},n_3-h_3,\ldots,n_{k-1}-h_{k-1}, n_k-j-h_k)}(t_{1}).
\end{multline}
When $n_{2}=\cdots=n_{q-1}=0$ and $n_{q}\geq 1$ for some $3 \leq q \leq k-1$, we have:
\begin{align*}
\mathcal{L}_N^{(0,\ldots,0,n_{q},\ldots, n_k)}(t_{1})
= &\frac{\mathsf{D}_{q-2}^{(0,\ldots,0,n_q-1,\ldots,n_k)}(t_{1})}{2Nt_1}
+\sum_{i=1}^{n_q+\cdots+n_k-1} \frac{1}{(2t_1)^{i+1}N}\sum_{j=0}^i \binom{i}{j} \nonumber\\
& \sum_{h_q+\cdots+h_k=i-j} \binom{n_q-1}{h_q}\binom{n_k}{h_k+j} \binom{n_{q+1}}{h_{q+1}} \cdots \binom{n_{k-1}}{h_{k-1}} k^{j+h_{k}}(h_{k}+j)!\prod_{n=q}^{k-1}n^{h_{n}}h_{n}! \nonumber\\
& \;\;\;\;\;\;\;\;\;\;\;\;\; \sum_{\substack{h_q'\le h_q \\ \cdots \\ h_{k-1}' \le h_{k-1}}} a_{0,\ldots,0,h_q',\ldots,h_{k-1}',h_{k}}^{(i+1,j)} \mathsf{D}_{q-2+\sum_{n=q-1}^{k-2}nh_{n+1}+(k-2)h_{k}-\sum_{n=q}^{k-1}h_{n}'+(k-1)j}^{(0,\ldots,0,n_q-1-h_q,n_{q+1}-h_{q+1},\ldots,n_{k-1}-h_{k-1}, n_k-j-h_k)}(t_{1}). 
\end{align*}
When $n_{2}=\ldots=n_{k-1}=0$ and $n_{k}\geq 1$, we have
\begin{align*}
\mathcal{L}_N^{(0,\ldots,0, n_k)}(t_{1})= \frac{\mathsf{D}_{k-2}^{(0,\ldots,0,n_k-1)}(t_{1})}{2Nt_1}
+\sum_{i=1}^{n_k-1}\frac{k^i}{(2t_{1})^{i+1}N}\frac{(n_{k}-1)!}{(n_{k}-1-i)!}
\sum_{j=0}^i \binom{i}{j} \: a_{0,\ldots,0,i-j}^{(i+1,j)}\mathsf{D}_{k-2+(k-1)i-(i-j)}^{(0,\ldots,0,n_{k}-1-i)}.
\end{align*}
\end{prop}
\begin{proof}
The result follows, in each case, by first using Lemma \ref{generaliedexpandingofpartialderivatives} with $m=n_q+\cdots+n_k$, and $q$ being the smallest integer so that $n_q>0$; and then evaluating the appropriate derivatives at $t_2=\cdots=t_k=0$.
\end{proof}

\subsection{Proofs of Theorems \ref{structureforgeneralfinitesize} \& \ref{painlevethm}} \label{inductionsect}

We will use induction on $k$, the number of variables $t_{2},\ldots,t_{k}$, to prove Theorem \ref{structureforgeneralfinitesize}. But before doing so, we will require some preliminary lemmas to set up the induction. To this end, let us first fix some notation to make the presentation clearer. Note that our end goal amounts to showing
\begin{equation}\label{defofmathfrakS}
    \mathcal{L}_{N}^{(n_2,\ldots, n_k)}(t_{1})
=\frac{1}{t_{1}^{\sum_{q=2}^{k}qn_{q}-1}}
\sum_{m=0}^{\sum_{q=2}^{k}(q-1)n_{q}}
t_1^{m-1}
P_{m}^{(s,N)}(t_{1})\frac{\mathrm{d}^{m}}{\mathrm{d}t_{1}^{m}}\Hfxn_{N}(t_{1},0,\ldots,0),
\end{equation}
for polynomials $P_{m}^{(s,N)}$ as described in the statement of Theorem \ref{structureforgeneralfinitesize} when $t_1\geq 0$. 
Note that the left hand-side of above equation is continuous at $t_{1}=0$, so the right-hand side has a limit when $t_{1}=0$. Thus, it suffices to prove \eqref{defofmathfrakS} holds for $t_1>0$. 

The following definition is then crucial for our inductive argument.

\begin{defn} Let $k\geq 2$ be an integer.
We define $\mathfrak{S}^{(k)} \subset (\mathbb{N}\cup \{0\})^{k-1}$ to be the set consisting of
$(0,\ldots,0)$ and all $(n_2,\ldots,n_k)$ such that \eqref{defofmathfrakS} holds for appropriate polynomials $P_{m}^{(s,N)}$ as described in the statement of Theorem \ref{structureforgeneralfinitesize}. Similarly, we define $\mathfrak{S}_{l,j}^{(k)}$ to be the set consisting of all $(n_2,\ldots, n_k) $ such that
\begin{equation*}
\mathcal{L}_{N,l,j}^{(n_{2},n_{3},\ldots,n_{k})}(t_{1})=\frac{1}{t_{1}^{l+\sum_{q=2}^{k}qn_{q}-1}}
\sum_{m=0}^{l+\sum_{q=2}^{k}(q-1)n_{q}} t_1^{m-1} P_{l,j,m}^{(s,N)}(t_{1})\frac{\mathrm{d}^{m}}{\mathrm{d}t_{1}^{m}}\Hfxn_{N}(t_{1},0,\ldots,0)
\end{equation*}
holds, where $P_{l,j,m}^{(s,N)}(t_{1})$, $m\geq 1$, and $t_{1}^{-1}P_{l,j,0}^{(s,N)}(t_{1})$ are polynomials of $t_{1}$ of degree at most $l+\sum_{q=2}^{k}qn_{q}-m$ and $l+\sum_{q=2}^{k}qn_{q}-1$, respectively, and the coefficients of these polynomials are polynomials in $N$ and $s$ with degrees at most $l+\sum_{q=2}^{k}(q-1)n_{q}$ and $l+\sum_{q=2}^{k}(q-1)n_{q}-1$ respectively. Finally, we define the intersection of all these sets:
\begin{equation}
\widetilde{\mathfrak{S}}^{(k)}
\overset{\textnormal{def}}{=} \bigcap_{1\leq j\leq l}\mathfrak{S}_{l,j}^{(k)} \cap \mathfrak{S}^{(k)}.
\end{equation}  
\end{defn}

We can now begin the proof.

\begin{lem}\label{lemmaforinduction}
Let $k\geq 3$ be an integer.
Suppose $(n_3,\ldots, n_k)\in (\mathbb{N}\cup\{0\})^{k-2}\backslash {(0,\ldots,0)}$ is so that for any $j=3,\ldots,k$, with $n_j\geq 1$,
  \begin{equation*}
      (m_2,m_3,\ldots,m_{j-1}, m_j, n_{j+1},\ldots,n_k)\in \widetilde{\mathfrak{S}}^{(k)}
  \end{equation*}
  for all $m_2, m_3,\ldots, m_{j-1}\in \mathbb{N}\cup\{0\}$ and $0\leq m_{j}\leq n_{j}-1$. Then, we also have
   \begin{equation}\label{inductionlemconclusion}
   (n_2,\ldots,n_k) \in \widetilde{\mathfrak{S}}^{(k)}, 
   \end{equation}
   for all $n_2\in \mathbb{N}\cup\{0\}$.
\end{lem}
\begin{proof}
We will prove the lemma by inducting on $n_2$. We start by noting that whenever $(n_3,\ldots, n_k)$ satisfies the assumptions of the lemma, $\mathsf{D}_j^{(m_2,\ldots,m_{j},n_{j+1},\ldots,n_k)}$ is a polynomial in $N$ for all $j\geq0$, for any such  $(m_2,\ldots,m_{j},n_{j+1},\ldots,n_k)$ as in the statement of the lemma, by virtue of \eqref{D0repinL} and \eqref{formulaforDj}. In particular, whenever this is the case, we have that
\begin{equation}\label{degreeconditiononDj}
\mathsf{deg}_N\left(\mathsf{D}_j^{(m_2,\ldots,m_{j},n_{j+1},\ldots,n_k)}\right)\leq j+2+\sum_{q=2}^j(q-1)m_q+\sum_{q=j+1}^{k}(q-1)n_{q}
\end{equation}
where $\mathsf{deg}_N(\cdot)$ denotes the degree of a polynomial in $N$.
Next, we start our induction by showing that \eqref{inductionlemconclusion} holds in the base case $n_2=0$. Indeed, in this case, by means of Proposition \ref{generalexpressionformathcalFn2nk}, and also \eqref{degreeconditiononDj}, we have that under the assumptions of the lemma,
\begin{equation}
\label{0811eq1}
(0,n_3,\ldots, n_k)\in \mathfrak{S}^{(k)}.
\end{equation}
where we note that the pre-factor $N^{-1}$ appearing in Proposition \ref{generalexpressionformathcalFn2nk} guarantees that the degree requirement we imposed previously is satisfied. Next, by \eqref{initialconditionforthegeneralcase},  we see that
\begin{equation}
\label{0811eq2}
(0,n_3,\ldots, n_k)\in  \mathfrak{S}_{1,1}^{(k)}.
\end{equation}
Hence, combining \eqref{D0repinL}, \eqref{formulaforDj}, \eqref{0811eq1}, \eqref{0811eq2}, and applying \eqref{prop415first}, it is clear that
\begin{equation*}
(1,n_3,\ldots,n_k) \in \mathfrak{S}^{(k)},
\end{equation*}
provided that for all integers $h_2,h_3,\ldots,h_k,j$ appearing on the RHS of \eqref{prop415first}, one has
\begin{equation*}
    (0,n_3-h_3,\ldots,n_{k-1}-h_{k-1}, n_k-j-h_k) \in \mathfrak{S}_{l,i}^{(k)}
\cap \mathfrak{S}_{l+1,i'}^{(k)}
\end{equation*}
where $1\leq i\leq l$, $1\leq i'\leq l+1$, and
\begin{equation*}
    l=\sum_{n=2}^{k-2} nh_{n+1} + (k-2) h_k - \sum_{n=3}^{k-1}h_n' + (k-1)j
\end{equation*}
for integers $h_n'\leq h_n$. Then, note that in case $n_{2}=1$, we must have $h_2=0$, so that one cannot have $h_3=\cdots=h_{k-1}=h_k=j=0$. Thus, by the assumption of the lemma, we have
\begin{equation*}
    (0,n_3-h_3,\ldots,n_{k-1}-h_{k-1}, n_k-j-h_k) \in \widetilde{\mathfrak{S}}^{(k)} \subseteq \mathfrak{S}_{l,j}^{(k)}
\cap \mathfrak{S}_{l+1,j}^{(k)}
\end{equation*}
Here, once again, we use the bound
\begin{equation*}
   \mathsf{deg}_N \left(\mathsf{D}_{\sum_{n=1}^{k-2}nh_{n+1}+(k-2)h_{k}-\sum_{n=3}^{k-1}h_{n}'+(k-1)j}^{(0,n_3-h_3,\ldots,n_{k-1}-h_{k-1}, n_k-j-h_k)}\right) \leq 2+\sum_{q=3}^k (q-1)n_q.
\end{equation*}
for the degree in $N$, where, as previously, we note the pre-factor $N^{-1}$ appearing in \eqref{prop415first}. From this, it follows, again by means of \eqref{initialconditionforthegeneralcase} that
\begin{equation*}
(1,n_3,\ldots,n_k) \in \mathfrak{S}_{1,1}^{(k)}
\end{equation*}
and also
\begin{equation*}
(0,n_3,\ldots, n_k)\in \mathfrak{S}_{2,1}^{(k)} \cap \mathfrak{S}_{2,2}^{(k)}.
\end{equation*}
Furthermore, applying the recursive relation in \eqref{generalrecursiveconditionformathcalG} iteratively, while using the assumption of the lemma, we see that
\begin{equation*}
(0,n_3,\ldots, n_k)\in \widetilde{\mathfrak{S}}^{(k)},
\end{equation*}
where we crucially use the fact that out of the matrices that appear in equation \eqref{generalrecursiveconditionformathcalG}, only $\mathbf{Q}_2^{(l)}$ has entries that depend on $N$, whose entries are $O(N)$, and that it only appears as a multiplicant of $\mathbf{L}_{l-2}^{(n_2,\ldots,n_k)}$. This completes the proof of the base case. \\

\noindent For the inductive step, assume \eqref{inductionlemconclusion} holds for all $n_2\leq a_2-1$ for some $a_{2}\geq 1$. Then, observe that by means of Proposition \ref{generalexpressionformathcalFn2nk}, \eqref{initialconditionforthegeneralcase} and the assumption of this lemma, we immediately have that
\begin{equation}
\label{0811eq3}
(a_2,n_3,\ldots, n_k)\in \mathfrak{S}^{(k)}\cap \mathfrak{S}_{1,1}^{(k)}.
\end{equation}
Then, we claim that applying Proposition \ref{generalexpressionformathcalFn2nk}, \eqref{initialconditionforthegeneralcase} and \eqref{0811eq3} give
that,
\begin{equation}\label{smallclaim}
    (a_2+1,n_3,\ldots, n_k)\in \mathfrak{S}^{(k)}\cap \mathfrak{S}_{1,1}^{(k)}.
\end{equation}
Indeed, the integer tuples that appear in the upper index of $\mathsf{D}$ in the second term on the right-hand side of \eqref{prop415first} can be classified as:
\begin{itemize}
    \item $(n_2,n_3,\ldots,n_k)$ with $n_2\leq a_2$,
    \item $(m_2,m_3,\ldots,m_k)$ such that there exists $j\in \{3,\ldots,k\}$ with $m_j\leq n_j-1$, $m_l= n_l$ for all $l=j+1,\ldots,k$.

\end{itemize}
But then observe that $(a_2,n_3,\ldots,n_k)$ does not appear in the upper index of $\mathsf{D}$ in the second sum on the right-hand side of \eqref{prop415first} since this is only possible for the terms corresponding to $h_2=\cdots=h_k=j=0$, in which case we must have
$i=0$ in \eqref{prop415first}.
So either $n_2\leq a_2-1$ or the second bullet point above applies. Hence, first combining all these observations with \eqref{D0repinL}, \eqref{formulaforDj}, and then using \eqref{initialconditionforthegeneralcase} we obtain \eqref{smallclaim}. Thus, using \eqref{initialconditionforthegeneralcase} again, one can then see,
\begin{equation*}
    (a_2,n_3,\ldots,n_k)\in \mathfrak{S}^{(k)} \cap \mathfrak{S}_{1,1}^{(k)}\cap \mathfrak{S}_{2,1}^{(k)}\cap \mathfrak{S}_{2,2}^{(k)}.
\end{equation*}
Now, using the above and the recursive relation \eqref{generalrecursiveconditionformathcalG} iteratively as before, while using the inductive assumption, one sees that
\begin{equation*}
(a_2,n_3,\ldots, n_k)\in \bigcap_{1\leq m\leq l} \mathfrak{S}_{l,m}^{(k)},
\end{equation*}
which completes the inductive step and hence the proof of the result.
\end{proof}

\begin{rmk}
In the proof of the above lemma, we only discussed the condition on the degree of $N$ for $P_{m}^{(s,N)}(t_{1})$ and $P_{l,j,m}^{(s,N)}(t_{1})$. However, all other conditions can be verified similarly.
\end{rmk}

\begin{proof}[Proof of Theorem \ref{structureforgeneralfinitesize}]
To prove the theorem, we will show that for any integer $k\geq 2$,
\begin{equation*}
\widetilde{\mathfrak{S}}^{(k)} = (\mathbb{N} \cup\{0\})^{k-1}.
\end{equation*}
This will be done by inducting on $k$. First,
we can use a similar argument to the proof of 
Lemma \ref{lemmaforinduction} to prove the case $k=2$ by replacing \eqref{prop415first} with \eqref{expressionsforfi2special} whenever it appears.
Suppose we have proved the case $k-1$ for some $k\geq 3$, i.e.,
$\widetilde{\mathfrak{S}}^{(k-1)} = (\mathbb{N} \cup\{0\})^{k-2}$. Since $\mathcal{L}_N^{(n_2,\ldots,n_{k-1},0)} = \mathcal{L}_N^{(n_2,\ldots,n_{k-1})}$ and
$\mathcal{L}_{N,l,j}^{(n_2,\ldots,n_{k-1},0)} = \mathcal{L}_{N,l,j}^{(n_2,\ldots,n_{k-1})}$, we have that $(n_2,\ldots,n_{k-1},0)\in \widetilde{\mathfrak{S}}^{(k)}$. Hence, for the sake of induction, suppose
\begin{equation}\label{inductionref1}
(n_2,\ldots, n_k)\in \widetilde{\mathfrak{S}}^{(k)} 
\end{equation}
for any $n_2,\ldots, n_{k-1}\in \mathbb{N} \cup \{0\}$ and $n_k\leq a_{k}-1$ for some $a_k\geq 1$.
Note that combining this assumption with Lemma \ref{lemmaforinduction}, we have that
\begin{equation*}
(n_2,0,\ldots,0,a_k)\in \widetilde{\mathfrak{S}}^{(k)} 
\end{equation*}
for all $n_2
\in \mathbb{N} \cup \{0\}$.
Next, observe that by applying iteratively Lemma \ref{lemmaforinduction} on the $n_3$-variable, one gets,
\begin{equation}\label{inductionref2}
(n_2,n_3,0,\ldots,0, a_k) \in \widetilde{\mathfrak{S}}^{(k)} 
\end{equation}
for all $n_2,n_3\in \mathbb{N} \cup \{0\}$. This, in turn, implies, by virtue of Lemma \ref{lemmaforinduction}, \eqref{inductionref1} and \eqref{inductionref2}, that
\begin{equation*}
    (n_2,0,1,0,\ldots,0,a_k) \in \widetilde{\mathfrak{S}}^{(k)} .
\end{equation*}
Hence, again applying the induction and Lemma \ref{lemmaforinduction} on the $n_3$-variable, we see that
\begin{equation*}
    (n_2,n_3,1,0,\ldots,0,a_k) \in \widetilde{\mathfrak{S}}^{(k)} 
\end{equation*}
for any $n_2,n_3\in \mathbb{N} \cup \{0\}$, which then implies
\begin{equation*}
     (n_2,0,2,0,\ldots,0,a_k) \in \widetilde{\mathfrak{S}}^{(k)} .
\end{equation*}
Continuing this, where we increase the $n_4$-variable by $1$ at each step, we conclude that
\begin{equation*}
    (n_2,n_3,n_4,0,\ldots,0,a_k)\in \widetilde{\mathfrak{S}}^{(k)}
\end{equation*}
for any $n_2,n_3,n_4\in \mathbb{N} \cup\{0\}$. This process can then be continued to conclude the inductive step that was set up in \eqref{inductionref1}, which is best described via the following figure:
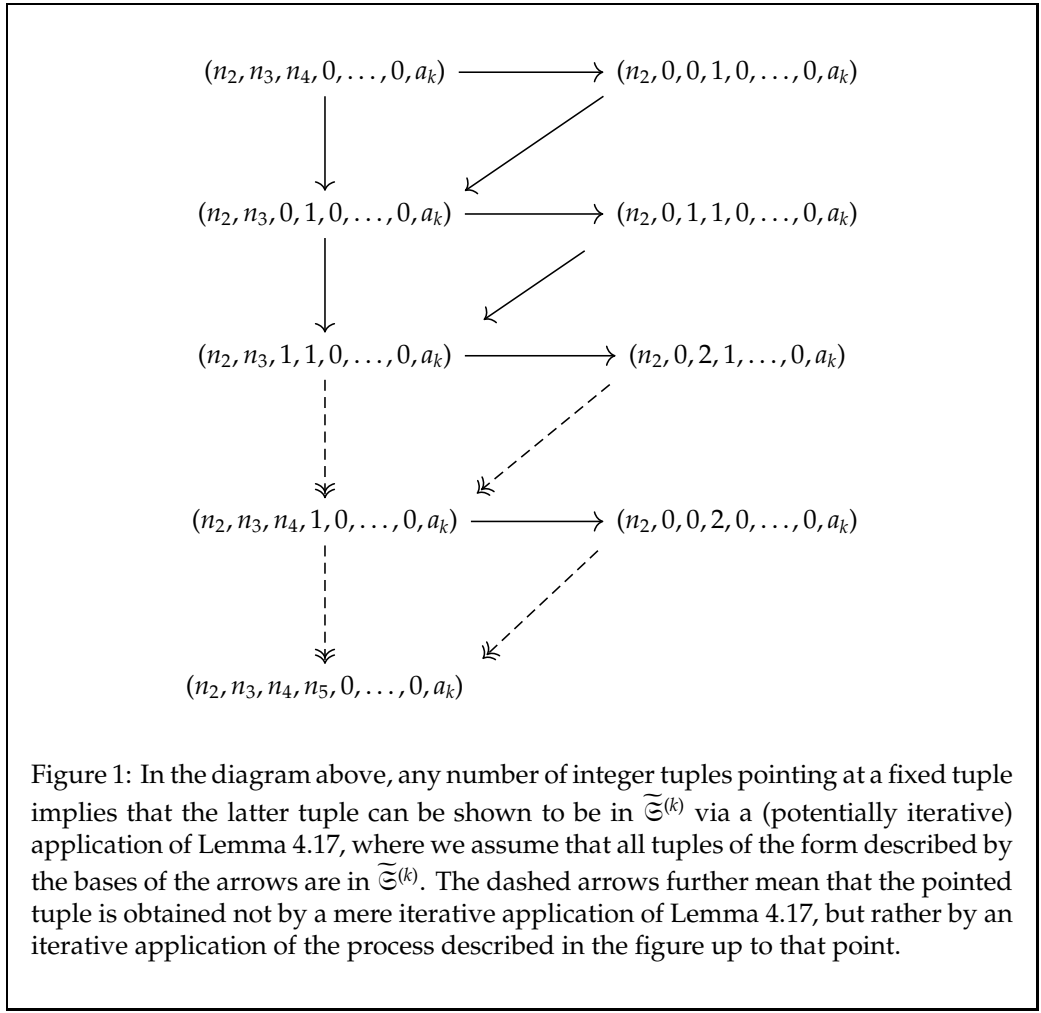
\begin{framed}
\begin{center}
\begin{tikzcd}[row sep=small, column sep =huge]

 (n_2,n_3,n_4,0,\ldots,0,a_k) \arrow[r] \arrow[dddd]  & (n_2,0,0,1,0,\ldots,0,a_k)  \arrow[ldddd, start anchor=south west, end anchor=north east]\\
                   \\
                   \\
                   \\
(n_2,n_3,0,1,0,\ldots,0,a_k)  \arrow[r] \arrow[dddd]   &   (n_2,0,1,1,0,\ldots,0,a_k) \arrow[ldddd, shorten=3mm, start anchor=south west, end anchor=north east]         \\
                                     \\
                                     \\
                                     \\
(n_2,n_3,1,1,0,\ldots,0,a_k) \arrow[dashed, two heads]{ddddd} \arrow[r] & (n_2,0,2,1,\ldots,0,a_k) \arrow[dashed, two heads, shorten=1mm, start anchor=south west, end anchor=north east]{lddddd} \\
                         \\
                         \\
                         \\
                         \\

(n_2,n_3,n_4,1,0,\ldots,0,a_k) \arrow[dashed, two heads]{ddddd} \arrow[r]  & (n_2,0,0,2,0,\ldots,0,a_k) \arrow[dashed, two heads, shorten=1mm, start anchor=south west, end anchor=north east]{lddddd} \\
                                          \\
                                          \\
                                          \\
                                          \\
(n_2,n_3,n_4,n_5,0,\ldots,0,a_k) &  \\
\end{tikzcd}
\captionof{figure}{In the diagram above, any number of integer tuples pointing at a fixed tuple implies that the latter tuple can be shown to be in $\widetilde{\mathfrak{S}}^{(k)}$ via a (potentially iterative) application of Lemma \ref{lemmaforinduction}, where we assume that all tuples of the form described by the bases of the arrows are in $\widetilde{\mathfrak{S}}^{(k)}$. The dashed arrows further mean that the pointed tuple is obtained not by a mere iterative application of Lemma \ref{lemmaforinduction}, but rather by an iterative application of the process described in the figure up to that point.}
\end{center}
\end{framed}
\noindent Continuing this process, we reach at its conclusion that
\begin{equation*}
    (n_2,n_3,n_4,n_5,\ldots,n_{k-1},a_k) \in \widetilde{\mathfrak{S}}^{(k)},
\end{equation*}
which completes the inductive step set up in \eqref{inductionref1} and hence the proof of the theorem.
\end{proof}
\begin{proof}[Proof of Theorem \ref{painlevethm}]
First, note that by an application of Lemma \ref{reductiontoints}, the differential equality \eqref{general structure1} holds for any $s\in \mathbb{R}$ with $s>2^{-1}(\sum_{q=2}^k qn_q-1)$, so that it now suffices to prove that for all $t_1 \geq 0$: 
 \begin{align}\label{eq:convergenceforpainleve}
       \frac{1}{N^{\sum_{m=2}^k mn_m}}\mathbb{E}_N^{(s)}\left[e^{-\textnormal{i}\frac{t_1}{N}\sum_{j=1}^N \mathsf{x}^{(N)}_j} \prod_{m=2}^k\left(\sum_{j=1}^N (\mathsf{x}^{(N)}_j-\textnormal{i})^m\right)^{n_m}\right] &\xrightarrow[]{N\to \infty} \mathbb{E}\left[e^{-\textnormal{i}t_1 \mathsf{q}_1(s)} 
        \prod_{m=2}^k \left(\mathsf{q}_m(s)\right)^{n_m}
        \right], \\
          \frac{\mathrm{d}^m}{\mathrm{d}t_1^m} \mathbb{E}_N^{(s)}\left[e^{-\textnormal{i}t_1\sum_{j=1}^N \frac{\mathsf{x}^{(N)}_j}{N}}\right] &\xrightarrow[]{N\to \infty} \frac{\mathrm{d}^m}{\mathrm{d}t^m} \mathbb{E}\left[e^{-\textnormal{i}t_1\mathsf{q}_1(s)}\right]. \nonumber
    \end{align}
    The second limit follows immediately, for any $m< 2s
+1$, from our results in Section \ref{Convergencesection}. To prove the first one, on the other hand, note that one can write the power sums in terms of a combination of smaller power sums and elementary symmetric polynomials via Newton's formula:
\begin{equation*}
       \mathrm{p}_m(x_1,x_2,\ldots,x_n) = (-1)^{m-1}m\mathrm{e}_m(x_1,\ldots,x_n)+\sum_{j=1}^{m-1}(-1)^{m-1+j} \mathrm{e}_{m - j} (x_1, \ldots, x_n) \mathrm{p}_j(x_1, \ldots, x_n).
\end{equation*}
    Hence, using the convergence results from Section \ref{Convergencesection} for $\frac{1}{N^k}\mathrm{e}_k(\mathsf{x}_1^{(N)}, \mathsf{x}_2^{(N)},\ldots,\mathsf{x}_N^{(N)})$ and applying the identity above iteratively, we see that for all $m=2,\ldots,k$,
    \begin{equation*}
        \frac{\mathrm{p}_m\left(\mathsf{x}_1^{(N)},\mathsf{x}_2^{(N)},\ldots, \mathsf{x}_N^{(N)}\right)}{N^m} \xrightarrow[]{N\to \infty} \mathsf{q}_m(s),
    \end{equation*}
    almost surely and in $L^{n_m}\Big(\mu_\infty^{(s)}\Big)$. In particular, the relation above still holds if $\mathrm{p}_m(\mathsf{x}_1^{(N)},\ldots,\mathsf{x}_N^{(N)})$ is replaced by $\sum_{j=1}^N \left(\mathsf{x}^{(N)}_j-\i\right)^m$. Thus, arguing as in the proof of Proposition \ref{exchangeablejointmom} we see that \eqref{eq:convergenceforpainleve} indeed holds, where the polynomials $\mathcal{A}_m^{(s)}(t_{1})$ are given by,
    \begin{equation*}
        \mathcal{A}_m^{(s)}(t_{1})=\lim_{N\to \infty}\frac{P_{m}^{(s,N)}(t_{1})}{N^{\sum_{q=2}^k (q-1)n_q}}.
    \end{equation*}
    This proves the result, except for the justification for the $t_{1}$-degrees of the polynomials $t_{1}^{m-1} \mathcal{A}_m^{(s)}(t_{1})$ being at most $\sum_{q=2}^k (q-1)n_q-1$ (which is smaller than the maximal degree that is given immediately by taking the limit of Theorem \ref{structureforgeneralfinitesize}). This is simply due to the fact that certain terms vanish while taking the limit. To be more precise, if one writes the expansion in $N$,
    \begin{equation*}
        \mathcal{L}_{N,l,q}^{(n_2,\ldots,n_k)}(t_1) = N^{l+\sum_{m=2}^k (m-1)n_m}  \mathcal{M}_{N,l,q}^{(n_2,\ldots,n_k)}(t_1) + O(N^{l+\sum_{m=2}^k (m-1)n_m-1}),
    \end{equation*}
    for each $1\leq q\leq l$, (and do the same for $ \mathcal{L}_{N}^{(n_2,\ldots,n_k)}$)
    then by comparing the leading order coefficients in $N$ in Propositions \ref{recursionfirstprop} and \ref{generalexpressionformathcalFn2nk}, one obtains analogous relations for $\mathcal{M}_{N,l,q}^{(n_2,\ldots,n_k)}$. Then, carefully running the same inductive process gives that the polynomials in $t_1$ contained in the terms $\mathcal{M}_{N,l,q}^{(n_2,\ldots,n_k)}$ have the desired degrees that are stated in Theorem \ref{painlevethm}. This completes the proof of the theorem.
\end{proof}
\begin{rmk}
    We could have defined $\mathcal{M}_{N,l,q}^{(n_2,\ldots,n_k)}$ earlier, and have run the inductive argument specifically for this term. However, we would also like to have a recursive algorithm to compute the joint moments at the finite-$N$ level, for various reasons explained in the introduction, for which we need information about the lower-order terms as well.
\end{rmk}
As mentioned in the introduction, the procedure described above not only proves the existence of a relation as in Theorem \ref{painlevethm}, but provides a recursive way to compute the polynomials $P_m^{(s,N)}(t_1)$ (and hence $\mathcal{A}_m^{(s)}(t_1)$). 
\begin{proof}[Proof of Proposition \ref{finiteNexplicit} and Corollary \ref{painlevecor}.] 
Both results are mere consequences of the relations
\beas
\frac{1}{N^2}\E_N^{(s)} \left[
e^{-\i \frac{t_{1}}{N} \sum_{j=1}^{N} x_j^{(N)}}\mathrm{p}_2\Big(\mathsf{x}_1^{(N)},\ldots, \mathsf{x}_N^{(N)}\Big)
\right]
=- \frac{2s}{t_{1}} \frac{\mathrm{d}}{\mathrm{d}t_{1}} 
\E_N^{(s)} \left[
e^{-\i \frac{t_{1}}{N} \sum_{j=1}^N x_j^{(N)}} 
\right]
- \frac{1}{N}\E_N^{(s)} \left[
e^{-\i \frac{t_{1}}{N} \sum_{j=1}^{N} x_j^{(N)}} 
\right].
\eeas
\begin{align*}
    \frac{1}{N^4}\mathbb{E}_N^{(s)} \Big[
e^{-\textnormal{i} \frac{t_{1}}{N}\sum_{j=1}^{N} \mathsf{x}_j^{(N)}} \mathrm{p}_2^2\Big(\mathsf{x}_1^{(N)},\ldots, \mathsf{x}_N^{(N)}\Big)
\Big] =
\frac{4s^2+2}{t_{1}^2} \frac{\mathrm{d}^2}{\mathrm{d}t_{1}^2} &
\mathbb{E}_{N}^{(s)}\left[e^{-\textnormal{i} \frac{t_{1}}{N}\sum_{j=1}^{N}\mathsf{x}_{j}^{(N)}}\right]
\\+\left(\frac{4s}{Nt_{1}}-\frac{12s^2}{t_{1}^3}\right)\frac{\mathrm{d}}{\mathrm{d}t_{1}}\mathbb{E}_{N}^{(s)}\left[e^{-\textnormal{i} \frac{t_{1}}{N}\sum_{j=1}^{N}\mathsf{x}_{j}^{(N)}}\right]
&+\left(\frac{1}{N^2}-\frac{2}{t_{1}^2}-\frac{4s}{Nt_{1}^2}\right)\mathbb{E}_{N}^{(s)}\left[e^{-\textnormal{i} \frac{t_{1}}{N}\sum_{j=1}^{N}\mathsf{x}_{j}^{(N)}}\right].
\end{align*}
that hold for $s\in \left(\frac{1}{2},\infty\right)$ and $s\in \left(\frac{3}{2},\infty\right)$ respectively.
Indeed, Corollary \ref{painlevecor} follows immediately after taking the $N\to \infty$ limit of these relations, whereas Proposition \ref{finiteNexplicit} follows by evaluating these relations at $t=0$, and then using the expressions for the first few moments of the statistic $\sum_{j=1}^N \mathsf{x}_j^{(N)}$, which were obtained in \cite[Lemma 2]{Basor_2019}.
\end{proof}

\addcontentsline{toc}{section}{Appendix}
\section*{Appendix}
\label{app:scripts}
As mentioned before, here we give explicit formulae for the matrices that were used in the proof of Theorem \ref{structureforgeneralfinitesize}. Through the rest of this section, we fix $l\in \mathbb{N}$. Then, we define, for $i,j=1,\ldots, l $,
\begin{align*}
   \left(\mathbf{B}^{(l)}\right)_{ij} =\begin{cases}
(-1)^{i+j-1}/j(j+1) & j\geq i; \\
-1/i & j=i-1; \\
0  & j<i-1; \\
(-1)^{i-1}/l & j=l.
\end{cases}
\end{align*}
For $i=1,\ldots,l$, $j=1,\ldots,l-1$, we define
\begin{align*}
 \left(\mathbf{Q}_{0}^{(l)}\right)_{ij}&=
\begin{cases}\label{definitionofC1}
(-1)^{i+j}\frac{j(l-j-2)+1-2s}{j(j+1)}& \text{if } i\leq j\leq l-1;\\
\frac{j(j+2s)-l+1}{j+1}& \text{if } j=i-1; \\
  0 & \text{if } j<i-1.
\end{cases}
\\
\left(\mathbf{Q}_{1}^{(l)}\right)_{ij}&=
\begin{cases}
(-1)^{i+j}\frac{1}{j(j+1)}& \text{if } i\leq j\leq l-1;\\
 -\frac{j}{j+1}& \text{if } j=i-1; \\
  0 & \text{if } j<i-1.
\end{cases}
\\
\left(\widetilde{\mathbf{Q}}_1^{(l)}\right)_{ij} &= \begin{cases}
\frac{(-1)^{i+j}}{j+1} (\frac{2N}{j}-j-2+l) & i\leq j \leq l-1; \\
-\frac{j(2N-j)+l-j-1}{j+1} & j=i-1; \\
0 & 1\leq j \leq i-2.
\end{cases}
\end{align*}
For $i=1,\ldots,l$, $j=1,\ldots,l-2$, we define
\begin{align*}
  \left(\mathbf{Q}_{2}^{(l)}\right)_{ij}  &= \begin{cases}
(-1)^{i+j}\frac{(s-2)N+j(j+3)-(j+2)(l-2)+2(s-1)}{(j+1)(j+2)}  & i-1\leq j \leq l-2; \\
\frac{(j+s)(j-N)}{j+2} & j=i-2; \\
0 & 1\leq j <  i-2.
\end{cases}
\\
\left(\widehat{\mathbf{Q}}_{2}^{(l)}\right)_{ij} &= \begin{cases}
(-1)^{i+j}\frac{N+2}{(j+1)(j+2)}  & i-1\leq j \leq l-2; \\
\frac{j-N}{j+2} & j=i-2; \\
0 & 1\leq j <  i-2.
\end{cases}
\\
\left(\widetilde{\mathbf{Q}}_2^{(l)}\right)_{ij} &= \begin{cases}
(-1)^{i+j}\frac{N^2+3N-j(j+3)+(j+2)(l-2)}{(j+1)(j+2)}  & i-1\leq j \leq l-2; \\
\frac{-N^2+N(2j+1)-j(j+1)}{j+2} & j=i-2; \\
0 & 1\leq j < i-2.
\end{cases}
\end{align*}
For $i=1,\ldots,l$, $j=1,\ldots,l+m-2$, $m\geq 3$, we define
\begin{equation*}
   \left(\mathbf{Q}_{m}^{(l)}\right)_{ij} = \begin{cases}
(-1)^{m-j-1}  & i=1, j=1,\ldots,m-1; \\
 \frac{(-1)^{i+j+m}(2-m)}{(j-m+1)(j-m+2)}  &  j>i+m-2;                \\
\frac{i+m-2}{i} & i\neq 1, j=i+m-2;\\
0 & i\neq 1, j\leq i+m-3.
\end{cases}
\end{equation*}

\bibliographystyle{siam}
\bibliography{main}

\noindent{\sc School of Mathematics, University of Edinburgh, James Clerk Maxwell Building, Peter Guthrie Tait Rd, Edinburgh EH9 3FD, U.K.}\newline
\href{mailto:theo.assiotis@ed.ac.uk}{\small theo.assiotis@ed.ac.uk}

\bigskip
\noindent
{\sc Department of Mathematics, Princeton University, 
Fine Hall, 304 Washington Rd, Princeton, NJ 08544, USA.}\newline
\href{mailto:magunes@princeton.edu}{\small magunes@princeton.edu}

\bigskip
\noindent
{\sc Mathematical Institute, University of Oxford, Andrew Wiles Building, Woodstock Road, Oxford, OX2 6GG, UK. }\newline
\href{keating@maths.ox.ac.uk}{\small keating@maths.ox.ac.uk}

\bigskip
\noindent
{\sc Mathematical Institute, University of Oxford, Andrew Wiles Building, Woodstock Road, Oxford, OX2 6GG, UK. }\newline
\href{weif0831@gmail.com}{\small weif0831@gmail.com}

\end{document}